\documentclass[11pt,a4paper,english]{article}

\usepackage{amsmath}
\usepackage{amsthm}
\usepackage{amssymb}
\usepackage{amsopn}
\usepackage{multicol}     
\usepackage{tikz-cd}
\usepackage{pgfplots}
\usepackage{wrapfig}
\usepackage{mathrsfs}
\usepackage[all]{xy}
\usepackage{mathtools,xparse}
\usepackage{theoremref}
\usepackage{combelow}
\usepackage{pifont}
\usepackage{pdfpages}
\usepackage{dsfont}
\usepackage{textgreek}
\usepackage{stackrel,amssymb}
\usepackage{enumerate}
\usepackage[colorlinks=false]{hyperref}

\usepackage{tabularx}
\usepackage{array}
\usepackage{mathtools}
\usepackage{etoolbox}
\usepackage{url}

\usepackage[normalem]{ulem}

\DeclareFontFamily{U}{matha}{\hyphenchar\font45}
\DeclareFontShape{U}{matha}{m}{n}{
<-6> matha5 <6-7> matha6 <7-8> matha7
<8-9> matha8 <9-10> matha9
<10-12> matha10 <12-> matha12
}{}
\DeclareSymbolFont{matha}{U}{matha}{m}{n}

\DeclareFontFamily{U}{mathx}{\hyphenchar\font45}
\DeclareFontShape{U}{mathx}{m}{n}{
<-6> mathx5 <6-7> mathx6 <7-8> mathx7
<8-9> mathx8 <9-10> mathx9
<10-12> mathx10 <12-> mathx12
}{}
\DeclareSymbolFont{mathx}{U}{mathx}{m}{n}

\DeclareMathDelimiter{\vvvert} {0}{matha}{"7E}{mathx}{"17}%

\newcommand{\diffto}{\xrightarrow{\raisebox{-0.2 em}[0pt][0pt]{\smash{\ensuremath{\sim}}}}}

\newcommand{\out}{\mathfrak{out}(\spl,\pi)}

\DeclareMathOperator{\N}{\mathbb{N}}
\DeclareMathOperator{\Z}{\mathbb{Z}}
\DeclareMathOperator{\R}{\mathbb{R}} 
\DeclareMathOperator{\C}{\mathbb{C}}

\DeclareMathOperator{\Cas}{Cas}
\DeclareMathOperator{\spl}{\mathfrak{sl}_2(\C)}

\DeclareMathOperator{\so}{\mathfrak{so}(3,1)}

\DeclareMathOperator{\su}{\mathfrak{su}_2}

\DeclareMathOperator{\herm}{Herm_2^+}

\DeclareMathOperator{\im}{Im}
\DeclareMathOperator{\g}{\mathfrak{g}}
\DeclareMathOperator{\kk}{\mathfrak{k}}

\DeclareMathOperator{\F}{\mathcal{F}}

\DeclareMathOperator{\Sk}{\mathfrak{S}}

\DeclareMathOperator{\id}{Id}

\DeclareMathOperator{\sign}{sign}

\DeclareMathOperator{\Ad}{Ad}

\DeclareMathOperator{\tr}{tr}

\DeclareMathOperator{\dif}{\mathrm{d}}
\DeclareMathOperator{\CC}{\mathscr{C}}
\DeclareMathOperator{\DD}{\mathscr{D}}
\DeclareMathOperator{\Lie}{\mathcal{L}}

\DeclareMathOperator{\nf}{\lvert \mathit{f}\rvert }

\newtheorem{theorem}{Theorem}

\newtheorem{proposition}{Proposition}[section]
\newtheorem{definition}[proposition]{Definition}

\newtheorem*{definition*}{Definition}

\newtheorem{lemma}[proposition]{Lemma}
\newtheorem{corollary}[proposition]{Corollary}
\newtheorem{remark}[proposition]{Remark}
\newtheorem*{remark*}{Remark}

\newtheorem{example}[proposition]{Example}
\newtheorem*{conjecture}{Conjecture}

\DeclareDocumentCommand{\norm}{ s m }{%
	\IfBooleanTF{#1}
	{#2}
	{\lVert{#2}\rVert}%
}
\DeclareDocumentCommand{\fnorm}{ s m }{%
	\IfBooleanTF{#1}
	{#2}
	{[]{#2}[]}%
}

\newcommand{\Addresses}{{
  \bigskip
  \footnotesize

Ioan M\u{a}rcu\cb{t},\par\nopagebreak
\textsc{Radboud University Nijmegen, 6500 GL Nijmegen, The Netherlands}\par\nopagebreak
  \textit{E-mail address}: \texttt{i.marcut@math.ru.nl}

\medskip

 Florian Zeiser,\par\nopagebreak
\textsc{University of Illinois at Urbana-Champaign, 61801 Urbana, United States}\par\nopagebreak
\textit{E-mail address}: \texttt{fzeiser@illinois.edu}
}}

\title{The Poisson linearization problem for $\mathfrak{sl}_2(\C)$\\ [1ex] \Large  Part I: Poisson cohomology}
\author{Ioan M\u{a}rcu\cb{t} and Florian Zeiser}

\pgfplotsset{compat=1.16}

\begin{document}
\maketitle

\begin{abstract}
This is the first of two papers, in which we prove a version of Conn's linearization theorem for the Lie algebra $\mathfrak{sl}_2(\C)\simeq \so$. Namely, we show that any Poisson structure whose linear approximation at a zero is isomorphic to the Poisson structure associated to $\mathfrak{sl}_2(\C)$ is linearizable. 

In this first part, we calculate the Poisson cohomology associated to $\mathfrak{sl}_2(\C)$, and we construct bounded homotopy operators for the Poisson complex of multivector fields that are flat at the origin.

In the second part, we will obtain the linearization result, which works for a more general class of Lie algebras. For the proof, we will develop a Nash-Moser method for functions that are flat at a point.
\end{abstract}

\section*{Introduction to Part I}

The space of $C^{\infty}$-multivector fields on a smooth manifold $M$:
 \begin{align*}
  \mathfrak{X}^{\bullet}(M)&:= \Gamma (\wedge ^{\bullet} TM)
 \end{align*}
carries a natural extension of the Lie bracket, called the Schouten-Nijenhuis bracket (see e.g.\ \cite{Laurent2013}), which makes $\mathfrak{X}^{\bullet}(M)$ into a graded Lie algebra:  
 \begin{equation*}
  \begin{array}{cccc}
   [\cdot,\cdot]:&\mathfrak{X}^{p+1}(M)\times \mathfrak{X}^{q+1}(M)&\to &\mathfrak{X}^{p+q+1}(M).
  \end{array}
 \end{equation*}
A Poisson structure on $M$ is a bivector field $\pi \in \mathfrak{X}^2(M)$ satisfying
\[[\pi,\pi]=0,\]
and the pair $(M,\pi)$ is called a Poisson manifold. These structures originate in Hamiltonian mechanics and were introduced in an abstract, geometric setting by Lichnerowicz \cite{Lich77}. 

Weinstein initiated the local study of Poisson manifolds in \cite{Wein83}, and the results and questions therein were the basis of a lot of research in Poisson geometry, including our paper. The Splitting Theorem \cite{Wein83} states that every Poisson structure is locally isomorphic to the product of a symplectic structure and a Poisson structure that vanishes at the point. Using Darboux coordinates on the symplectic part, it suffices to understand the behavior of Poisson structures around zeros. This led Weinstein to formulate:

\vspace*{0.2cm}

\noindent{\underline{\textbf{The question of linearization}}}\vspace*{0.2cm}

Let $x\in M$ be a zero of $\pi$, i.e., $\pi_{x}=0$. Then $\pi$ induces a Lie algebra structure on $\g_{x}:= T^*_{x}M$, called the \textbf{isotropy Lie algebra} at $x$, with bracket
\[ [\dif_{x}g,\dif_{x}h ]:=\dif_{x}\pi(\dif g,\dif h),\quad \quad g,h\in C^{\infty}(M).\] 

Any Lie algebra $(\g ,[\cdot,\cdot])$ has a corresponding linear Poisson structure $\pi_{\g}$ on the dual vector space $\mathfrak{g}^*$, defined by:
\begin{align}\label{eq: linear Poisson}
      \pi_{\g,\xi}(X,Y):= \xi([X,Y]) \quad \text{where} \ \xi \in \mathfrak{g}^*, \ X,Y\in \g. 
\end{align}

Applying this construction to the isotropy Lie algebra $\g_x$ of $\pi$ at a zero $x\in M$, one obtains the linear Poisson structure $\pi_{\g_{x}}$ on $\mathfrak{g}^*_{x}=T_{x}M$, which plays the role of the first order approximation of $\pi$ at $x$.

\begin{definition*}
    A Poisson structure $\pi$ is called \textbf{linearizable} around a zero $x\in M$ if there exists a Poisson diffeomorphism
\[ \Phi: (U,\pi)\diffto (V,\pi_{\g_{x}}),\quad \textrm{with}\quad \Phi(x)=0,\]
where $x\in U\subset M$ and $0\in V\subset  \mathfrak{g}^*_{x}$ are open sets.
\end{definition*}

Sometimes, being linearizable depends only on the isotropy Lie algebra.
\begin{definition*}
A Lie algebra $\g$ is called \textbf{Poisson non-degenerate} if any Poisson structure $\pi$ is linearizable around any zero $x$ satisfying $\g_x\simeq \g$. Otherwise, $\mathfrak{g}$ is called \textbf{Poisson degenerate}.
\end{definition*} 

Also in \cite{Wein83}, Weinstein showed that every semisimple Lie algebra is non-degenerate in the category of formal power series (i.e., the power series of the Poisson diffeomorphism $\Phi$ exists) and he conjectured that the same was true in the analytic category. This was proven shortly thereafter by Conn in \cite{Conn84}. Weinstein also showed that $\mathfrak{sl}_2(\R)$ is Poisson degenerate in the smooth category. 

In order to give an overview of the current state of the linearization problem for semisimple Lie algebras $\g$ (in the smooth category), recall that the \textbf{real rank} of $\g$ is the dimension of $\mathfrak{a}$ in its Iwasawa decomposition
\[ \g =\kk\ \oplus\ \mathfrak{a}\ \oplus\ \mathfrak{n}\]
(see e.g.\ \cite{Knapp13}). The following results are known:

\begin{enumerate}[leftmargin=!]
\item[$\mathrm{dim}(\mathfrak{a})=0$\phantom{12}] I.e., if $\g $ is semisimple of compact type, then $\g$ is Poisson non-degenerate. This was proven by Conn in \cite{Conn85} using the Nash-Moser technique. More recently, Crainic and Fernandes provided a geometric proof in \cite{CF11}. For the lowest dimensional case, i.e.\ $\g \simeq \mathfrak{so}(3)$, an earlier proof was obtained by Dazord in \cite{Daz}, and in fact, using the language of completely integrable 1-forms, a proof can be traced back to Reeb's thesis \cite{Reeb}.
\item[$\mathrm{dim}(\mathfrak{a})\geq 2$\phantom{12}] In this case, it was shown by Weinstein in \cite{Wein87} that $\g$ is Poisson degenerate.
\item[$\mathrm{dim}(\mathfrak{a})=1$\phantom{12}] If the compact part $\kk$ is not semisimple then $\g$ is also Poisson degenerate as shown in \cite[Theorem 4.3.2]{DZ} by Monnier and Zung. The example of $\mathfrak{sl}_2(\R)$ given by Weinstein is the lowest dimensional Lie algebra of this type.
\end{enumerate}

The remaining cases for semisimple Lie algebras are still open:
\begin{conjecture}[Dufour $\&$ Zung]
If $\g$ is semisimple of real rank one, i.e.\ $\mathrm{dim}(\mathfrak{a})=1$, and $\kk$ is semisimple, then $\g$ is Poisson non-degenerate.
\end{conjecture} 
The lowest dimensional Lie algebra of real rank one with semisimple compact part is the real Lie algebra $\spl$, with Iwasawa decomposition
\begin{equation*}
    \begin{array}{rcl}
         \spl& = &  \su \oplus \ \R \cdot \begin{pmatrix}
         1&0\\
         0&-1
         \end{pmatrix}
         \oplus \C \cdot
         \begin{pmatrix}
         0&1\\
         0&0
         \end{pmatrix}
    \end{array}
\end{equation*}

We prove the conjecture for this case:
\begin{theorem}\label{theorem: nondegenerate}
The real Lie algebra $\spl$ is Poisson non-degenerate.
\end{theorem}

The two parts of the proof correspond to the two parts of this paper:

\begin{enumerate}
\item [I.] $2^{\textrm{nd}}$ Poisson cohomology\ $=0$ $\quad +\quad $   ``nice'' cochain homotopies.
\item[II.] The Nash-Moser method.
\end{enumerate}

The same scheme was applied by Conn in \cite{Conn85} for compact semisimple Lie algebras. However, the lack of compactness makes both parts of the problem much more involved in our case. In particular, the homotopy operators we obtain in the first part are defined only on multivector fields that are flat at the origin, i.e., they vanish to infinite order, and the estimates they satisfy are not the classical tameness estimates used in the Nash-Moser method. To deal with this, in the second part, we develop Nash-Moser techniques for flat functions, which should be useful in many other geometric problems. Also, the linearization result from the second part can be applied to any Lie algebra for which one can solve the first part.

\vspace*{0.2cm}

\noindent{\underline{\textbf{Poisson cohomology}}}\vspace*{0.2cm}

Lichnerowicz realized in \cite{Lich77} that any Poisson structure $\pi$ on a manifold $M$ induces a differential $\dif_{\pi}$ on the complex of multivector fields:
\[ \dif_{\pi}:= [\pi,\cdot]: \mathfrak{X}^{\bullet}(M) \to \mathfrak{X}^{\bullet+1}(M) \ \ \ \text{ with } \ \ \ \dif_{\pi}^2=0.\]
The cohomology groups, denoted by $H^{\bullet}(M,\pi)$, are  called the \textbf{Poisson cohomology} of $(M,\pi)$. In low degrees the cohomology groups have geometric interpretations. Of special interest to us is the second Poisson cohomology group $H^2(M,\pi)$ which governs infinitesimal deformations modulo trivial deformations. Despite its importance in questions such as the linearization or deformation problems, the calculation of Poisson cohomology is by no means standard at this point, due to a lack of general methods. 

Conn obtained the vanishing of the second Poisson cohomology group using methods from representation theory for compact Lie groups. A more effective way to prove this is by using the symplectic Lie groupoid and regarding Poisson cohomology as invariant foliated cohomology on the fibers of the groupoid -- a method which is now standard in the literature \cite{Wein,GW92,Xu92,Marcut,PMCT}. However, none of these methods apply for Poisson structures associated to non-compact Lie algebras. In \cite{MZ}, we have calculated the Poisson cohomology of the Lie algebra $\mathfrak{sl}_2(\R)$ by first dividing it into the \emph{formal} part and the \emph{flat} part. The calculation of formal cohomology is an algebraic problem, which is solved using classical representation theory. The computation of flat cohomology is an analytic problem, for which we developed methods for the calculation of foliated cohomology and a spectral sequence for regular Lie algebroids, and adapted these to the singular setting using flat functions. In this paper we develop these methods further and adapt them to calculate the Poisson cohomology groups of $\mathfrak{sl}_2(\C)$, which are detailed in Theorem \ref{theorem: PC}. In particular, we obtain that
\[H^2(\mathfrak{sl}_2(\C)^*,\pi_{\spl})=0.\]
In order to apply the Nash-Moser method, we prove Theorem \ref{main theorem}, which is a quantitative version of Theorem \ref{theorem: PC} for flat Poisson cohomology. In particular, Theorem \ref{main theorem} gives cochain homotopies in degree two for the subcomplex of multivector fields that are flat at the origin. These operators restrict to closed balls centered at the origin and satisfy certain estimates (as in Definition \ref{definition: property H}). These properties will be used in Part II of the paper to apply the Nash-Moser method.

\begin{remark*}\rm
For the state of the linearization problem in low dimension see \cite{DZ}, for a discussion of the linearization question in the wider context see \cite{FM04}, and for the historical context see \cite[Appendix 2]{CF11}.
\end{remark*}

\subsection*{Acknowledgments} We would like to thank Marius Crainic, Camille Laurent-Gengoux, Rui Loja Fernandes, Philippe Monnier and Alan Weinstein for their feedback and encouragements, and for serving on the defense committee of the second author's PhD thesis, which contains an early version of this paper. 

We would like to thank also Pietro Baldi, whose advice made us consider atypical norms of spaces of flat functions. This interaction happened during the winter school \emph{Implicit Function Theorems in Geometry and Dynamics}, in February 2020 -- our gratitude goes to the organizers as well. 

Finally, the second author thanks Instituto de
Matem\'atica Pura e Aplicada (IMPA) for its hospitality during the final stages of the project.
\tableofcontents

\section{The Poisson cohomology of $\mathfrak{sl}_2(\C)$}\label{section: main result}

In Subsection \ref{section: linear Poisson} we recall the interpretation of Poisson cohomology of a linear Poisson structure in terms of Lie algebra cohomology. We will see that for a semisimple Lie algebra $\g$ the Poisson cohomology of its associated linear Poisson structure can be described by its flat and formal Poisson cohomology. In Subsection \ref{section: setting} we fix some notation for the Lie algebra $\spl$, and then describe the symplectic foliation and the algebra of Casimirs. In Subsection \ref{section: formal poisson} we compute the formal Poisson cohomology groups of the linear Poisson structure associated with the real Lie algebra $\spl$. Finally, in Subsection \ref{section: PC statement}, we give the description of the Poisson cohomology groups of $\spl$, which will be proven in the following three sections.


\subsection{Poisson cohomology of linear Poisson structures}\label{section: linear Poisson}

A Poisson structure on a vector space is called linear if the set of linear functions is closed under the Poisson bracket. Such Poisson structures are in one-to-one correspondence with Lie algebra structures on the dual vector space via \eqref{eq: linear Poisson}. Explicitly, for a Lie algebra $(\g,[\cdot,\cdot])$, the Poisson structure $\pi$ on $\mathfrak{g}^*$ is determined by the condition that the map $l:\mathfrak{g}\to C^{\infty}(\mathfrak{g}^*)$, which identifies $\g$ with $(\mathfrak{g}^*)^*$, 
is a Lie algebra homomorphism:
\[\{l_X,l_Y\}_{\pi}=l_{[X,Y]},\]
where $\{\cdot,\cdot\}_{\pi}$ is the Poisson bracket on $C^{\infty}(\mathfrak{g}^*)$ corresponding to $\pi$. In particular, $C^{\infty}(\mathfrak{g}^*)$ becomes a $\mathfrak{g}$-representation, with $X\cdot h:=\{l_X,h\}$. Moreover, the Poisson complex of $(\mathfrak{g}^*,\pi)$ is isomorphic to the Chevalley-Eilenberg complex of $\g$ with coefficients in $C^{\infty}(\mathfrak{g}^*)$ \cite[Prop 7.14]{Laurent2013}
\begin{equation}\label{iso_complexes}
(\mathfrak{X}^{\bullet}(\mathfrak{g}^*),\dif_{\pi})\simeq (\wedge^{\bullet}\mathfrak{g}^*\otimes C^{\infty}(\mathfrak{g}^*),\dif_{EC}).
\end{equation}
This identification allows for the use of techniques from Lie theory in the calculation of Poisson cohomology. 

Under the isomorphism \eqref{iso_complexes}, the subcomplex of multivector fields on $\mathfrak{g}^*$ that are flat at $0$ corresponds to the Eilenberg-Chevalley complex of $\mathfrak{g}$ with coefficients in the subrepresentation $C_0^{\infty}(\mathfrak{g}^*)\subset C^{\infty}(\mathfrak{g}^*)$ consisting of smooth functions that are flat at zero:
\begin{align}\label{eq: multivector fields identification}
    \mathfrak{X}^{\bullet}_0(\mathfrak{g}^*)\simeq \wedge^{\bullet}\mathfrak{g}^*\otimes C^{\infty}_0(\mathfrak{g}^*).
\end{align}
We call the cohomology of this complex the \textbf{flat Poisson cohomology} and denote it by $H^{\bullet}_0(\mathfrak{g}^*,\pi)$.
The quotient complex is naturally identified with 
\[(\wedge^{\bullet} \mathfrak{g}^*\otimes \R[[\mathfrak{g}]],\dif_{EC}),\]
where $\R[[\mathfrak{g}]]:=C^{\infty}(\mathfrak{g}^*)/C^{\infty}_0(\mathfrak{g}^*)$ is the ring of formal power series of functions on $\mathfrak{g}^*$. The resulting cohomology will be called the \textbf{formal Poisson cohomology at} $0\in \mathfrak{g}^*$ denoted by $H^{\bullet}_{F}(\mathfrak{g}^* ,\pi )$. By the above, this is isomorphic to the cohomology of $\mathfrak{g}$ with coefficients in $\R[[\mathfrak{g}]]$, i.e\ 
\[H^{\bullet}_{F}(\mathfrak{g}^*,\pi)\simeq H^{\bullet}(\mathfrak{g},\R[[\mathfrak{g}]]).\]
Moreover, the short exact sequence
 \begin{align*}
  0\to C^{\infty}_0(\mathfrak{g}^*)\to C^{\infty}(\mathfrak{g}^*)\stackrel{j^{\infty}_0}{\longrightarrow} \R[[\g]]\to 0,
 \end{align*}
 induces a long exact sequence in cohomology: 
\begin{equation}\label{jet}
\ldots \stackrel{j^{\infty}_0}{\to} H^{\bullet-1}_{F}(\mathfrak{g}^*,\pi)\stackrel{\partial}{\to} 
H^{\bullet}_{0}(\mathfrak{g}^*,\pi)\to H^{\bullet}(\mathfrak{g}^*,\pi)\stackrel{j^{\infty}_0}{\to}H^{\bullet}_{F}(\mathfrak{g}^*,\pi )\stackrel{\partial}{\to}\ldots
 \end{equation}

If the Lie algebra $\g$ is semisimple, we can say more:
\begin{proposition}\cite[Corollary 4.5]{MZ}\label{ses smooth formal}
For a semisimple Lie algebra $\g$, the Poisson cohomology of $(\mathfrak{g}^*,\pi)$ fits into the short exact sequence 
\[0\to H^{\bullet}_0(\mathfrak{g}^*,\pi)\to H^{\bullet}(\mathfrak{g}^*,\pi)\stackrel{j^{\infty}_0}{\to} H^{\bullet}_F(\mathfrak{g}^*,\pi)\to 0.\]
\end{proposition}

Elements in the $0$-th Poisson cohomology are called \emph{Casimir functions} 
\[\Cas(\mathfrak{g}^*,\pi):=H^0(\mathfrak{g}^*,\pi).\]
These are the same as functions on $\mathfrak{g}^*$ that are invariant under the coadjoint action. The flat and the formal Casimir functions will be denoted:
\[\Cas_0(\mathfrak{g}^*,\pi) \quad \textrm{and}\quad \Cas_F(\mathfrak{g}^*,\pi).\]

The formal Poisson cohomology of semisimple Lie algebras is well-known:   
\begin{proposition}\label{formal cohomology}
\cite[Proposition 4.2 \& 4.3]{MZ} Let $(\mathfrak{g}^*,\pi)$ be the linear Poisson structure associated to a semisimple Lie algebra $\mathfrak{g}$.
\begin{enumerate}
    \item There exist $n=\dim \mathfrak{g} - \max \mathrm{rank}(\pi)$ algebraically independent homogeneous polynomials $f_1 ,\dots , f_n$ such that
    \begin{align*}
        \Cas_F(\mathfrak{g}^*,\pi)=\R[[f_1 ,\dots ,f_n ]] \subset \R[[\mathfrak{g}]].
    \end{align*}
    Here $\R[[f_1,\dots ,f_n]]$ denotes formal power series in the polynomials $f_j$.
    \item For the formal Poisson cohomology of $\g$ at $0$ we have that: 
    \begin{align*}
        H^{\bullet}_{F}(\mathfrak{g}^*,\pi)&\simeq H^{\bullet}(\g)\otimes \Cas _{F}(\mathfrak{g}^*,\pi).
    \end{align*}
\end{enumerate}
\end{proposition}

In the next subsections we will use these results for $\spl$.

\subsection{The foliation and the Casimirs}\label{section: setting}
In this subsection we describe the linear Poisson structure associated with the real Lie algebra $\spl$ in terms of its complex counterpart, and we discuss its symplectic foliation and the algebra of Casimir functions.


We consider the following complex coordinates on $\spl$:
\begin{equation}\label{eq: identification}
     \C^3\simeq \mathfrak{sl}_2(\C)\quad
     (z_1,z_2,z_3)\mapsto A:=
     \begin{pmatrix}
     iz_1 &-z_2+iz_3 \\
     z_2+iz_3&-iz_1
     \end{pmatrix}
 \end{equation}
Throughout the paper we identify $\spl\simeq \spl^*$ via the trace form:
\[\spl\diffto \spl^*, \quad A\mapsto \big(B\mapsto \mathrm{tr}(A\cdot B)\big).\]
We use this identification to regard the linear Poisson structure associated with the \emph{complex} Lie algebra $\spl$ as a complex Poisson structure on $\spl$. In the coordinates from above, this becomes: 
\begin{align}\label{eq:complex basis}
  \pi_{\C} &:=z_1\partial_{z_2} \wedge \partial_{z_3} +z_2 \partial_{z_3} \wedge \partial_{z_1} +z_3\partial_{z_1}\wedge \partial_{z_2}
 \end{align}
The real and the imaginary part of $\pi_{\C}$ yield two Poisson structures 
 \begin{align}\label{eq: real identification}
      \pi=\pi_{1}:=4\cdot \Re (\pi_{\C}), \ \ \ \ \pi_{2} =4\cdot \Im(\pi_{\C})
\end{align}
The constant $4$ ensures that the linear Poisson structure associated with the \emph{real} Lie algebra $\spl$ is:
\[(\spl,\pi).\]

We denote the corresponding \emph{singular symplectic foliation} by:
\[(\mathcal{F},\omega_{1}).\]
From general properties of complex Poisson structures (see \cite{LGStXu}), it follows that the leaves of $\mathcal{\F}$ coincide with the adjoint orbits on $\mathfrak{sl}_2(\C)$, and agree also with the leaves of $\pi_2$ and of $\pi_{\C}$. The leaves can be described using the level sets of the basic, complex-valued Casimir function:
\begin{align}\label{eq:casimir}
 &f=f_1+ i\, f_2 :\spl\to \C,\nonumber \\
 &f(A):= \det (A)=z_1^2+z_2^2+z_3^2.
\end{align}
More precisely, the leaves are the following families of submanifolds:
 \[S_{z}:=f^{-1}(z), \ \ z\in \C \setminus\{ 0\},\]
and the singular cone $S_0:=f^{-1}(0)$ decomposes into two leaves:
\[\{0\} \ \ \ \ \ \textrm{ and }\ \ \ \ \  S^{\text{reg}}_{0}:=f^{-1}(0)\setminus \{0\}. \]

The map $f$ generates all Casimir functions on $\spl$:

\begin{proposition}\label{prop:Casimirs}
The algebra of Casimir functions of $(\spl,\pi)$ is isomorphic to the algebra of smooth functions on $\C$, via the map: 
\begin{equation*}
f^*:C^{\infty}(\C)\diffto \Cas(\mathfrak{sl}_2(\C),\pi),\quad g\mapsto g\circ f.
\end{equation*}
\end{proposition}
\begin{proof}
Clearly, the assignment is well-defined and injective. Surjectivity follows because $f$ restricts to a surjective submersion on the open dense set $\mathfrak{sl}_2(\C)\setminus \{0\}$, whose fibers are precisely the leaves. 
\end{proof}

\begin{remark}
The \textbf{leaf-space} of $\F$, denoted $\spl/\mathcal{F}$, is obtained by identifying points in $\spl$ that belong to the same leaf. The \textbf{regular leaf-space} $(\spl\backslash\{0\})/\mathcal{F}$ has a smooth structure such that the quotient map is a submersion. In fact, $f|_{\mathfrak{sl}_2(\C)\setminus \{0\}}$ descends to a diffeomorphism: 
\[ (\spl\backslash\{0\})/\mathcal{F}\simeq \C.\] 
Therefore, Proposition \ref{prop:Casimirs} implies that the restriction map
\[\Cas(\mathfrak{sl}_2(\C),\pi)\to C^{\infty}\big((\spl\backslash\{0\})/\F\big)\]
is an isomorphism from the algebra of Casimirs to the algebra of smooth functions on the regular leaf-space.
 \end{remark}


 \subsection{Formal Poisson cohomology of $\mathfrak{sl}_2(\C)$}\label{section: formal poisson}

We compute the formal Poisson cohomology of $(\mathfrak{sl}_2(\C),\pi)$. For this, we introduce the real coordinates $\{x_j,y_j\}_{j=1}^3$ on 
$\mathfrak{sl}_2(\C)$:
 \[ z_j:= x_j+i\cdot y_{j}, \quad \text{ for } j\in \{1,2,3\}.\]
By Proposition \ref{formal cohomology}, the invariant polynomials on $\mathfrak{sl}_2(\C)$ are generated by the functions $f_{1}$ and $f_{2}$ from $f=f_1+i\cdot f_2$ in \eqref{eq:casimir}. Therefore, in order to determine the formal Poisson cohomology of $(\mathfrak{sl}_2(\C),\pi)$, we need to calculate the Lie algebra cohomology of $\spl$. We obtain the following result.

\begin{proposition}\label{formal cohomology sl2} 
The formal Poisson cohomology groups of $(\mathfrak{sl}_2(\C),\pi)$ at the origin are given by
\[H^k_F(\mathfrak{sl}_2(\C),\pi)=
\begin{cases}
\R[[f_1,f_2]]& \text{ if }k=0\\
\R[[f_1,f_2]]\cdot C_{\mathcal{R}} \oplus \R[[f_1,f_2]]\cdot C_{\mathcal{I}}& \text{ if } k=3\\
\R[[f_1,f_2]]\cdot C_{\mathcal{R}}\wedge C_{\mathcal{I}}& \text{ if } k=6\\
0& \text{ else}
\end{cases}\]
where $C_{\mathcal{R}}$ and $C_{\mathcal{I}}$ are $\frac{4}{3}$ of the real and imaginary part of the Cartan three-form, respectively. Explicitly:
\begin{align*}
    C_{\mathcal{R}}=&\ \frac{1}{2}( \partial_{x_1}\partial_{x_2}\partial_{x_3}-\partial_{y_1}\partial_{y_2}\partial_{x_3}-\partial_{x_1}\partial_{y_2}\partial_{y_3}-\partial_{y_1}\partial_{x_2}\partial_{y_3})\\
    C_{\mathcal{I}}=&\ \frac{1}{2} ( \partial_{y_1}\partial_{y_2}\partial_{y_3}-\partial_{y_1}\partial_{x_2}\partial_{x_3}-\partial_{x_1}\partial_{y_2}\partial_{x_3}-\partial_{x_1}\partial_{x_2}\partial_{y_3})
\end{align*}
 \end{proposition}
 \begin{proof}
 By Propositions \ref{formal cohomology} we have 
\[H^{\bullet}_F(\mathfrak{sl}_2(\C),\pi)= H^{\bullet}(\spl)\otimes \R[[f_1,f_2]].\]
From standard Lie theory we know that $H^{\bullet}(\mathfrak{sl}_2(\C))$ is isomorphic to a free exterior algebra with two generators in degree three (see \cite{Chev50} or also \cite[Corollary 3.7]{Sol02}). 
The complex-valued Cartan three-form is given by
\begin{align*}
    3\, \partial_{z_1}\partial_{z_2}\partial_{z_3}
\end{align*}
Hence taking $\frac{4}{3}$ of the real and imaginary part yields $C_{\mathcal{R}}$ and $C_{\mathcal{I}}$, respectively. Note that $C_{\mathcal{R}},C_{\mathcal{I}}\in(\wedge^{3}\spl)^{\spl}$ and from \cite[Theorem 19.1]{Chev1948} we know that every element in $H^{\bullet}(\spl)$ has a unique $\spl$-invariant representative. 
As a result we obtain for the Lie algebra cohomology $H^{\bullet}(\spl)$:
\begin{align*}
         H^0(\spl) &=\R,\\ 
         H^{3}(\spl)&=\R\cdot 
         C_{\mathcal{R}} \oplus \R\cdot C_{\mathcal{I}}, \\
         H^{6}(\spl)&=\R\cdot C_{\mathcal{R}}\wedge C_{\mathcal{I}}
\end{align*}
and $H^{k}(\spl)=0$ in all other degrees.
\end{proof}

\subsection{Poisson cohomology of $\spl$}\label{section: PC statement}

Recall that Poisson cohomology is a module over the space of Casimir functions, which, for $\spl$, is isomorphic to $C^{\infty}(\C)$ (see Proposition \ref{prop:Casimirs}). We will describe the Poisson cohomology groups as modules over this algebra, and we will also describe the induced Schouten Nijenhuis bracket.

The following two vector fields on $\C\setminus \{0\}$ will play an important role:
\begin{align*}
        Y_1:=-y\partial_x+ (|z|+x)\partial_y  \quad \text{ and }\quad     Y_2:= (|z|-x)\partial_x -y\partial_y,
\end{align*}
where we used the coordinate $z=x+i\cdot y$ on $\C$. 
These satisfy, the relations
 \begin{align*}
    [Y_1,Y_2]=Y_2 \qquad \textrm{and} \qquad 
    (|z|-x)\cdot Y_1 = -y\cdot Y_2.
\end{align*}

The mild singularities of these vector fields imply that, when multiplied with flat functions, they became smooth, flat vector fields on $\C$. Motivated by the theorem below, we denote the resulting $C^{\infty}(\C)$-module by:
\begin{equation}\label{eq:flat:module:of:vector:fields}
\out:=\{g_1Y_1+g_2Y_2\, |\, g_1,g_2\in C^{\infty}_0(\C)\} \subset \mathfrak{X}^1_0(\C).
\end{equation}
The bracket relation above shows that $\out$ is a Lie subalgebra of the algebra of flat vector fields. The second relation above shows that we have an isomorphism of $C^{\infty}(\C)$-modules: 
\[
\out\simeq 
\frac{C_0^{\infty}(\C)\oplus C^{\infty}_0(\C)}{\big\{(g_1,g_2)\ |\  yg_1=(|z|-x) g_2\big\}}.
\]
This module will appear more frequently in the paper (see Subsection \ref{subsection: twisted module}).

We are ready now to describe the Poisson cohomology.

\begin{theorem}\label{theorem: PC} The Poisson cohomology groups of $(\spl,\pi)$ are
\begin{itemize}
    \item $H^0(\spl,\pi)\simeq C^{\infty}(\C)$, via the isomorphism:
    \[g\circ f\in H^0(\spl,\pi) \qquad \leftrightarrow \qquad   g\in C^{\infty}(\C);\]
    \item 
$H^1(\spl,\pi)\simeq \out$, as $C^{\infty}(\C)$-modules and as Lie algebras, via the natural action of $H^1(\spl,\pi)$ on Casi\-mirs; i.e.,
\[[X]\in H^1(\spl,\pi)   \qquad \leftrightarrow \qquad Y\in \out\]
if and only if, for all $g\in C^{\infty}(\C)$:
\[\Lie_X(g\circ f)=\Lie_Y(g)\circ f;\]
\item $H^2(\spl,\pi)=0$;
    \item 
$H^3(\spl,\pi)\simeq 
C^{\infty}(\C)\oplus C^{\infty}(\C)/C^{\infty}_0(\C)
$ as $C^{\infty}(\C)$-modules, where to $(g_1,[g_2])\in C^{\infty}(\C)\oplus C^{\infty}(\C)/C^{\infty}_0(\C)$ we associate: 
    \[[g_1\circ f\cdot C_{\mathcal{R}}+g_2\circ f \cdot C_{\mathcal{I}}] \in H^3(\spl,\pi);\]
\item $H^4(\spl,\pi)\simeq H^1(\spl,\pi)$, via the isomorphism: 
\[[X]\in H^1(\spl,\pi)\quad \leftrightarrow \quad [X\wedge C_\mathcal{R}]\in H^4(\spl,\pi);\]
\item $H^5(\spl,\pi)=0$;
\item $H^6(\spl,\pi)\simeq C^{\infty}(\C)/C^{\infty}_0(\C)$, via the isomorphism: 
\[[g\circ f\cdot C_{\mathcal{R}}\wedge C_{\mathcal{I}}]\in H^6(\spl,\pi) \qquad \leftrightarrow \qquad [g]\in C^{\infty}(\C)/C^{\infty}_0(\C).\]
\end{itemize}
\end{theorem}
Note that the wedge product on $H^{\bullet}(\spl,\pi)$ is determined from the descriptions in the theorem. The Schouten bracket is determined by the isomorphism of Lie algebras from $H^1(\spl,\pi)\simeq \out$, the usual Leibniz rule, and the following relations:
\begin{proposition}\label{Proposition:action:H3}
For any $[X]\in H^1(\spl,\pi)$, we have:
\begin{align*}
[[X,C_{\mathcal{R}}]]&=-\Lie_X(\log|f|)\cdot [C_{\mathcal{R}}]\in H^3(\spl,\pi),\\
[[X,C_{\mathcal{I}}]]&=0\in H^3(\spl,\pi).
 \end{align*}
\end{proposition}
The proofs of these results will be completed in Subsection \ref{section: algebra}.
\section{Flat foliated cohomology of $\spl$}\label{section flat foli}


In this section we define the flat foliated cohomology of the singular foliation $\F$ on $\spl$. For the calculation of this cohomology, we first apply an averaging argument for the compact subgroup $SU(2)\subset SL(2,\C)$. The more involved part of the calculation uses a so-called \emph{cohomological skeleton} for the foliation $\F$, i.e., a subvariety that contains the cohomological information of the closed leaves. A retraction to the cohomological skeleton is further used to build homotopy operators for the flat foliated complex. This is based on Lemma \ref{lemma: infinite time pullback} and Proposition \ref{lemma: homotopies}, whose proofs we delay to Section \ref{subsection: homotopy operators}. 

The flat foliated cohomology and the homotopy operators will be the building blocks for the computation of flat Poisson cohomology and the existence of ``good'' homotopy operators for the flat Poisson complex. These will play a crucial role in the proof of the linearization theorem. 

\subsection{The statement and explanations of the results}\label{section: foliated ch}

The singular foliation $\F$ of $(\spl,\pi)$ restricts on $\spl\backslash\{0\}$ to a regular, corank two, unimodular foliation. In this subsection we first recall the foliated de Rham complex of a regular foliation. Then we adapt this notion to introduce the flat foliated complex of $\F$ on $\spl$, which gives rise to flat foliated cohomology. We also state the main results of Section \ref{section flat foli}: we describe the flat foliated cohomology groups and claim the existence of homotopy operators. The proofs of these results will occupy the rest of Section \ref{section flat foli} and Section \ref{subsection: homotopy operators}.

\subsubsection{Foliated cohomology}
We first look at a regular foliation $\mathcal{F}$ on a manifold $M$, and we denote the \textbf{complex of foliated forms} by 
\[(\Omega^{\bullet}(\mathcal{F}), \dif_{\mathcal{F}})\ \ \ \text{
where }\ \ \Omega^{\bullet}(\mathcal{F}):=\Gamma(\wedge^{\bullet} T^*\mathcal{F}).\]
This complex consists of smooth families of differential forms along the leaves of $\mathcal{F}$, and $\dif_{\mathcal{F}}$ is the leafwise de Rham differential. The resulting cohomology is called the \textbf{foliated cohomology} of $\F$, denoted $H^{\bullet}(\F)$. The normal bundle to $\F$, denoted by $\nu:=TM/T\F$, carries the \textbf{Bott connection}
\[\nabla:\Gamma(T\F)\times \Gamma(\nu)\to \Gamma(\nu),\ \ \ \  \nabla_X(\overline{V}):=\overline{[X,V]},\]
which, via the usual formula for the exterior derivative, induces a differential on $\nu$-valued forms $(\Omega^{\bullet}(\F,\nu),\dif_{\nabla})$. This yields the cohomology of $\F$ with values in $\nu$, denoted $H^{\bullet}(\F,\nu)$. Similarly, the dual connection on $\nu^*$ gives rise to a connection on the exterior powers $\wedge^{q}\nu^*$, and we obtain the complexes $(\Omega^{\bullet}(\F,\wedge^{q}\nu^*),\dif_{\nabla})$, with cohomology groups $H^{\bullet}(\F,\wedge^{q}\nu^*)$.

In particular, for $q=\mathrm{codim}(\F)$, we have a canonical isomorphism
\[(\Omega^{\bullet}(\F,\wedge^q\nu^*),\dif_{\nabla})\simeq (\CC^{\bullet},\dif),\]
where $(\CC^{\bullet},\dif)$ is the subcomplex of $(\Omega^{\bullet+q}(M),\dif)$, defined by:
\[ \CC^{p}:=\big\{\alpha\in\Omega^{p+q}(M)\, |\, \forall\, v_1,\ldots,v_p\in T_x\F\, :\,  i_{v_1}\ldots i_{v_p}\alpha\in \wedge^q\nu_x^*\big\}.\]
Assume that $\F$ is coorientable and let $\varphi\in \Gamma(\wedge^q\nu^*)\subset \Omega^q(M)$ be a defining $q$-form for $\F$, i.e., $\varphi$ is nowhere zero and $T\F=\ker (\varphi)$. Then the subcomplex $\CC^{\bullet}$ admits the alternative description:
\[(\CC^{\bullet},\dif)=(e_{\varphi}(\Omega^{\bullet}(M)),\dif) \subset (\Omega^{\bullet+q}(M),\dif),\]
where $e_{\varphi}(-)=\varphi\wedge (-)$ denotes the exterior product with $\varphi$.

The foliation $\F$ is called \textbf{unimodular} if there exists a defining $q$-form $\varphi$ which is closed: $\dif \varphi=0$. Such a form is parallel for the dual of the Bott connection, and it induces an isomorphism of complexes: 
\begin{align}\label{eq: cochain isomorphism}
    (\Omega^{\bullet}(\F),\dif_{\F})\diffto (\Omega^{\bullet}(\F,\wedge^q \nu^*),\dif_{\nabla}), \ \ \eta\mapsto \varphi \otimes  \eta.
\end{align}
Therefore, in this case, the subcomplex $\CC^{\bullet}$ can be used to calculate the foliated cohomology of $\F$.

\subsubsection{Flat foliated cohomology of $\spl$}
Note that on $\spl\setminus\{0\}$ the regular foliation $\F|_{\spl\setminus\{0\}}$ is unimodular with defining $2$-form
\[\varphi:=\dif f_1\wedge \dif f_2,\quad \textrm{where}\quad f=f_1+i f_2.\]
We denote by $\Omega_0^{\bullet}(\spl)$ the set of differential forms that are flat at the origin, i.e., all partial derivatives of the coefficients vanish at the origin. Motivated by the discussion in the previous subsection, we introduce: 
\begin{definition}\label{definition: flat foliated ch}
We call the cohomology of the complex 
\begin{equation}\label{ffcomplex}
(\CC^{\bullet}:=e_{\varphi}(\Omega^{\bullet}_0(\spl)),\dif)
\end{equation}
the \textbf{flat foliated cohomology} of $\spl$, and we denote it by $H^{\bullet}_0(\F)$.
\end{definition}
The rest of the section is devoted to the calculation of the groups $H^{\bullet}_0(\F)$.

\subsubsection{The foliated cohomology away from $S_0$}\label{subsection:away from S_0}

To gain some geometric intuition, let us first discuss the foliated cohomology of $\F$ restricted to the open set $\spl\backslash S_0$. Here, the foliation is given by a the fibers of a locally trivial fibration:
\[ f:\spl \setminus S_0 \to \C^{\times}:= \C\setminus \{0\}.\]

For any such locally trivial fibration $p:E\to B$ with connected fibers, if the cohomology groups of the fibers are finite-dimensional, then they fit into a smooth vector bundle:
\[\mathcal{H}^{\bullet}\to B,\quad \mathcal{H}^{k}_b:=H^{k}(p^{-1}(b)).\]
The global sections of $\mathcal{H}^{\bullet}$ can be identified with the cohomology of the simple foliation $\mathcal{F}(p)$ associated to the submersion $p$:
\[H^{\bullet}(\mathcal{F}(p))\simeq \Gamma(\mathcal{H}^{\bullet}).\]

In our case, the fibers $f^{-1}(z)=S_z$, for $z\in \C^{\times}$, are all diffeomorphic to $T^*S^2$, which is homotopy equivalent to $S^2$. 
Hence, we have that
\[H^{k}(\F|_{\spl\backslash S_0})\simeq \Gamma(\mathcal{H}^k)=0,\quad \textrm{for } k\notin \{0,2\},\] 
and clearly  
\[H^{0}(\F|_{\spl\backslash S_0})\simeq C^{\infty}(\C^{\times}).\]
The situation in degree 2 is a bit more involved, because the rank 1 vector bundle $\mathcal{H}^2\to \C^{\times}$ is non-trivial. The leafwise symplectic structures $\omega_1$ and $\omega_2$, associated to the Poisson structures $\pi_1$ and $\pi_2$ from \eqref{eq: real identification}, induce sections
\[[\omega_1], [\omega_2]\in \Gamma(\mathcal{H}^2).\]
We will see that the vanishing loci of these sections are the positive real half-line and the negative real half-line, respectively:
\begin{align*}
&[\omega_1(z)]=0\in H^2(f^{-1}(z))\iff\quad z\in (0,\infty),\\
&[\omega_2(z)]=0\in H^2(f^{-1}(z))\iff\quad z\in (-\infty,0).
\end{align*}
So, at any point in $\C^{\times}$ at least one of $[\omega_1]$ and $[\omega_2]$ is non-zero, and therefore they span all sections: 
\[\Gamma(\mathcal{H}^2)=\big\{g_1[\omega_1]+g_2[\omega_2]\ |\ g_1,g_2\in C^{\infty}(\C^{\times})\big\}.\]
Because $\mathrm{rank}(\mathcal{H}^2)=1$, the two sections are not independent, and we will prove that they satisfy the relation:
\begin{align}\label{relation:omega's}
    (|z|-x)[\omega_1] = -y[\omega_2],
\end{align}
where $z=x+iy$. The second foliated cohomology is given by:
\[H^{2}(\F|_{\spl\backslash S_0})\simeq 
\Gamma(\mathcal{H}^2)\simeq \frac{C^{\infty}(\C^{\times})\oplus C^{\infty}(\C^{\times})}{\big\{(g_1,g_2)\ |\  yg_1=(|z|-x) g_2\big\}}.\]

\subsubsection{The twisted module}\label{subsection: twisted module}
In the previous subsection, we dealt with the non-trivial line bundle $\mathcal{H}^2\to \C^{\times}$. This bundle does not extend smoothly at the origin. However, to express flat foliated cohomology, we use as inspiration the description of $\Gamma(\mathcal{H}^2)$ above to make sense of its ``space of sections flat at the origin".

The following $C_0^{\infty}(\C)$-module will play a crucial role in this paper: 
\begin{align}\label{eq: definition M}
    \mathcal{M}\subset C^{\infty}_0(\C)\oplus C^{\infty}_0(\C),\quad \mathcal{M}:=\{(g_1,g_2)\, |\, yg_1=-(|z|+x)g_2\}.
\end{align}

For later use, let us give some properties of the module $\mathcal{M}$.

\begin{proposition}\label{proposition: twisted module}
The following hold:
\begin{enumerate}[(1)]
    \item We have an isomorphism of $C_0^{\infty}(\C)$-modules 
    $J:\mathcal{M}\diffto \mathcal{K}$, where 
\[\mathcal{K}\subset C^{\infty}_0(\C)\oplus C^{\infty}_0(\C),\quad \mathcal{K}:=\{(g_1,g_2)\, |\, yg_1=(|z|-x)g_2\},\]
given by $J(g_1,g_2)=(-g_2,g_1)$. 
    \item We have the direct sum decomposition:
    \[\mathcal{M}\oplus \mathcal{K}= C^{\infty}_0(\C)\oplus C^{\infty}_0(\C).\]
\end{enumerate}
\end{proposition}
\begin{proof}
For item (1), one uses the relation $(|z|-x)(|z|+x)=y^2$. 

The direct sum decomposition corresponds to the projections:
\begin{align}
p_{\mathcal{M}}:C^{\infty}_0(\C)&\oplus C^{\infty}_0(\C)\to \mathcal{M},\quad 
p_{\mathcal{K}}:C^{\infty}_0(\C)\oplus C^{\infty}_0(\C)\to \mathcal{K},\nonumber\\
p_{\mathcal{M}}(g_1,g_2)&:=\frac{1}{2|z|}\left((|z|+x)g_1-y g_2, -yg_1 + (|z|-x)g_2\right),\label{eq: projection m}\\
p_{\mathcal{K}}(g_1,g_2)&:=\frac{1}{2|z|}\left((|z|-x)g_1+y g_2, yg_1 + (|z|+x)g_2\right).\nonumber \qedhere
\end{align}
\end{proof}

\subsubsection{The description of the flat foliated cohomology}\label{section: result ffch}

In this subsection we describe the groups $H^{\bullet}_0(\F)$. Because of the flatness condition at the origin, which we impose on the forms, the singularity of the foliation will bring no contributions in cohomology. In fact, with some adaptations, the discussion from the previous subsection can be extended to describe the flat foliated cohomology. 

In degree zero, we obtain smooth functions on $\spl$ that are constant along the leaves of $\F$ and flat at $0$. Of course, these are precisely the Casimir functions of $(\spl,\pi)$ (described in Proposition \ref{prop:Casimirs}) that are flat at 0. We have the following isomorphism:
\[C_0^{\infty}(\C)\diffto  H^{0}_{0}(\F)\quad \quad g \mapsto e_{\varphi}(g \circ f).\]

To represent the second cohomology classes, we will need the following:
\begin{lemma}\label{lemma: extension omega}
For $i=1,2$, there exist 2-forms $\tilde{\omega}_i\in \Omega^2(\spl \setminus \{0\})$ that extend the leafwise symplectic forms $\omega_i$
\[\omega_i =\tilde{\omega}_i|_{T\F}\quad \textrm{on}\quad \spl \setminus \{0\},\]
and such that, when multiplied with $R^2=\mathrm{tr}(AA^*)$, become smooth at $0$:
\[ R^2 \cdot \tilde{\omega}_i \in \Omega^2(\spl).\]
\end{lemma}

\begin{proof}
In the coordinates from Subsection \ref{section: formal poisson}, the extensions are defined as:
\begin{align*}
    \tilde{\omega}_1:&=\frac{2}{R^2}\sum_{c.p.}
    \big(
    x_1 (\dif y_2\wedge \dif y_3-\dif x_{2}\wedge\dif x_3) -y_1 (\dif x_2\wedge \dif y_3+ \dif y_2 \wedge\dif x_3)\big)\\
&=\frac{1}{|z|^2}\sum_{c.p.}
    \Re \big(-z_1 \dif \overline{z}_2\wedge\dif \overline{z}_3\big)\\
    \tilde{\omega}_2:&=\frac{2}{R^2}\sum_{c.p.}\big(y_1(\dif y_2\wedge\dif y_3 -\dif x_{2}\wedge\dif x_{3})
    +x_1 (\dif x_{2}\wedge\dif y_3+ \dif y_2\wedge\dif x_{3})\big)\\
&=\frac{1}{|z|^2}\sum_{c.p.}
    \Im \big(-z_1 \dif \overline{z}_2\wedge\dif \overline{z}_3\big)
\end{align*}
where $\sum_{c.p.}$ denotes the sum over all cyclic permutations of the indexes $1,2,3$. A direct calculation shows that they satisfy:
\[\pi_i^{\sharp} \circ \tilde{\omega}_i^{\flat} \circ \pi_i^{\sharp} = \pi_i^{\sharp}.\]
Another way to describe $\tilde{\omega}_i$ is as following: we define the two vector fields $V_1,V_2\in \mathfrak{X}^1(\spl \setminus\{0\})$ by
\begin{align}\label{eq: v}
    V_1:=\frac{1}{R^2}\sum_{i=1}^{3}(x_{i}\partial_{x_i}-y_{i}\partial_{y_i})\quad \text{ and }\quad V_2:= \frac{1}{R^2}\sum_{i=1}^{3}(x_i\partial_{y_i}+y_i\partial_{x_i})
\end{align}

The vector fields $V_1$ and $V_2$ trivialize the normal bundle of $\F$ away from the origin as they satisfy the relations
\begin{align}\label{eq:v transversality}
    \dif f_j(V_i)=\delta_{ij}.
\end{align}
The 2-form $\tilde{\omega}_i$ is the unique extension of $\omega_i $ satisfying
\[ i_{V_j}\tilde{\omega}_i=0 \ \ \ \ \ \text{ for }\ j \in \{1,2\}.\qedhere\]
\end{proof}

The extensions from the lemma have only mild singularities at the origin, in particular, for any flat function $h\in C^{\infty}_0(\spl)$, 
\[h\cdot \tilde{\omega}_i\in \Omega^{2}_0(\spl).\]


The following result describes the flat foliated cohomology of $\F$, and to state it we use the $C^{\infty}_0(\C)$-modules $\mathcal{M}$ and $\mathcal{K}$ from the previous subsection.  
The proof will occupy the rest of Section \ref{section flat foli} and Section \ref{subsection: homotopy operators}.

\begin{proposition}\label{proposition:flat foliated ch}
For $k\notin \{0,2\}$, $H^k_0(\F)\simeq 0$, and 
\begin{itemize}
    \item for $k=0$, we have the isomorphism:
    \begin{equation}\label{eq: zero fol iso}
    \begin{array}{cccc}
         &C_0^{\infty}(\C)&\xrightarrow{\simeq} & H^{0}_{0}(\F)\\
         &g&\mapsto &e_{\varphi}(g \circ f)
    \end{array}
    \end{equation}
    \item for $k=2$, we have a surjective map
    \begin{equation} \label{eq: quotient}
        \begin{array}{cccc}
             &C_0^{\infty}(\C)\oplus C^{\infty}_0(\C)&\to & H^{2}_{0}(\F) \\
             &(g_1,g_2)&\mapsto & [e_{\varphi}(g_1\circ f\cdot \tilde{\omega}_1 +g_2\circ f\cdot \tilde{\omega}_2)]
        \end{array}
    \end{equation}
with kernel $\mathcal{K}$. Hence  \eqref{eq: quotient} restricts to a $C^{\infty}_0(\C)$-linear isomorphism
    \begin{equation}\label{eq:2nd:flat:foliated}
\mathcal{M} \diffto  H^{2}_{0}(\F).
    \end{equation}
\end{itemize}
\end{proposition}

\begin{remark}
Note that the 2-forms
\[e_{\varphi}(g_1\circ f\cdot \tilde{\omega}_1 +g_2\circ f\cdot \tilde{\omega}_2)\in \Omega^{4}_0(\spl)\]
are independent of the chosen extensions. However, the explicit extensions were needed in order to show that these forms are indeed part of our complex.
\end{remark}

\subsubsection{Good homotopy operators}

In the second part of the paper we prove Theorem \ref{theorem: nondegenerate} using the Nash-Moser method. For this, mere descriptions of the cohomology will not suffice. We will need to construct so-called \emph{homotopy operators}, which play the role of partial inverses to the differential. To be of use in the analysis part, these maps need to satisfy a locality condition with respect to closed balls around the origin and certain estimates. We collect these properties in the following: 

\begin{definition}\label{definition: property H}
Let $V$ and $W$ be two vector bundles over an inner product space $H$ (we will only use $H=\spl$ and $H=\C$). Consider vector subspaces of sections
\[\mathscr{V}\subset \Gamma_0(V)\quad \textrm{and}\quad \mathscr{W}\subset \Gamma_0(W),\]
where $\Gamma_0(\cdot)$ denotes the space of sections that are flat at the origin. Let $\iota_r:\overline{B}_r\to H$ denotes the inclusion of the closed ball of radius $r>0$ and denote $\mathscr{V}_r:=\iota_r^*(\mathscr{V})$ and $\mathscr{W}_r:=\iota_r^*(\mathscr{W})$. After choosing trivializations of $V$ and $W$, we have norms  $\fnorm{\cdot}_{n,k,r}$ on
$\mathscr{V}_r$ and $\mathscr{W}_r$, defined for $s\in \mathscr{V}_r$
by
\begin{align}\label{eq: flat norms}
    \fnorm{s}_{n,k,r}:= \sup_{x\in \overline{B}_r}\ \sup_{a\in \N^m_0\, :\, |a|\le n}|x|^{-k}|D^a s(x)| \ \in \ [0,\infty).
\end{align}

We say that a linear operator: 
\[\ell:\mathscr{V}\to \mathscr{W}\]
satisfies \textbf{property SLB} for some $(a,b,c)\in \N_0^3$ (SLB from semi-local and bounded), if it induces linear operators $\ell_r$, for $r>0$, such that the diagrams:
\begin{equation}\label{eq: property H diagram}
    \begin{tikzpicture}[baseline= (a).base]
    \node[scale=.9] (a) at (0,0){
    \begin{tikzcd}
    \mathscr{V}\arrow{r}{\ell}\arrow{d}[swap]{\iota_r^*}& \mathscr{W} \arrow{d}{\iota_r^*}\\
    \mathscr{V}_r\arrow{r}{\ell_{r}}&\mathscr{W}_r
    \end{tikzcd}
    };
\end{tikzpicture}
\end{equation}
commute, and such that the following estimates hold
\begin{align}\label{eq: property H estimate}
    \fnorm{\ell_{r}(\alpha)}_{n,k,r}\le C\, \fnorm{\alpha}_{n+a,k+bn+c,r},\quad \forall\, \alpha\in \mathscr{V}_r,
\end{align}
for all $n,k\in \N_0$, where $C=C(r,n,k)$ depends continuously on $r$. 
\end{definition}

\begin{example} Note that any linear partial differential operator $\ell$ of degree $d$ satisfies the SLB property of $(d,0,0)$. 
\end{example}

\begin{remark}
Throughout the paper we will encounter many operators which satisfy the SLB property and we will keep track of the triples $(a,b,c)\in \N^3_0$ which appear. However, the reader may ignore their precise values, the only important thing for our proof is that the SLB property holds for some $(a,b,c)$.  
\end{remark}

This type of operators are closed under composition. 
\begin{lemma}\label{lemma: slb property}
Let $U, V$ and $W$ be three vector bundles over an inner product space $H$ and consider two linear operators:
\begin{align*}
    \ell_1 : \mathscr{U}\subset \Gamma_0(U)\to \mathscr{V}\subset \Gamma_0(V) \quad \textrm{and}\quad \ell_2:\mathscr{V}\subset \Gamma_0(V)\to \mathscr{W}\subset \Gamma_0(W).
\end{align*}
If $\ell_i$ satisfies property SLB for $(a_i,b_i,c_i)$, $i=1,2$, then the operator $\ell_2\circ \ell_1$ satisfies property SLB for $(a_1+a_2,b_1+b_2, a_2b_1 +c_1+c_2)$.
\end{lemma}

To calculate the flat foliated cohomology, we will reduce the complex from \eqref{ffcomplex} to quasi-isomorphic subcomplexes. It will be useful to package this procedure into the following standard notion (see, e.g., \cite{LS87}):

\begin{definition}\label{definition: dr}
A \textbf{deformation retraction} of a complex $(C^{\bullet},d)$ consists of a \textbf{cochain projection}, i.e., a cochain map
\begin{align*}
    p:(C^{\bullet},d) \to (C^{\bullet},d) \quad \textrm{such that}\quad p\circ p=p,
\end{align*}
together with \textbf{cochain homotopies} between $p$ and $\mathrm{id}_{C}$, i.e., linear maps
\begin{align*}
    h:C^{\bullet+1}\to C^{\bullet} \quad \textrm{such that}\quad p-\mathrm{id}_{C}= h\circ d+d\circ h.
\end{align*}
\end{definition}

The conditions in the definition imply that the maps:
\[(C^{\bullet},d)\stackrel{p}{\to}(p(C^{\bullet}),d)\quad \textrm{and}\quad (p(C^{\bullet}),d)\stackrel{\mathrm{incl}}{\hookrightarrow}(C^{\bullet},d)\]
are inverses to each other in cohomology, thus we have an isomorphism:  
\[H^{\bullet}(C)\simeq H^{\bullet}(p(C)).\]

The following is the quantitative version of Proposition \ref{proposition:flat foliated ch} as it gives homotopy operators for the complex $(\CC^{\bullet},\dif)$ computing the flat foliated cohomology from \eqref{ffcomplex}. The proof will be given later. 

\begin{proposition}\label{proposition: property H fol}
There exists a deformation retraction $(p_{\F},h_{\F})$ for the complex $(\CC^{\bullet},\dif)$ such that: 
\begin{enumerate}[(1)]
    \item $\dif\circ p_{\F}=0$;
    \item $h_{\F}$ satisfies property SLB for $(0,5,35)$.
\end{enumerate} 
\end{proposition}

The first condition says that the differential vanishes on $p_{\F}(\CC^{\bullet})$. Thus, from the discussion above, every cohomology class has a unique representative in $p_{\F}(\CC^{\bullet})$, and so we have an isomorphism:
\[p_{\F}(\CC^{\bullet})\simeq H^{\bullet}_{0}(\F).\]

The second condition together with the homotopy relation imply that also $p_{\F}$ satisfies property SLB, but for $(1,5,40)$. In particular all these operators can be restricted to closed balls centered at the origin. 

\subsection{Averaging the $SU(2)$-action}\label{section: su2 average}

The operators from Proposition \ref{proposition: property H fol} will be constructed in two steps. In this subsection, we perform the first step: we average with respect to the action of $SU(2)\subset SL(2,\C)$. Consider the operator 
\[p_{SU(2)}:\Omega^{\bullet}(\spl)\to \Omega^{\bullet}(\spl),\]
\[\alpha\mapsto \int_{SU(2)}\Ad_{U}^*(\alpha)\mu(U),\]
where we integrate with respect to the normalized Haar measure $\mu$ on $SU(2)$. We show that $p_{SU(2)}$ induces a deformation retraction of the subcomplex computing flat foliated cohomology. 

\begin{proposition}\label{proposition: invariance skeleton}
There are operators $h_{SU(2)}$ such that $(p_{SU(2)},h_{SU(2)})$ is a deformation retraction of $(\Omega^{\bullet}(\spl),\dif)$. Moreover, this pair restricts to a deformation retraction of $(\CC^{\bullet},\dif)$, and both restricted maps satisfy property SLB for $(0,0,0)$.
\end{proposition}
\begin{remark}
The proof follows a general method for constructing homotopy operators for actions of compact and connected groups on complexes, which we have learned from \cite{VV14}. We have discussed this method in some greater generality in \cite{thesis_Florian}, here we apply it explicitly to the case of $SU(2)$.
\end{remark}

\begin{proof}
It follows easily that 
$p_{SU(2)}$ is a cochain projection of the complex $(\Omega^{\bullet}(\spl),\dif)$ whose image is the $SU(2)$-invariant subcomplex. 

Next, consider the pullback of the Haar measure (which we regard as a volume form, after choosing an orientation) to $\mathfrak{su}(2)$ via the exponential:
\[ \lambda:=\exp^*(\mu)\in \Omega^3(\mathfrak{su}(2)).\]

One easily checks (see e.g.\ \cite[Section 1.5]{DK}) that the exponential map of $SU(2)$ restricts to a diffeomorphism: 
    \[\exp:\mathcal{I}\diffto SU(2)\backslash\{-\mathds{1}\},\]
    where $\mathcal{I}\subset \mathfrak{su}(2)$
denotes the open ball of radius $\sqrt{2}\pi$ with respect to the norm $|X|=\sqrt{\mathrm{tr}(XX^*)}$ and on the boundary it satisfies \[\exp (\partial\mathcal{I})=\{-\mathds{1}\}.\]
This allows us to write:
\[p_{SU(2)}(\alpha)=\int_{SU(2)}\Ad_{U}^*(\alpha)\mu(U) 
=\int_{\overline{\mathcal{I}}}\Ad_{\exp(X)}^*(\alpha)\lambda(X).
\]
Denote by $\mathrm{ad}_X\in \mathfrak{X}^1(\spl)$ the vector field whose flow is $\Ad_{\exp(tX)}$, i.e.\
\[\mathrm{ad}_X(A):=[X,A].\]
Then, using Cartan's formula, we have that:
\[\frac{\dif}{\dif t}\Ad_{\exp(tX)}^*(\alpha)=\dif \Ad_{\exp(tX)}^* i_{\mathrm{ad}_X}(\alpha)+
\Ad_{\exp(tX)}^* i_{\mathrm{ad}_X}\dif (\alpha).
\]
This implies that:
\begin{align*}
p_{SU(2)}(\alpha)-\alpha
=&\int_{\overline{\mathcal{I}}}\big(\Ad_{\exp(X)}^*(\alpha)-\alpha\big)\lambda(X)\\
=&\int_{\overline{\mathcal{I}}}\int_0^1
\frac{\dif}{\dif t}\Ad_{\exp(tX)}^*(\alpha)\, \dif t\lambda(X)\\
=&\dif 
\int_{\overline{\mathcal{I}}}\int_0^1 \Ad_{\exp(tX)}^* i_{\mathrm{ad}_X}(\alpha)\dif t\lambda(X)\\
&+\int_{\overline{\mathcal{I}}}\int_0^1 \Ad_{\exp(tX)}^* i_{\mathrm{ad}_X}(\dif \alpha)\dif t\lambda(X).
\end{align*}
Thus, if we define:
\[h_{SU(2)}(\beta):=\int_{\overline{\mathcal{I}}}\int_0^1 \Ad_{\exp(tX)}^* i_{\mathrm{ad}_X}(\beta)\dif t\lambda(X),
\]
then we obtain the homotopy relation:
\[p_{SU(2)}(\alpha)-\alpha=\dif h_{SU(2)}(\alpha)+h_{SU(2)}(\dif \alpha).\]

Next, let us note that both $p_{SU(2)}$ and $h_{SU(2)}$ preserve
\[\CC^{\bullet}=e_{\varphi}(\Omega^{\bullet}_0(\spl))\subset \Omega^{\bullet+2}(\spl).\]
This follows because $\varphi=\dif f_1\wedge\dif f_2$ is $SU(2)$-invariant, and also satisfies $i_{\mathrm{ad}_X}\varphi=0$, and because the action fixes the origin. 

Next, since the $SU(2)$-action preserves the metric on $\spl$, it follows that both operators induce operators above closed balls around the origin, i.e.\ the first condition of property SLB holds. The inequalities from the second condition follow by using the explicit form of the norms.
\end{proof}


\subsection{The ``cohomological skeleton'' $\mathfrak{S}$} \label{section: skeleton}

The main ingredient to calculate the flat foliated cohomology is what we call the ``cohomological skeleton'' $\mathfrak{S}$ associated with the foliation $\F$. One may think of $\Sk$ as an object that contains the cohomological information of the closed leaves of $\F$. In our case, we will see that this is satisfied by the set of normal matrices, and so we define the \textbf{cohomological skeleton} as 
\begin{align*}
\Sk :=\{ A\in \spl |[A,A^*]=0\}.
\end{align*}

This subsection is divided as follows. We first give an alternative descriptions of $\Sk$, then we give an explicit continuous projection map $r_{\Sk}$ from $\mathfrak{sl}_2(\C)$ onto $\Sk$ and study its properties, then we show that the smooth locus of $\Sk$ is an embedded submanifold and we find a desingularization for $\Sk$. 

\subsubsection{Characterizations of $\Sk$}

To give other descriptions of $\Sk$, we first recall some basic linear algebra. Since any $A\in \spl$  satisfies its own characteristic equation, we have
\begin{equation}\label{eq:invariant polynomials a}
    \begin{array}{rcc}
          A^2 = -f\mathds{1}, \  & &\  (A^*)^2=-\overline{f}\mathds{1}, 
    \end{array}
\end{equation}
The standard norm squared on $\mathfrak{sl}_2(\C)$ is given by
\begin{align}\label{eq:norm}
  R^2=R(A) ^2&:= |A|^2= \tr(A A^*). 
\end{align}
For $A\in \spl$, the characteristic equation of $AA^*$ is: 
\begin{align}\label{eq:invariant polynomials aah}
    (AA^*)^2- R^2 AA^*+\nf^2\mathds{1}=0.
\end{align}
Since $AA^*$ has real eigenvalues, the discriminant of this equation is positive, which implies the following very useful inequality:
\begin{equation}\label{eq:norm-fiber inequality}
   2\nf\leq R^2.
\end{equation}

Normal matrices can be described as follows.
\begin{lemma}\label{lemma: characterization skeleton}
For $A\in \spl$ the following are equivalent:
\begin{enumerate}[(1)]
    \item $A\in \Sk$
    \item $A$ is unitarily diagonizable
    \item $R^2=2|f|$
    \item $AA^*=\nf \mathds{1}$.
\end{enumerate}
\end{lemma}
\begin{proof}
The equivalence (1)$\iff$ (2) is the content of the spectral theorem.

Next, by Schur's decomposition theorem, any $A\in \spl$ is unitary equivalent to an upper triangular matrix $\Lambda$:
\[ A=U\Lambda U^{*}, \quad U\in SU(2), \quad   \Lambda =\begin{pmatrix}
     \lambda_{11} & \lambda_{12} \\
     0&-\lambda_{11}
     \end{pmatrix}\]
Using this, we obtain the following expression for $|f|$:
\begin{align*}
    \nf  =\det(AA^*)^{\frac{1}{2}} = \det(\Lambda \Lambda^*)^{\frac{1}{2}}=|\lambda_{11}|^2
\end{align*}
and for the norm of $A$ we find: 
\begin{align*}
    R^2 =|A|^2=\tr(AA^*)=\tr(\Lambda \Lambda^*) = 2|\lambda_{11}|^2+ |\lambda_{12}|^2.
\end{align*}
Thus, $R^2=2|f|$ if and only if $\lambda_{12}=0$, i.e., $\Lambda$ is diagonal. This proves the implication (3) $\implies$ (2). The other implications: 
\[ (2)\ \  \implies  \ \ (4)\ \  \implies  \ \ (3) \]
are immediate. 
\end{proof}

\subsubsection{The retraction to the cohomological skeleton}\label{subsubsec_retraction}
The goal of this subsection is to define a continuous retraction $r_{\Sk}$ from $\spl$ onto $\Sk$ and study its properties. In Section \ref{subsection: homotopy operators} we will realize $r_{\Sk}$ as the infinite flow of a vector field.

First, we introduce the smooth function:
\begin{equation*}
    \begin{array}{cccc}
         g_{\infty} :&\spl \setminus S_0&\to &\herm(\C) \\
         & A&\mapsto&\left(\mathds{1}+\frac{1}{\nf}AA^*\right)^{\frac{1}{2}}
    \end{array}
\end{equation*}
where $\herm(\C)$ denotes the set of positive definite Hermitian $2\times 2$-matrices. In defining this map, we used the diffeomorphism: 
\[\sqrt{\cdot}: \herm(\C)\to \herm(\C).\]
Note that elements in $\herm(\C)$ are invertible. Define the retraction: 
\begin{align}\label{retract}
r_{\Sk}&:\spl \to \Sk\subset \spl\nonumber \\
r_{\Sk} (A)&:=
   \begin{cases}
    g_{\infty}(A)^{-1}Ag_{\infty}(A)&\text{ if } f(A)\ne 0\\
    0&\text{ if } f(A)= 0
\end{cases}
\end{align}

\begin{lemma}\label{lemma:properties of p}
The map $r_{\Sk}$ is a projection onto $\Sk$:
\[r_{\Sk}(\spl)=\Sk \quad \textrm{and}\quad r_{\Sk}|_{\Sk}=\mathrm{id}_{\Sk},\]
and has the following properties: 
it is continuous on $\spl$ and smooth on $\spl\setminus S_0$, it is $SU(2)$-equivariant and it preserves the fibers of $f$.
\end{lemma}
\begin{proof}
We first check that the image of $r_{\Sk}$ lies indeed inside of $\Sk$. Note that by \eqref{eq:invariant polynomials a} we have the identity
\begin{align*}
    A \ g_{\infty}(A)^{2}A^*  =&\ AA^*+\nf \mathds{1}  = \nf \,  g_{\infty}(A)^{2},
\end{align*}
and therefore:
\begin{equation}\label{eq: retraction well-defined}
    r_{\Sk}(A)r_{\Sk}(A)^* = \ g_{\infty}(A)^{-1}A \ g_{\infty}(A)^{2}A^* g_{\infty}(A)^{-1}= \nf\, \mathds{1}.
\end{equation}
We easily obtain that $r_{\Sk}$ preserves $f$
\[f(r_{\Sk}(A))=\det (r_{\Sk}(A))= \det (A) =f(A),\]
and therefore Lemma \ref{lemma: characterization skeleton} (4) and \eqref{eq: retraction well-defined} imply that $r_{\Sk}(A)\in \Sk$. 

Next, for $A\in \Sk\backslash S_0$, again by Lemma \ref{lemma: characterization skeleton} (4), we have that $g_{\infty}(A)=\sqrt{2}\mathds{1}$, and so $r_{\Sk}(A)=A$. Since $\Sk\cap S_0=\{0\}$,
we conclude that $r_{\Sk}|_{\Sk}=\id_{\Sk}$.

Using again \eqref{eq: retraction well-defined} we obtain
\[ \lvert r_{\Sk}(A)\rvert^2 =2\nf,\]
which proves continuity of $r_{\Sk}$ at points in $S_0\subset \spl$. Continuity and smoothness on the set $\spl\setminus S_0$ follow from the formula of $r_{\Sk}$ on this set. That $r_{\Sk}$ is $SU(2)$-equivarients follows immediately. 
\end{proof}

\subsubsection{The cochain projection associated to $r_{\Sk}$}

Even though $r_{\Sk}$ is not smooth along $S_0$, we will show in Section \ref{subsection: homotopy operators} that it induces a pullback map on flat forms:
\begin{lemma}\label{lemma: infinite time pullback}
For every $\alpha \in \Omega_{0}^{\bullet}(\spl)$, the form $r_{\Sk}^*(\alpha)|_{\spl\backslash S_0}$ extends to a smooth form on $\spl$, which is flat at $0$:
\[p_{\Sk}(\alpha)\in \Omega^{\bullet}_0(\spl).\]
\end{lemma}

By Lemma \ref{lemma:properties of p}, $r_{\Sk}^2=r_{\Sk}$, and therefore $p_{\Sk}$ is a cochain projection:
\[p_{\Sk}:(\Omega^{\bullet}_0(\spl),\dif)\to (\Omega^{\bullet}_0(\spl),\dif).\] 
The same lemma also implies that $p_{\Sk}$ commutes with $e_{\varphi}$ (it suffices to check this away from $S_0$), therefore $p_{\Sk}$ restricts to a cochain projection of $(\CC^{\bullet},\dif)$. 

The proof of the following proposition will occupy most of Section \ref{subsection: homotopy operators}:

\begin{proposition}\label{lemma: homotopies}
There are operators $h_{\Sk}$ such that $(p_{\Sk},h_{\Sk})$ is a deformation retraction of $(\Omega^{\bullet}_0(\spl),\dif)$. Moreover, this pair restricts to a deformation retraction of $(\CC^{\bullet},\dif)$, and $h_{\Sk}$ satisfies property SLB for $(0,5,35)$.
\end{proposition}

By combining the operators from Propositions \ref{proposition: invariance skeleton} and \ref{lemma: homotopies} we obtain the deformation retraction from Proposition \ref{proposition: property H fol} (for such compositions, see e.g.\ \cite[Lemma 2.4]{LS87})

\begin{corollary}\label{corollary:chain:proj} 
The pair $(p_{\F},h_{\F})$ defined by 
\[p_{\F}:=p_{SU(2)}\circ p_{\Sk} \quad \text{and}\quad h_{\F}:=p_{SU(2)}\circ h_{\Sk}+h_{SU(2)}\]
is a deformation retraction for the complex $(\CC^{\bullet},\dif)$, and $h_{\F}$ satisfies property SLB for $(0,5,35)$.
\end{corollary}
\begin{proof}
By Lemma \ref{lemma:properties of p}, $r_{\Sk}$ is $SU(2)$-equivariant, and this implies that $p_{\Sk}$ and $p_{SU(2)}$ commute. Hence $p_{\F}$ is a cochain projection. The homotopy relation follows easily: 
\begin{align*}
p_{\F}-\mathrm{id}=&p_{SU(2)}(p_{\Sk}-\mathrm{id})+p_{SU(2)}-\mathrm{id}\\
=&p_{SU(2)}(\dif h_{\Sk}+h_{\Sk}\dif )+\dif h_{SU(2)}+ h_{SU(2)}\dif \\
=&\dif (p_{SU(2)} h_{\Sk}+ h_{SU(2)})+(p_{SU(2)}h_{\Sk}+h_{SU(2)})\dif\\
=&\dif h_{\F }+ h_{\F}\dif.
\end{align*}
Property SLB for the original operators implies it for $h_{\F}$. 
\end{proof}

In the following subsections, we will see that $\dif\circ p_{\F}=0$, which will conclude the proof of Proposition \ref{proposition: property H fol}.

\subsubsection{A desingularization of $\Sk$}\label{section: descing skeleton}

The skeleton $\Sk$ is a smooth submanifold away from $0$. In this subsection, we construct a desingularization for $\Sk$. 

First note that, by Lemma \ref{lemma: characterization skeleton} (2), $\Sk$ is precisely the image of the map: 
\begin{equation*}\label{eq:surjective map}
    \begin{array}{cccc}
         \widehat{\rho}:&SU(2)\times \C &\to &\spl  \\
         &(U,\lambda)&\mapsto& U \begin{pmatrix}
         i\lambda&0\\
         0&-i\lambda
         \end{pmatrix}U^*
    \end{array}
\end{equation*}
This map is invariant under the right action on $SU(2)$ of the subgroup 
\[\Big\{
\begin{pmatrix}
          z&0\\
          0&\overline{z}
         \end{pmatrix}\, |\, z\in S^1
\Big\}.\]
Therefore, the map descends to a map on the quotient:
\[\rho: S^2\times \C \to \spl,\]
where we identify $SU(2)/S^1\simeq S^2$ via the Hopf fibration:
\begin{equation}\label{eq:Hopf}
         \begin{pmatrix}
         a&b\\
         -\overline{b}&\overline{a}
         \end{pmatrix} \mapsto (\lvert a\rvert^2 -\lvert b\rvert^2,\Im(-2ab),\Re(-2ab)),
\end{equation}
where we regard $S^2\subset \R^3$. For later use, let us note that, if we use coordinates $w=(w_1,w_2,w_3)\in S^2\subset \R^3$, then $\rho$ is given by
\begin{equation}\label{eq:cohomology map}
    \rho(w,\lambda)=\begin{pmatrix}
         i\lambda \cdot w_1 &\lambda \cdot (-w_2+iw_3)\\
         \lambda \cdot (w_2+iw_3)&-i\lambda \cdot w_1
         \end{pmatrix}
\end{equation}

\begin{lemma}\label{lemma:identification of skeleton}
The map $\rho$ has the following properties:
\begin{enumerate}[(1)]
\item $\rho$ is $SU(2)$-equivariant for the natural left action on $S^2\simeq SU(2)/S^1$ and the action on $\spl$ by conjugation;
\item $f\circ \rho(w,\lambda)=\lambda^2$, and so the preimages of the leaves of $\F$ are given by 
\[\rho^{-1}(S_z)=S^2\times\big\{\pm\sqrt{z}\big\},\]
where $\pm\sqrt{z}$ are the two square roots of $z\in \C$;
\item $\rho$ descends to a smooth map 
\[\underline{\rho}:S^2\times_{\Z_2}\C\to \spl,\]
where $\Z_2$ acts on $S^2\times \C\subset \R^5$ by multiplication with $-1$; 
\item the regular part of the skeleton $\Sk^{\textrm{reg}}:=\Sk\backslash\{0\}$ is an embedded submanifold of $\spl$ and 
$\underline{\rho}$ restricts to a diffeomorphism:
\begin{equation}\label{eq:desingularization}
\underline{\rho}:S^2\times_{\Z_2}\C^{\times}\diffto \Sk^{\text{reg}}.
    \end{equation}
\end{enumerate}
\end{lemma}

\begin{remark}
We think of the map $\underline{\rho}:S^2\times_{\Z_2}\C\to \Sk$ as a \textbf{desingularization} of the skeleton. However, for calculations it will be more convenient to use the map $\rho$.
\end{remark}

\begin{proof}
Items (1), (2) and (3) are immediate. An easy way to show that $\Sk^{\textrm{reg}}$ is an embedded submanifold is to note that, because $\Sk=\mathrm{im}(\widehat{\rho})$, it follows that $\Sk^{\mathrm{reg}}$ is the preimage under the submersion:
\[\mathrm{pr}:\spl\backslash\{0\}\to \mathbb{CP}(\spl)\]
of a (compact!) orbit of the induced $SU(2)$-action. 


For the last part, we first show that \eqref{eq:desingularization} is a bijection. Surjectivity follows because the image of $\widehat{\rho}$ is $\Sk$. To show injectivity, we need to understand when two elements $(U,\lambda),(\tilde{U},\tilde{\lambda})\in SU(2)\times \C $ have the same image under $\widehat{\rho}$. First, by comparing the determinants, we have that $\lambda=\pm \tilde{\lambda}$. For $\lambda =\tilde{\lambda}\ne 0$, note that the centralizer of the matrix \[\begin{pmatrix}
\lambda &0 \\
0& -\lambda
\end{pmatrix}\]
in $SU(2)$ the subgroup
\[S^1\simeq \Big\{ V(\mu):= \begin{pmatrix}
\mu&0\\
0&\overline{\mu}
\end{pmatrix}
| \ \mu \in S^1\Big \}.\]
In the case $\lambda =-\tilde{\lambda}\ne 0$ we use that 
\[ \begin{pmatrix}
\lambda &0 \\
0& -\lambda
\end{pmatrix}= \gamma \begin{pmatrix}
-\lambda &0 \\
0& \lambda
\end{pmatrix}\gamma^*, \quad \text{where}\quad  \gamma :=\begin{pmatrix}
0&1\\
-1&0
\end{pmatrix}\]
Using these, we obtain for $\lambda\neq 0$:
\[ \widehat{\rho}(U,\lambda)=\widehat{\rho}(\tilde{U},\tilde{\lambda}) \ \  \Longleftrightarrow \ \   \exists\, \mu \in S^1 , i\in \{0,1\}\, : \,  (\tilde{U},\tilde{\lambda} )=(UV(\mu)\gamma^i ,(-1)^i\lambda).\]
By first taking the quotient modulo $S^1$, and identifying $SU(2)/S^1\simeq S^2$ via \eqref{eq:Hopf}, we obtain that $\widehat{\rho}$ descends to $\rho$. Next, using the relation 
\[ V(\mu)\gamma = \gamma V(-\overline{\mu}).\]
we obtain that the right action of $\gamma$ on $SU(2)$ induces a well-defined $\Z_2$-action on $SU(2)/S^1$, which under the map \eqref{eq:Hopf} becomes multiplication by $-1$ on $S^2$. Hence, we obtain for $\lambda\neq 0$ that:
\[\rho(w,\lambda)=\rho(\tilde{w},\tilde{\lambda}) \ \  \Longleftrightarrow \ \  (\tilde{w},\tilde{\lambda} )=(\pm w ,\pm \lambda).\]
This implies that \eqref{eq:desingularization} is injective. 

Next, it is easy to see that the restriction of $\rho$ to $S^2\times \C^{\times}$ is an immersion. This implies that the map from \eqref{eq:desingularization} is an injective immersion whose image is $\Sk^{\mathrm{reg}}$. But since $\Sk^{\mathrm{reg}}$ is an embedded submanifold, it follows that \eqref{eq:desingularization} is a diffeomorphism.
\end{proof}

\subsubsection{Using the desingularization to calculate cohomology}\label{subsection:desing}

By Corollary \ref{corollary:chain:proj}, the flat foliated cohomology $H^{\bullet}_0(\F)$ can be computed using the subcomplex $p_{\F}(\CC^{\bullet})\subset \CC^{\bullet}$. In this subsection, we will describe this subcomplex by using the desingularization map $\rho$. As a consequence we obtain the description of the flat foliated cohomology from Proposition \ref{proposition:flat foliated ch} and also conclude the proof of Proposition \ref{proposition: property H fol}.

To describe the result, we first discuss the behavior of flat functions on $\C$ under pullback along the map:
\[\sigma:\C\to \C, \quad \sigma(\lambda)=-\lambda.\]
Consider the eigenspaces: 
\[C^{\infty}_{0,\pm \sigma}(\C):=\{g\in C^{\infty}_0(\C)\,|\, g=\pm \sigma^*(g)\},\]

The following properties of these spaces will be proven in Appendix \ref{eigenspaces}:
\begin{lemma}\label{lemma:pullbacks and invariant functions}
\begin{enumerate}[(1)]
    \item The following map is an isomorphism of rings: 
    \[C^{\infty}_0(\C)\diffto C^{\infty}_{0,\sigma}(\C)\]
    \[g\mapsto g\circ \mathrm{sq}, \quad \textrm{where}\quad \mathrm{sq}(\lambda):=\lambda^2.\]
    \item The following map is surjective:
    \begin{equation*}
        \begin{array}{cccc}
             & C^{\infty}_0(\C)\oplus C^{\infty}_0(\C) &\to &C^{\infty}_{0,-\sigma}(\C)\\
             &(g_1,g_2)&\mapsto &-\lambda_1\cdot g_1\circ \mathrm{sq}+\lambda_2\cdot g_2\circ \mathrm{sq},
        \end{array}
    \end{equation*} 
where $\lambda_1+i\lambda_2=\lambda$, and its kernel is given by $\mathcal{K}$.
\item The map from (2) restricts to a linear isomorphism:
\[\mathcal{M} \diffto C^{\infty}_{0,-\sigma}(\C).\]
\end{enumerate}
\end{lemma}

\vspace{-5pt}
Next, we introduce the complex:
\[(\DD^{\bullet},\dif)\subset 
(\Omega^{\bullet+2}(S^2\times \C),\dif)
\]
given by
\begin{equation}\label{eq: foliated complex skeleton}
\DD^{k} = \left\{\begin{array}{lr}
        C^{\infty}_{0,\sigma}(\C)\, \rho^*(\varphi), & \text{for } k=0\\
       \quad\quad \quad 0, & \text{for } k=1\\
        C^{\infty}_{0,-\sigma}(\C)\, \rho^*(\varphi)\wedge \omega_{S^2}, & \text{for } k=2
        \end{array}\right.
\end{equation}
where we regard functions on $\C$ as functions on $S^2\times\C$ via the pullback along the second projection and $\omega_{S^2}$ is the standard area form on $S^2$
\begin{align}\label{eq: standard form s2}
    \omega_{S^2}= \sum_{c.p.}w_1 \dif w_2\wedge \dif w_3
\end{align}
Clearly, $\DD^{\bullet}$ consists only of closed forms. Note that, by Lemma \ref{lemma:identification of skeleton} (2),
\begin{equation}\label{eq:pull back of varphi}
    \rho^*(\varphi)=4(\lambda_1^2+\lambda_2^2)\dif \lambda_1\wedge \dif \lambda_2.
\end{equation}

We have the following:

\begin{proposition}\label{proposition: chain isomorphism}
The map $\rho^*$ restricts to a cochain isomorphism:
\[\rho^*:(p_{\F}(\CC^{\bullet}),\dif)\diffto (\mathscr{D}^{\bullet},\dif=0).\]
\end{proposition}

\begin{proof}
We split the proof into several steps:
\smallskip

\noindent\underline{Step 1}: Injectivity. Let $\alpha\in p_{\F}(\CC^{\bullet})$ such that $\rho^*(\alpha)=0$. Recall that, by Lemma \ref{lemma:identification of skeleton}, $\rho:S^2\times \C^{\times}\to \Sk^{\textrm{reg}}$ is a 2-to-1 covering map, and therefore $i^*(\alpha)=0$, where $i:\Sk^{\textrm{reg}}\hookrightarrow\spl$ denotes the inclusion. Since $r_{\Sk}(\spl\backslash S_0)=\Sk^{\textrm{reg}}$, it follows that $r_{\Sk}^*(\alpha)|_{\spl\backslash S_0}=0$, and so $p_{\Sk}(\alpha)=0$. Since $\alpha\in p_{\F}(\CC^{\bullet})$, we have that:
\[\alpha=p_{\F}(\alpha)=p_{SU(2)}\circ p_{\Sk}(\alpha)=0.\]
This proves injectivity.
\smallskip

\noindent\underline{Step 2}: $\mathrm{Im}(\rho^*)\subset \mathscr{B}^{\bullet}$, where we denote 
\begin{align*}
    \mathscr{B}^{\bullet}:=e_{\rho^*(\varphi)}(\Omega^{\bullet}_0(S^2\times \C))^{SU(2)\times \Z_2}\subset \Omega^{\bullet+2}(S^2\times \C),
\end{align*} 
and the subscript $_0$ stands for forms flat at $S^2\times\{0\}$. Note that, since the action of $SU(2)\times \Z_2$ on $S^2\times\C$ preserves the submanifold $S^2\times \{0\}$, the subcomplex in \eqref{eq: foliated complex skeleton} is well-defined. Note that also $\rho^*(\varphi)$ is $\Z_2$-invariant. 

Step 2 is a consequence of the following facts: elements in $p_{\F}(\CC^{\bullet})$ are $SU(2)$-invariant, all elements in the image of $\rho^*$ are $\Z_2$-invariant, elements in $p_{\F}(\CC^{\bullet})$ are flat at the origin and $\rho(S^2\times \{0\})=0$.

\smallskip

\noindent\underline{Step 3}: $\DD^{\bullet}=\mathscr{B}^{\bullet}$. Clearly, $\DD^{\bullet}\subset \mathscr{B}^{\bullet}$. For the converse, note that 
\[\Omega^k(S^2)^{SU(2)} = \left\{\begin{array}{lr}
        \R, & \text{for } k=0\\
    0, & \text{for } k=1\\
        \R\, \omega_{S^2}, & \text{for } k=2
        \end{array}\right.\]
Next, note that if we evaluate any $\alpha \in \mathscr{B}^{k}$ along $S^2\times \{\lambda\}$, we obtain: 
\[ \alpha|_{S^2\times \{\lambda\}} = \rho^*(\varphi)_{\lambda}\wedge \beta,\]
with $\beta\in \Omega^k(S^2)^{SU(2)}$. This implies that $\mathscr{B}^1=0$. For $k=0$ we immediately obtain $\alpha=h\, \rho^*(\varphi)$, with $h\in C^{\infty}_{0,\sigma}(\C)$. For $k=2$, we obtain that $\alpha=h\, \rho^*(\varphi)\wedge \omega_{S^2}$. Since $\gamma^*(\omega_{S^2})=-\omega_{S^2}$, where $\gamma(w)=-w$, it follows that $h\in C^{\infty}_{0,-\sigma}(\C)$.

\smallskip

\noindent\underline{Step 4}: $\DD^{\bullet}\subset \mathrm{Im}(\rho^*)$. In degree 0, by Lemma \ref{lemma:pullbacks and invariant functions}, every element in $\DD^0$
can be written as $e_{\rho^*(\varphi)}(g\circ \mathrm{sq})$, for some $g\in C^{\infty}_0(\C)$. By Lemma \ref{lemma:identification of skeleton} (2):
\begin{equation*}
\rho^*(e_{\varphi}(g\circ f))=e_{\rho^*(\varphi)}(g\circ \mathrm{sq}).
\end{equation*}
Next, since $f$ is $SU(2)$-invariant and $r_{\Sk}$ preserves $f$, we have that 
\[p_{\F}(e_{\varphi}(g\circ f))=e_{\varphi}(g\circ f),\]
which proves the statement in degree 0.

For degree 2, we recall first the explicit form of $\rho$ from \eqref{eq:cohomology map}
\begin{align}\label{eq:cohomology map explicit}
\rho(w_1,w_2,w_3,\lambda)=(z_1,z_2,z_3), \quad z_j = \lambda\cdot w_j 
\end{align}
where we use the coordinates $z_j=x_j+iy_j$ on $\spl$ from Section \ref{section: formal poisson}. Using these in the explicit formula for $\tilde{\omega}_{i}$ from Lemma \ref{lemma: extension omega}, we have: 
\begin{align*}
\rho^*(\tilde{\omega}_1+i\tilde{\omega}_2)&=\rho^*\Big(-\frac{1}{|z|^2}\sum_{c.p.}
    z_1 \dif \overline{z}_2\wedge\dif \overline{z}_3\Big)\\
    &=-\frac{1}{|\lambda|^2}\sum_{c.p.}
    w_1 \lambda (\overline{\lambda} \dif w_2 +w_2\dif\overline{\lambda})\wedge
    (\overline{\lambda} \dif w_3 +w_3\dif\overline{\lambda})\\
&=-\sum_{c.p.} \overline{\lambda}\, w_1 \dif w_2\wedge \dif w_3+ w_1(w_3\dif w_2-w_2\dif w_3)\wedge\dif \overline{\lambda}\\
&=-\overline{\lambda}\,\omega_{S^2}.
\end{align*}
Hence, we obtain:
\begin{equation}\label{eq:pullback forms}
    \begin{array}{rcl}
         \rho^*(\tilde{\omega}_1)&=&-\lambda_1\, \omega_{S^2}\\ \rho^*(\tilde{\omega}_2)&=&\lambda_2\, \omega_{S^2}.
    \end{array}
\end{equation}

Consider an element 
\[\alpha=e_{\rho^*(\varphi)}(h\cdot \omega_{S^2})\in\DD^2,\quad \textrm{with}\quad h\in C^{\infty}_{0,-\sigma}(\C).\] 
By Lemma \ref{lemma:pullbacks and invariant functions}, there are two flat functions $g_1,g_2\in C^{\infty}_0(\C)$ such that 
\[h=-\lambda_1\cdot g_1\circ \mathrm{sq}+\lambda_2\cdot g_2\circ \mathrm{sq}.\]
Define the form
\[\eta:=e_{\varphi}\big(g_1\circ f\cdot \tilde{\omega}_1+g_2\circ f \cdot\tilde{\omega}_2\big)\in \CC^2.\]
By Lemma \ref{lemma:identification of skeleton} (2) and \eqref{eq:pullback forms}, we have that:
\[\rho^*(\eta)=\alpha.\]
The proof will be concluded once we show that:
\begin{equation}\label{eq:final:to:prove}
\rho^*(p_{\F}(\eta))=\alpha.
\end{equation}
First, since $\eta$ is closed, we have that:
\[p_{\F}(\eta)-\eta=\dif h_{\F}(\eta),\]
and so 
\begin{equation}\label{eq:homotopy on eta}
\rho^*(p_{\F}(\eta))-\alpha= \dif \rho^*(h_{\F}(\eta)).    
\end{equation}
Next, since $\rho^*(p_{\F}(\eta))\in \mathscr{D}^2$, there exists $\tilde{h}\in C^{\infty}_{0,-\sigma}(\C)$ such that 
\[\rho^*(p_{\F}(\eta))=e_{\rho^*(\varphi)}(\tilde{h}\cdot \omega_{S_2}\big).\]
Next, since $h_{\F}(\eta)\in \CC^1$, we can write 
\[\rho^*(h_{\F}(\eta))=e_{\rho^*(\varphi)}(\theta),\]
for some $\theta\in \Omega^1(S^2\times \C)$. Thus equation \eqref{eq:homotopy on eta} can be written as:
\[e_{\rho^*(\varphi)}\big((\tilde{h}-h)\cdot\omega_{S^2}-\dif \theta \big)=0.\]
Next, using the form of $\rho^*(\varphi)$ from \eqref{eq:pull back of varphi}, we contract the last equation with $\partial_{\lambda_1}$ and $\partial_{\lambda_2}$, then divide it by $4|\lambda|^2$, and we pull it back via the embedding $i_{\lambda}:S^2\hookrightarrow S^2\times\C$, $w\mapsto (w,\lambda)$; we obtain: 
\[(\tilde{h}(\lambda)-h(\lambda))\cdot\omega_{S^2}=\dif i_{\lambda}^*(\theta).\]
Passing to cohomology, this implies that $\tilde{h}=h$, and so \eqref{eq:final:to:prove} holds. 
\end{proof}

We can prove now the two statements about flat foliated cohomology:
\begin{proof}[Proof of Proposition \ref{proposition: property H fol}]
By Proposition \ref{proposition: chain isomorphism}, $\dif=0$ on $p_{\F}(\CC^{\bullet})$. The rest of the statement was already proven in Corollary \ref{corollary:chain:proj}.
\end{proof}

\begin{proof}[Proof of Proposition \ref{proposition:flat foliated ch}]
Because $p_{\F}$ is a deformation retraction of $(\CC^{\bullet},\dif)$ and because $\dif\circ p_{\F}=0$, we have that every flat foliated cohomology class has a unique representative in $p_{\F}(\CC^{\bullet})$; thus, by Proposition \ref{proposition: chain isomorphism}, we have the isomorphisms: 
\[H^{k}_0(\F)\simeq p_{\F}(\CC^{k})\simeq \DD^{k}.\]

This implies that $H^{k}_0(\F)=0$, for $k\notin\{0,2\}$. 

The case $k=0$ follows from the proof of Proposition \ref{proposition: chain isomorphism}, where we have seen that the set
\[\{e_{\varphi}(g\circ f)\, |\, g\in C^{\infty}_0(\C)\}\]
is fixed under $p_{\F}$ and that $\rho^*$ maps this set isomorphically to $\DD^{0}$. Of course, this case follows also from Proposition \ref{prop:Casimirs}, because $H^0_0(\F)$ consists of Casimir functions of $(\spl,\pi)$ that are flat at the origin (times $\varphi$). 

Finally, let $k=2$. For $g_1,g_2\in C^{\infty}_0(\C)$, denote by 
\[\eta(g_1,g_2):=e_{\varphi}(g_1\circ f\cdot \tilde{\omega}_1+g_2\circ f\cdot\tilde{\omega}_2)\in \CC^2.\]
Consider the map: 
\[\alpha: C^{\infty}_0(\C)\oplus C^{\infty}_0(\C)\to \DD^2,\]
\[(g_1,g_2)\mapsto \rho^*\circ p_{\F}(\eta(g_1,g_2)).\]
By the proof of Proposition \ref{proposition: chain isomorphism}, $\alpha$ is given by
\[\alpha(g_1,g_2)=(-\lambda_1\cdot g_1\circ \mathrm{sq}+\lambda_2\cdot g_2\circ \mathrm{sq})\rho^*(\varphi)\wedge\omega_{S^2}.\]
By part (2) of Lemma \ref{lemma:pullbacks and invariant functions}:
\[\ker(\alpha)=\mathcal{K},\]
and by part (3), $\alpha$ restricts to an isomorphism between the spaces:
\[ \alpha|_{\mathcal{M}}:\mathcal{M}\diffto \DD^2.\]
By combining the following:
\begin{itemize}
    \item[-]$\rho^*:p_{\F}(\CC^2)\diffto \DD^2$ is an isomorphism;
    \item[-] every class in $H^2_0(\F)$ has a unique representative in $p_{\F}(\CC^2)$;
    \item[-] $\eta(g_1,g_2)$ and $p_{\F}(\eta(g_1,g_2))$ are in the same cohomology class;
\end{itemize}    
the statement follows.
\end{proof}

\section{Homotopy operators for the foliated complex}\label{subsection: homotopy operators}

In this section we prove Lemma \ref{lemma: infinite time pullback} and Proposition \ref{lemma: homotopies}. In Subsection \ref{section: definition}, we construct the homotopy operators as the infinite flow of a specific vector field, and then we reduce Lemma \ref{lemma: infinite time pullback} and Proposition \ref{lemma: homotopies} to their more concrete versions: Lemma \ref{convergence of the flow} and Lemma \ref{well-defined}. We prove these results in Subsection \ref{section: convergence of the flow} by using a detailed analysis of the flow of the vector field, which is provided in Subsection \ref{subsection_analysis_flow}.

\subsection{Construction of the homotopy operators}\label{section: definition}
In order to construct the homotopy operators from Proposition \ref{lemma: homotopies}, we build a foliated homotopy between the identity map and the retraction $r_{\Sk}$. We will do this using the flow of the vector field $W\in \mathfrak{X}^1(\mathfrak{sl}_2(\C))$ given by:
  \begin{align}\label{W}
   W_A=&\frac{1}{4}[A,[A,A^*]].
  \end{align}
  Here the bracket denotes the standard Lie bracket on $\mathfrak{sl}_2(\C)$ and we think of $W$ as a real vector field.
\begin{lemma}\label{lemma: properties vector field}
 The vector field $W$ has the following properties:
 \begin{enumerate}[(1)]
     \item $W$ vanishes precisely on $\Sk$.
     \item $W$ is tangent to the foliation: $\Lie_W f=0$.
     \item $W$ is $SU(2)$-invariant.
     \item $\Lie_W(R^2 )\leq 0$, with equality only along $\Sk$.
\end{enumerate}
\end{lemma}
\begin{proof}
\begin{enumerate}[(1)]
    \item For the norm of $W$ we get:
    \begin{equation*}
        \begin{array}{rcl}
        |W_A|^2& = & \frac{1}{16}\tr([A,[A,A^*]][[A,A^*],A^*]\\
        & \stackrel{\eqref{eq:invariant polynomials a}}{=} & \frac{1}{16}\tr(4(f\cdot A^*+AA^*A)(\overline{f}\cdot A+A^*AA^*))\\
        &\stackrel{\eqref{eq:invariant polynomials a}}{=} & \frac{1}{4}\tr(\nf^2(A^*A-2AA^*)+(AA^*)^3)\\
        &=&\frac{1}{4}\tr(AA^*((AA^*)^2-\nf^2))\\
        &\stackrel{\eqref{eq:invariant polynomials aah}}{=}&\frac{1}{4}\tr(AA^*)(\tr(AA^*)^2-4\nf^2)\\
        &=&\frac{1}{8}R^2 (R^4-4\nf^2)
        \end{array}
    \end{equation*}
    which is zero iff $A\in \Sk$ due to Lemma \ref{lemma: characterization skeleton}.
    \item A straightforward calculation for $A\in \spl \setminus S_0$ gives:
    \begin{equation*}
        \begin{array}{ccl}
             \Lie_W f &=& \frac{\dif}{\dif t}f(A+tW_A)|_{t=0}  \\
             & =&\frac{\dif}{\dif t}\det(A+tW_A)|_{t=0}\\
            &=&\det(A)\tr(A^{-1}W_A)\\ &=&\frac{1}{4}\det(A)\tr([A,A^*]-A^{-1}[A,A^*]A)\\
             &=&0
        \end{array}
    \end{equation*}
    Since $\spl \setminus S_0$ is dense, (2) holds everywhere on $\spl$.
    \item Straightforward.
    \item We check this again by a direct calculation:
    \begin{equation*}
        \begin{array}{ccl}
             \Lie_W (R^2)&=& \frac{\dif}{\dif t}\tr((A+tW_A)(A^*+tW_A^*))|_{t=0}\\
             &=&\frac{1}{4}\tr([A,[A,A^*]]A^*+A[[A,A^*],A^*])\\
             &=&\frac{1}{4}\tr(4(\nf^2\mathds{1}-(AA^*)^2))\\
             &\stackrel{\eqref{eq:invariant polynomials aah}}{=}&4\nf^2 -R^4\\
             &\stackrel{\eqref{eq:norm-fiber inequality}}{\leq} &0
        \end{array}
    \end{equation*}
    with equality iff $A\in \Sk$ again due to Lemma \ref{lemma: characterization skeleton}.\qedhere
\end{enumerate}
\end{proof}

The last property in Lemma \ref{lemma: properties vector field} implies that the flow lines starting at a point inside the closed ball $\overline{B}_r(0)$ are trapped inside that ball; hence the flow is defined for all positive time, and will be denoted by:
  \begin{equation*}
  \begin{array}{cccc}
   \phi:&\R_{\geq 0}\times\  \mathfrak{sl}_2(\C) &\to&\mathfrak{sl}_2(\C)\\
   &(t,A)&\mapsto&\phi_t(A).
  \end{array}
  \end{equation*}
The above properties of $W$ imply that:
\begin{itemize}
\item $\phi_t$ fixes $\Sk$;
\item $\phi_t^*(f)=f$;
\item $\phi_t$ is $SU(2)$-equivariant.
\end{itemize}
In particular, $\phi_t^*$ preserves the complex $(\CC^{\bullet},\dif)$. Also $i_W$ preserves $\CC^{\bullet}$: since $\Lie_W f=0$, we have that $\Lie_{W}f_1=i_W\dif f_1=0$ and $\Lie_W f_2=i_W\dif f_2=0$. 

Next, Cartan's magic formula implies the homotopy relation
\begin{equation}\label{eq: at time t}
\phi_t^*(\alpha)-\alpha=\dif \circ h_t(\alpha)+h_t\circ \dif(\alpha),
\end{equation}
for all $\alpha\in \Omega^{\bullet}_0(\spl)$, where
\[h_t:\Omega^{\bullet}_0(\spl)\to \Omega^{\bullet-1}_0(\spl), \ \ h_t(\alpha)=\int_0^{t}i_W\phi_s^*(\alpha)\dif s.\]
By the discussion above, note that also $h_t$ preserves $\CC^{\bullet}$.

In order to prove Proposition \ref{lemma: homotopies} we will take the limit as $t\to\infty$ in \eqref{eq: at time t}. In the following subsection we will give explicit formulas for $\phi_t$, which imply the point-wise limit: 
\begin{align}\label{infinity now}
\lim_{t\to\infty}\phi_t(A)=r_{\Sk}(A), \ \ \ \forall\  A\in \mathfrak{sl}_2(\C).
\end{align}
The following results are much more involved, and their proofs will occupy the rest of this section:
\begin{lemma}\label{convergence of the flow}
On $\mathfrak{sl}_2(\C)\backslash S_0$, we have that \[\lim_{t\to\infty}\phi_t=r_{\Sk}\] 
with respect to the compact-open $C^{\infty}$-topology.
\end{lemma}

\begin{lemma}\label{well-defined}
For any $\alpha\in \Omega_{0}^{\bullet}(\mathfrak{sl}_2(\C))$, the following limit exists:
   \begin{align}\label{limit h}
    h_{\Sk}(\alpha):=\lim_{t \to \infty}{h_{t}(\alpha)}\in \Omega_{0}^{\bullet-1}(\mathfrak{sl}_2(\C)),
   \end{align}
with respect to the compact-open $C^{\infty}$-topology. Moreover, $h_{\Sk}$ satisfies property $SLB$ for $(0,5,35)$.
\end{lemma}

Recall that the existence of a limit with respect to the compact-open $C^{\infty}$-topology means that all partial derivatives converge uniformly on compact subsets; more details are given in 
Subsection \ref{subsubsection_limits}. 

The results above suffice to complete our proofs:

\begin{proof}[Proofs of Lemma \ref{lemma: infinite time pullback} and Proposition \ref{lemma: homotopies}]
Since the limit \eqref{limit h} is uniform on compact subsets with respect to all $C^k$-topologies, $h_{\Sk}$ satisfies
   \begin{align*}
    \dif h_{\Sk}(\alpha)&=\lim_{t \to \infty}{\dif h_{t}(\alpha)}.
   \end{align*}
From \eqref{eq: at time t}, we obtain that for any $\alpha\in \Omega_0^{\bullet}(\mathfrak{sl}_2(\C))$
\[\lim_{t\to\infty}\phi_t^*(\alpha)=\alpha+\dif h_{\Sk}(\alpha)+h_{\Sk}(\dif \alpha)\]
holds for the compact-open $C^{\infty}$-topology. On the other hand, on $\mathfrak{sl}_2(\C)\backslash S_0$, $\lim_{t\to\infty}\phi_t=r_{\Sk}$, and since this limit is also with respect to the compact-open $C^{\infty}$-topology, we have that 
\[\lim_{t\to\infty}\phi_t^*(\alpha)|_{\mathfrak{sl}_2(\C)\backslash S_0}=r_{\Sk}^*(\alpha)|_{\mathfrak{sl}_2(\C)\backslash S_0}.\]
Therefore, $r_{\Sk}^*(\alpha)|_{\mathfrak{sl}_2(\C)\backslash S_0}$ extends to a smooth form on $\mathfrak{sl}_2(\C)$, which we denote by $p_{\Sk}(\alpha)$. This implies Lemma \ref{lemma: infinite time pullback} and the homotopy equation:
\[p_{\Sk}(\alpha)-\alpha=\dif h_{\Sk}(\alpha)+h_{\Sk}(\dif \alpha).\]
Finally, since $h_t$ commutes with $e_{\varphi}$, and this condition is closed, we have that $h_{\Sk}(\CC^{\bullet})\subset \CC^{\bullet-1}$. Thus, the above relation holds on $\CC^{\bullet}$, and so we obtain also Proposition \ref{lemma: homotopies}.
\end{proof}


\subsection{The analysis of the flow}\label{subsection_analysis_flow}

In this subsection we analyse the flow $\phi_t$ of the vector field $W$ from \eqref{W}. 

\subsubsection{An explicit formula for the flow}\label{subsection:explicit flow}

The result will be expressed in terms of the function: 
\begin{align*}
          & g:\R_{\ge 0}\times \mathfrak{sl}_2(\C) \to  
          \herm(\C) \\
          & g_t(A):=\left(\mathds{1}+\frac{\tanh(\nf t)}{\nf}AA^*\right)^{\frac{1}{2}}
\end{align*}
and $g_t(A)$ is defined for $f(A)=0$ by continuity as $\left(\mathds{1}+t AA^*\right)^{\frac{1}{2}}$. Here we use the square root of a positive definite Hermitian matrix, as in Subsection \ref{subsubsec_retraction}. Note also that $g$ is a smooth function; this follows because $\tanh(u)/u$ is an even function, and therefore it is a smooth function of $u^2$. Finally, note that, for $f(A)\neq 0$, $\lim_{t\to \infty}g_t(A)=g_{\infty}(A)$. 

We give now an explicit formula for the flow and for the evolution along the flow lines of the norm and of the measure of ``a-normality":
 \begin{lemma}\label{le:flow}
    The flow of $W$ starting at $A\in \mathfrak{sl}_2(\C)$ is given by:
    \begin{equation}\label{eq:A_t}
    A_t:=\phi_t(A)= g_t(A)^{-1}A g_t(A),     
    \end{equation}
    for $t\geq 0$, and moreover, $A_t$ satisfies: 
    \begin{align}
        R_t^2&:=\tr(A_tA_t^*)=2\nf\frac{2\nf\tanh(2\nf t)+ R^2}{2\nf+R^2\tanh(2\nf t)}\label{eq:R_t^2}\\
        K_t&:=[A_t,A_t^*]=
        \frac{2\nf}{2\nf\cosh(2\nf t)+R^2\sinh(2\nf t)}[A,A^*]\label{eq: commutator flow}
    \end{align}
 \end{lemma}

\begin{proof}
First we prove \eqref{eq:R_t^2}. In the proof of Lemma \ref{lemma: properties vector field} (4), we have seen that $\mathcal{L}_W(R^2)=4|f|^2-R^4$. Therefore, $x_t:=R_t^2$ satisfies the following ODE:
\begin{equation}\label{eq:ODE}
x'_t=a^2-x_t^2,
\end{equation}
where $a=2|f|$. The solution of this ODE starting at $x_0=x$ is:
\begin{equation}\label{sol_ODE}
x_t=\frac{a\tanh(a t)+ x}{1+\tanh(a t)\frac{x}{a}}=\frac{\dif }{\dif t}\ln\Big(\cosh(at)+\sinh(at)\frac{x}{a}\Big).
\end{equation}
The second expression will be used later, and the first yields \eqref{eq:R_t^2}.

To prove equation \eqref{eq:A_t}, we first calculate the positive semi-definite Hermitian matrices $S_t:= A_tA_t^*\ge 0$. Using the definition of $W$ \eqref{W} and the characteristic equations \eqref{eq:invariant polynomials a}, we obtain the following ODE for $S_t$:
\begin{align*}
S'_t&=\frac{1}{4}\big(
[A_t,[A_t,A_t^*]]A_t^*+A_t[[A_t,A_t^*],A_t^*]\big)\\
&=\frac{1}{4}\big(2
(A_t)^2(A_t^*)^2-4(A_tA_t^*)^2+A_t(A_t^*)^2A_t+A_t^*(A_t)^2A_t^*\big)\\
&=\nf^2 \mathds{1}-S_t^2.
\end{align*}
This is an ODE of the same form as \eqref{eq:ODE}, but instead of being in $\R$ it is in the commutative algebra spanned by $\mathds{1}$ and $AA^*$, and one can easily see that the solution is of the same form as \eqref{sol_ODE}:
\[S_t =\frac{ \nf \tanh(\nf t)\mathds{1}+AA^*}{ \mathds{1}+\frac{\tanh(\nf t)}{\nf}AA^*} =\frac{\dif }{\dif t}\ln\Big(\cosh(\nf t)\mathds{1}+\frac{\sinh(\nf t)}{\nf}AA^*\Big).\]
In writing this formulas we use that any smooth function $\xi:\R_{>0}\to \R$ ``extends" to a smooth function on $\xi:\herm(\mathbb{C})\to \mathrm{Herm}(\mathbb{C})$, which follows from ``basic functional calculus''. For a diagonal matrix, the extension acts by applying $\xi$ to each diagonal entry, and in general, one uses that the extension is invariant under conjugation by unitary matrices.

Next, we note that the defining ODE for $A_t$ can be written in using $S_t$:
\begin{align*}
A'_t&=\frac{1}{4}[A_t,[A_t,A_t^*]]\\
&=\frac{1}{4}\big(
(A_t)^2A_t^*-2A_tA_t^*A_t+A_t^*(A_t)^2\big)\\
&=\frac{1}{2}[A_t,S_t], 
\end{align*}
where we used the characteristic equation \eqref{eq:invariant polynomials a}. Now, define 
 \[X_t:=\frac{1}{2}\int_0^t{S_s \dif s}=\frac{1}{2}\ln\Big(\cosh(\nf t)\mathds{1}+\frac{\sinh(\nf t)}{\nf}AA^*\Big).\]
The matrices $X_t ,X'_t $ belong to the commutative subalgebra of $\mathfrak{gl}_2(\C)$ generated by $AA^*$, and therefore they commute, which implies that: 
\[\frac{\dif}{\dif t}e^{X_t}=X'_t e^{X_t}=e^{X_t}X'_t.\]
Therefore: 
\[\frac{\dif}{\dif t}e^{X_t}A_te^{-X_t}=e^{X_t}([X'_t,A_t]+A'_t)e^{-X_t}=e^{X_t}(\frac{1}{2}[S_t,A_t]+A'_t)e^{-X_t}=0,\]
and so we obtain a closed formula:
\[A_t=e^{-X_t}Ae^{X_t}.\]
The formula \eqref{eq:A_t} from the statement follows because: 
\[e^{X_t}=\Big(\cosh(\nf t)\mathds{1}+\frac{\sinh(\nf t)}{\nf}AA^*\Big)^{\frac{1}{2}}=
\sqrt{\cosh(\nf t)} g_t(A).
\]

Finally, in order to prove \eqref{eq: commutator flow}, we calculate the derivative of the commutator $K_t$ using the definition of $W$ \eqref{W}:
\begin{align*}
K'_t&=\frac{1}{4}\big(
[[A_t,[A_t,A_t^*]],A_t^*]+[A_t,[[A_t,A_t^*],A_t^*]]\big)\\
&=\frac{1}{2}\big(
(A_t)^2(A_t^*)^2-2(A_tA_t^*)^2+2(A_t^*A_t)^2-(A_t^*A_t)^2\big)\\
&=-R_t^2K_t,
\end{align*}
where we have used the characteristic equations \eqref{eq:invariant polynomials a} and \eqref{eq:invariant polynomials aah}. This ODE has as solution: 
\[K_t=e^{-\int_{0}^tR^2_s\dif s}K_0.\]
Using the second form of \eqref{sol_ODE}, this expression becomes \eqref{eq: commutator flow}.
\end{proof}

\subsubsection{Estimates for the flow and its derivative}

To write the estimates, it will be convenient to introduce the following smooth function:
\begin{equation}\label{varsigma}
\epsilon_t:=\frac{1}{\cosh(2|f|t)+\frac{\sinh(2\nf t)}{2\nf}R^2},  
\end{equation}
which appeared in the formula \eqref{eq: commutator flow}, i.e.\ $K_t=\epsilon_t [A,A^*]$.

The following notation will be very useful. For $u\in \frac{1}{2}\N_{0}$ and $v\in \N_{0}$, let \[\mu_{u,v}:=\sum_{j=0}^{p}t^{u-\frac{j}{2}}R^{v-j}, \ \ p:=\mathrm{min}(2u,v).\]

The following inequalities is immediate: 
\begin{equation}\label{submult}
\mu_{u,v}\, \mu_{u',v'}\leq C \mu_{u+u', v+v'}
\end{equation}
for some constants $C=C(u,u',v,v')$. In particular:
\begin{equation*}
R^{v'}\mu_{u,v}=\mu_{0,v'}\mu_{u,v}\leq \mu_{u, v+v'}.
\end{equation*}
Also, for $w\in \frac{1}{2}\N_0$, note that
\begin{equation}\label{submult1}
 \mu_{u,v}\leq \mu_{u+w, v+2w}.
\end{equation}
We will use these inequalities later on.

In this section, for convenience, we rename the coordinates $(x_i,y_i)$ on $\spl$ from Subsection \ref{section: formal poisson} as $(x_i,x_{i+3})$. Then we obtain partial derivative operators $D^a$, for $a\in \N_0^6$, acting on functions in $C^{\infty}(\spl)$ (see Appendix \ref{subsection: derivatives}). 

As before, we denote $A_t:=\phi_t(A)$. 

\begin{proposition}\label{lemma:bounds on the flow}
For any $a\in \N_0^6$ with $n= \lvert a\rvert$ there is $C=C(a)$ such that:
 \begin{align*}
 \lvert D^a(A_t)\rvert\, &\le\,
 C\,\mu_{2n+\frac{1}{2},3n+2},\\
 \lvert D^a (A'_{t})\rvert\, &\le\, C\, \epsilon_t\, \mu_{2n+\frac{1}{2},3n+4}.
\end{align*}
\end{proposition}

\subsubsection{Proof of Proposition \ref{lemma:bounds on the flow}}

The proof is split into several steps, which occupy the rest of this section. First, we introduce some auxiliary quantities.




Any smooth even function $a(x)$ on $\R$ can be written as $a(x)=b(x^2)$, for a smooth function $b$ on $\R_{\geq 0}$. Using this fact, we introduce 3 functions $\theta_1,\theta_2$ and $\theta_3$ on $\R_{\geq 0}$ defined by:
\begin{equation}\label{theta_tanh}
\theta_1(x^2)=\frac{\tanh(x)}{x},\qquad
\theta_2(x^2)=\cosh(x),\qquad 
\theta_3(x^2)=\frac{\sinh(x)}{x}.  
\end{equation}

The derivatives of these functions satisfy: 
\begin{lemma}\label{lemma:estimate theta}
For $n\geq 0$, there is $C=C(n)$, such that for all $x\geq 0$:
\begin{align*}
|\theta_1^{(n)}(x^2)|&\leq C\frac{1}{(1+x)^{2n+1}},\\
|\theta_2^{(n)}(x^2)|&\leq C\frac{\cosh(x)}{(1+x)^{n}},\\
|\theta_3^{(n)}(x^2)|&\leq C\frac{\cosh(x)}{(1+x)^{n+1}}.
\end{align*}
\end{lemma}
\begin{proof}
It suffices to show the first inequality for $x\geq 1$, and this is equivalent to showing that the following function is bounded:  
\[|x^{2n+1}\theta_1^{(n)}(x^2)|.\]
We prove this by induction. For $n=0$, this follows using $|\tanh(x)|\leq 1$. 
Taking $n\geq 1$ derivatives in the defining equation of $\theta_1$ from \eqref{theta_tanh}, and multiplying by $x^n$, one obtains an expression of the form:
\[x^{2n+1}\theta_1^{(n)}(x^2)=\sum_{1\leq i\leq \frac{n+1}{2}}c_{n,i}x^{2(n-i)+1}\theta_1^{(n-i)}(x^2)+\frac{x^n}{2^n}\tanh^{(n)}(x).\]
By induction, the elements of the sum are bounded. Next, using that \[\tanh'(x)=1-\tanh(x)^2 \quad \textrm{and} \quad \tanh^{(2)}(x)=-2\tanh'(x)\tanh(x),\] we can write $\tanh^{(n)}(x)=\tanh'(x)P_n(\tanh(x))$, where $P_n$ is a polynomial. Clearly, $P_n(\tanh(x))$ and also $x^n\tanh'(x)=\frac{x^n}{\cosh{(x)}}$ are bounded functions, which concludes the proof of the first inequality.

The second inequality is easier: taking $n$ derivatives in the defining equation of $\theta_2$ from \eqref{theta_tanh}, we obtain
\[x^{n}\theta_2^{(n)}(x^2)=\sum_{n/2\leq i<n}c_{n,i}x^{2i-n}\theta_2^{(i)}(x^2)+\frac{1}{2^n}\cosh^{(n)}(x).\]
Using induction, the terms on the right are bounded by $\cosh(x)$, which is enough to conclude the proof of the second inequality. 

The third follows in the same way as the second.
\end{proof}

First, we estimate the derivatives of $\nf^2$:
\begin{lemma}\label{lemma:norm f estimate}
There exists $0<C$ such that for any $a\in \N_0^6$ with $n=\lvert a \rvert$:
\[ \lvert D^a (\nf^2)\rvert  \le C
    \begin{cases}
     \nf\cdot R &\text{ if } n= 1\\
     R^{4-n}&\text{ if }  2\le n \le 4\\
    0 &\text{ else.}
    \end{cases}\]
\end{lemma}
\begin{proof}
The cases $n\neq 1$ follow because $f$ is a homogeneous polynomial of degree 2. The case $n=1$, follows by applying the Leibniz rule in the decomposition $|f|^2=f\overline{f}$.
\end{proof}

Next, we estimate the size of $g_t$:

\begin{lemma}\label{lemma:gt norm}
There is a constant $C$ such that for any $A\in \spl$, we have: 
\[ \lvert g_t(A)\rvert\leq C\mu_{\frac{1}{2},1} \quad \textrm{and}\quad \lvert g_t^{-1}(A)\rvert \leq C.\]
\end{lemma}
\begin{proof}
By the characteristic equation \eqref{eq:invariant polynomials aah} the eigenvalues of $AA^*$ satisfy:
\[\lambda_{-}+\lambda_{+}=R^2,\quad \lambda_{-}\lambda_{+}=|f|^2.\]
The eigenvalues of $g_t(A)^2$ are 
\[1+\frac{\tanh(\nf t)}{\nf}\lambda_{\pm},\]
therefore, using also Lemma \ref{lemma:estimate theta}, we obtain:
\begin{align*}
|g_t(A)|&=\sqrt{\tr(g_t(A)^2)}=\Big(2+\frac{\tanh(\nf t)}{\nf}R^2\Big)^{\frac{1}{2}}\\
&\leq  C\Big(1+\frac{tR^2}{1+t\nf}\Big)^{\frac{1}{2}}\leq C\sqrt{1+tR^2}.
\end{align*}
The first inequality follows now from the following:
\[\sqrt{1+tR^2}\leq 1+t^{\frac{1}{2}}R=\mu_{\frac{1}{2},1}.\]
Similarly, we have that: 
\begin{align*}
\tr(g_t(A)^{-2})&=\frac{1}{1+\frac{\tanh(\nf t)}{\nf}\lambda_{-}}+\frac{1}{1+\frac{\tanh(\nf t)}{\nf}\lambda_+}\\
&=\frac{2+\frac{\tanh(\nf t)}{\nf}R^2}{1+\frac{\tanh(\nf t)}{\nf}R^2+\tanh(\nf t)^2}\\
&\leq 2
\end{align*}
which implies the second inequality.
\end{proof}

We calculate the size of the derivatives of $g_t^2$:
\begin{lemma}\label{lemma:gt estimate}
For any $a\in \N_0^6$ with $0<n=\lvert a\rvert$ there is a constant $C=C(a)$ such that:
\[\lvert D^a(g_t^2(A))\rvert \le C \mu_{n+1,n+2}.\]
\end{lemma}
\begin{proof}
Using the general chain rule (see Appendix \ref{subsection: derivatives}) and Lemma \ref{lemma:estimate theta}, we first calculate:
\[|D^a(t\theta_1(t^2|f|^2))|\leq C\sum_{k=1}^n\Big(\frac{t}{1+t|f|}\Big)^{2k+1}{\sum}' \prod_{j=1}^{k} D^{a^{j}}(\nf^2),\]
where $\sum '$ sums over all non-trivial $a^j$ with $a=\sum_{j=1}^{k}{a^j}$. Clearly, only the terms with $1\le \vert a^j\rvert \le 4$ are non-zero. Denote by $k_p$ the number of $a^j$'s with $\vert a^j\rvert=p$. Then
using Lemma \ref{lemma:norm f estimate}, we obtain:
\[ |\prod_{j=1}^{k} D^{a^{j}}(\nf^2)|\le C |f|^{k_1}R^{k_1}R^{4k_2-2k_2}R^{4k_3-3k_3}R^{4k_4-4k_4}=C|f|^{k_1}R^{4k-n-2k_1},\]
where we have also used that $k=\sum k_p$ and $n=\sum p k_p$. These decompositions immediately imply that $l:=2k-k_1$ satisfies $m:=\mathrm{max}(k,n/2)\leq l\leq M:=\mathrm{min}(n,2k)$. Therefore, we obtain:
\begin{align*}
|D^a(t\theta_1(t^2|f|^2))|&\leq C\sum_{k=1}^n\sum_{m\leq l\leq M}\Big(\frac{t}{1+t|f|}\Big)^{2k+1}|f|^{2k-l}R^{2l-n}\\
&= C\sum_{k=1}^n\sum_{m\leq l\leq M}\Big(\frac{t|f|}{1+t|f|}\Big)^{2k-l}
\Big(\frac{t}{1+t|f|}\Big)^{l+1}
R^{2l-n}\\
&\leq C\sum_{\frac{n}{2}\leq l\leq n}t^{l+1}R^{2l-n}\\
&\leq C\mu_{n+1,n}.
\end{align*}

Next, we apply the Leibniz rule in $g_t^2(A)= \mathds{1} +t\theta(t^2|f|^2) AA^*$, and obtain:
\begin{align*}
|D^a(g_t^2(A))|&\leq\sum_{b+c=a}|D^b(t\theta_1(t^2|f|^2))||D^{c}(AA^*)|\\
&\leq C\sum_{q=0}^2\mu_{n-q+1,n-q}\,\mu_{0,2-q}\\
&\leq C\sum_{q=0}^2\mu_{n-q+1,n+2-2q}\\
&\leq C\mu_{n+1,n+2}
\end{align*}
where for the first terms we have use the inequalities from above, and, because $AA^*$ is a homogeneous polynomial of degree two, that the second term is non-zero only for $q=|c|\leq 2$, and in these cases \[|D^{c}(AA^*)|\leq CR^{2-q}=C\mu_{0,2-q},\]
and finally, the inequalities for the polynomials $\mu_{k,m}$ \eqref{submult} and \eqref{submult1}. 
\end{proof}

\begin{lemma}\label{corollary:g deriv estimate}
For any $a\in \N_0^6$ with $n=\lvert a\rvert>0$ there is a constant $C=C(a)$ such that:
\begin{align*}
    \lvert D^a(g_t^{\pm 1}(A))\rvert \le &\, C \mu_{2n,3n}.
\end{align*}
\end{lemma}
\begin{proof}
Note that by the chain rule (see Appendix \ref{subsection: derivatives}) we have:
\[ D^a(g_t(A))=\sum_{1\leq k\leq n}D^k(x^{\pm \frac{1}{2}})(g_t^2(A)){\sum}' \prod_{j=1}^{k} D^{a^{j}}(g_t^2(A))\]
where $\sum '$ sums over all non-trivial $a^j$ with $a=\sum_{j=1}^{k}{a^j}$. The first part can be estimated using Lemma \ref{lemma:gt norm} by a constant: 
\begin{align*}
    \lvert D^k(x^{\pm \frac{1}{2}})(g_t^2(A))\rvert \leq &\ C\lvert g_t^{-2k\pm 1}(A)\rvert \leq C. \end{align*}
For the second part we set $n_j:=\lvert a^j\rvert $ and use Lemma \ref{lemma:gt estimate} and \eqref{submult} to get:
\begin{align*}
    \lvert\prod_{j=1}^{k}D^{a^{j}}(g_t^2(A))\rvert\le &\, C \prod_{j=1}^{k}\mu_{n_j+1,n_j+2}\\
    \le &\, C \mu_{n+k,n+2k}\\
    \le &\, C \mu_{2n,3n}.
\end{align*} 
Combining these, we obtain the conclusion. 
\end{proof}

\begin{proof}[Proof of the first part of Proposition \ref{lemma:bounds on the flow}]
We use the Leibniz rule to obtain
\begin{align*}
    D^a (A_t)=& \ \sum_{a^1+a^2+a^3=a}D^{a^1}(g_t^{-1}(A))D^{a^2}(A) D^{a^3}(g_t(A)).
\end{align*}
Denote by $n_j:=|a^j|$. Clearly, $n_2=0$ or $n_2=1$. To each term, we apply either Lemma \ref{lemma:gt norm} or \ref{corollary:g deriv estimate}. By Lemma \ref{lemma:gt norm}, the inequality from Lemma \ref{corollary:g deriv estimate} for $g_t^{-1}(A)$ holds also for $n_1=0$, and therefore we only need to consider separately the case when $n_3=0$ and $n_3>0$. In the first case, we obtain: 
\begin{align*}
|D^{a^1}(g_t^{-1}(A))||D^{a^2}(A)| |g_t(A)|&\leq C(\mu_{2n,3n}R+\mu_{2n-2,3n-3})\mu_{\frac{1}{2},1}\\
&\leq C \mu_{2n+\frac{1}{2},3n+2}.
\end{align*}
For the case with $n_3>0$, we also obtain:
\begin{align*}
|D^{a^1}(g_t^{-1}(A))||D^{a^2}(A)| |D^{a^3}(g_t(A))|
\leq & C(\mu_{2n_1,3n_1}R\mu_{2(n-n_1),3(n-n_1)}\\
&+ \mu_{2n_1,3n_1}\mu_{2(n-n_1-1),3(n-n_1-1)})\\
\leq& C(\mu_{2n,3n}R+
\mu_{2n-2,3n-3})\\
\leq& C\mu_{2n,3n+1}\\
\leq &C \mu_{2n+\frac{1}{2},3n+2}.\qedhere
 \end{align*}
\end{proof}

Next, we estimate the derivatives of the commutator:

\begin{lemma}\label{lemma:bounds on the commutator}
For any $a\in \N_0^6$ with $n=\lvert a\rvert $ there is a constant $C=C(a)$ such that:
\[ \lvert D^a(K_t)\rvert\le C \epsilon_t\, \mu_{2n,3n+2}.\]
\end{lemma}
\begin{proof}
First we estimate $\varsigma_t:=\epsilon_t^{-1}$. For this, we write:
\[\varsigma_t=\theta_2(4\nf^2t^2 )+tR^2\theta_3(4\nf^2t^2),\]
and estimate the two components. Applying Lemma \ref{lemma:estimate theta} and an argument similar to the proof of Lemma \ref{lemma:gt estimate}, we obtain that: 
\begin{align*}
|D^a(\theta_2(4\nf^2 t^2))|&\leq C \sum_{k=1}^n|\theta_2^{(k)}(4\nf^2 t^2)|t^{2k}{\sum}'\prod_{j=1}^k |D^{a^j}(|f|^2)|\\
&\leq C \cosh(2|f|t) \sum_{k=1 }^n
\sum_{m\leq l\leq M}
\Big(\frac{t^2}{1+\nf t}\Big)^k|f|^{2k-l}R^{2l-n}\\
&= C \cosh(2|f|t) \sum_{k=1 }^n
\sum_{m\leq l\leq M}
\Big(\frac{t|f|}{1+\nf t}\Big)^{2k-l}\frac{t^lR^{2l-n}}{(1+|f|t)^{l-k}}\\
&\leq C \cosh(2|f|t) 
\sum_{n/2\leq l\leq n}
t^lR^{2l-n}\\
&\leq   C \cosh(2|f|t)\mu_{n,n},
\end{align*}
where $m=\mathrm{max}(k, n/2)$ and $M:=\mathrm{min}(n,2k)$ and $n=|a|$. 

By a similar argument, we have that:
\begin{align*}
|D^a(t\theta_3(4\nf^2 t^2))|&\leq C \sum_{k=1}^n|\theta_3^{(k)}(4\nf^2 t^2)|t^{2k+1}{\sum}'\prod_{j=1}^k |D^{a^j}(|f|^2)|\\
&\leq C \cosh(2|f|t) \sum_{k=1 }^n
\sum_{m\leq l\leq M}
\Big(\frac{t|f|}{1+\nf t}\Big)^{2k-l}\frac{t^{l+1}R^{2l-n}}{(1+|f|t)^{l+1-k}}\\
&\leq C \cosh(2|f|t) 
\sum_{n/2\leq l\leq n}
t^{l+1}R^{2l-n}\\
&\leq  C \cosh(2|f|t)\mu_{n+1,n}.
\end{align*}
Therefore, applying the Leibniz rule, we obtain:
\begin{align*}
|D^a(tR^2\theta_3(4\nf^2t^2))|&\leq C\sum_{a^1+a^2=a}|D^{a^1}(R^2)||D^{a^2}(t\theta_3(4\nf^2t^2))|\\
&\leq C\cosh(2|f|t)
\sum_{q=0}^2R^{2-q}\mu_{n+1-q,n-q}\\
&\leq C\cosh(2|f|t)
\sum_{q=0}^2\mu_{n+1-q,n+2-2q}\\
&\leq C\cosh(2|f|t)\mu_{n+1,n+2}.
\end{align*}

Combining the estimates obtained so far, and using \eqref{submult1}, we obtain:
\[|D^a(\varsigma_t)|\leq C \cosh(2|f|t) \mu_{n+1,n+2}.
\]

Using this inequality, the chain rule,
that and $\cosh(2\nf t)\leq \varsigma_t$, we get: 
\begin{align*}
    \lvert D^{a}(\epsilon_t)\rvert \le&  C  \sum_{1\leq k\leq n}\varsigma_t^{-k-1} {\sum}' \prod_{j=1}^{k} |D^{a^{j}}(\varsigma)|\\
     \le & C\epsilon_t  \sum_{1\leq k\leq n} {\sum}'\prod_{j=1}^k\mu_{n_j+1,n_j+2}\\
     \le & C\epsilon_t  \sum_{1\leq k\leq n}  \mu_{n+k,n+2k}\\
     \le & C  \epsilon_t \mu_{2n,3n}
     \end{align*} 
where ${\sum}'$ is over all non-trivial decompositions $a=\sum_{j=1}^{k}a^{j}$, and $n_j=|a^j|$.

Recall from \eqref{eq: commutator flow} that $K_t=\varsigma^{-1}[A,A^*]$. By applying the Leibniz rule in this equality, and using the inequality above, we obtain the conclusion:
\begin{align*}
|D^a(K_t)|&\leq \sum_{a=a^1+a^2}|D^{a^1}(\epsilon_t) ||D^{a^2}([A,A^*])|\\
&\leq C \epsilon_t \sum_{q=0}^2\mu_{2n-2q,3n-3q}R^{2-q}\\
&\leq C\epsilon_t \sum_{q=0}^2\mu_{2n-2q,3n-4q+2}\\
&\leq C\epsilon_t \mu_{2n,3n+2}.\qedhere
\end{align*}
\end{proof}

\begin{proof}[Proof of the second part of Proposition \ref{lemma:bounds on the flow}]
In $A_t'=\frac{1}{4}[A_t,K_t]$ we use the Leibniz rule, then the first part of Proposition \ref{lemma:bounds on the flow}, then Lemma \ref{lemma:bounds on the commutator}, and we obtain the inequality from the statement:
\begin{align*}
 |D^a (A'_{t})|&\leq C \sum_{a^1+a^2=a}|D^{a^1}(A_t)||D^{a^2}(K_t)|\nonumber \\
 &\leq C\epsilon_t\sum_{n_1+n_2=n}\, \mu_{2n_1+\frac{1}{2},3n_1+2}
 \mu_{2n_2,3n_2+2}\nonumber \\
 &\leq C\, \epsilon_t\, \mu_{2n+\frac{1}{2},3n+4}.\qedhere
\end{align*}
\end{proof}

\subsection{Proofs of Lemma \ref{convergence of the flow}
and Lemma \ref{well-defined}}\label{section: convergence of the flow}

In the proof below we apply the criterion for the existence of limits of smooth families of functions from Appendix \ref{subsubsection_limits}.

\subsubsection{Proof of Lemma \ref{convergence of the flow}}
The set $\spl\backslash S_0=\{|f|>0\}$ is covered by the compact sets:
\[K(r,\delta):=\{R\leq r\}\cap \{|f|\geq \delta \}\subset \spl,\ \ \delta, r>0.\]
By the criterion in Appendix \ref{subsubsection_limits}, we need to bound the partial derivatives of $A'_t$ uniformly on $K(r,\delta)$ by an integrable function. By Lemma \ref{lemma:bounds on the flow}
\[\lvert D^a (A'_{t})\rvert \le C\epsilon_t\,\mu_{2n+\frac{1}{2},3n+4}\le C
(1+r)^{3n+4}e^{-2\delta t}t^{2n+\frac{1}{2}},\]
for $t\geq 1$, where we used that 
\[\epsilon_t\leq \cosh^{-1}(2|f|t)\leq e^{-2|f|t}\leq e^{-2\delta t}.\]
Since the right hand side is an integrable function, the conclusion follows.

\subsubsection{A crucial inequality}

The following inequality will play a crucial role in the proof of Lemma \ref{well-defined}:

\begin{lemma}\label{lemma: finite integrals}
For any $q\geq 1$ there is a constant $C=C(q)$ such that:
\[\epsilon_t\, R_t^{2q}\leq C \frac{R^{2q}}{(1+tR^2)^{q}},\]
for all $t\geq 0$ and all $A\in \spl$.
\end{lemma}

\begin{proof}
The inequality is equivalent to 
\begin{align}\label{eq:inequality rewrite}
    \Big(\frac{1}{\cosh(2|f|t)+\frac{\sinh(2\nf t)}{2\nf}R^2}\Big)^{\frac{1}{q}}\Big(\frac{R_t}{R}\Big)^2\leq  \frac{C}{1+R^{2}t}.
\end{align}
Denote $x:=2 \nf t$ and $y:=R^2t$. Then by Lemma \ref{le:flow}, we have:  
\[    \Big(\frac{R_t}{R}\Big)^2= 
\frac{\cosh(x)+\frac{x}{y}\sinh(x)}{\cosh(x)+\frac{y}{x}\sinh(x)}.\]
Using \eqref{eq:norm-fiber inequality}, i.e.\ that $x\leq y$, we have that:
\[    \Big(\frac{R_t}{R}\Big)^2\leq 
\frac{\cosh(x)+\sinh(x)}{\cosh(x)+\frac{y}{x}\sinh(x)}=
\frac{e^x}{\cosh(x)+\frac{y}{x}\sinh(x)}.\]
Therefore, \eqref{eq:inequality rewrite} follows if we prove:
\[ 
\frac{e^x}{\big(\cosh(x)+\frac{y}{x}\sinh(x)\big)^{1+\frac{1}{q}}}\leq \frac{C}{1+y}\]
for $0\leq x\leq y$. Clearly, it suffices to show this for $y\geq 1$, and on this domain this follows if we show that the function 
\[l(x,y):=\frac{ye^x}{\big(\cosh(x)+\frac{y}{x}\sinh(x)\big)^{1+\frac{1}{q}}}\]
is bounded on $0\leq x\leq y$. We have that:
\[\partial_yl(x,y)=\frac{e^x(\cosh(x)-y\frac1q\frac{\sinh(x)}{x})
}{\big(\cosh(x)+\frac{y}{x}\sinh(x)\big)^{2+\frac{1}{q}}}.\]
Therefore, for a fixed $x$, the $l(x,y)$ attains its maximum in 
\[y_x:=\frac{qx}{\tanh(x)}\geq x.\]
We have that 
\[l(x,y_x)=\frac{q}{(1+q)^{1+\frac{1}{q}}}\frac{xe^x}{\sinh(x)\cosh(x)^{\frac{1}{q}}}.\]
Since $l(x,y)\leq l(x,y_x)$, and  
\[\lim_{x\to 0}l(x,y_x)=l(0,q)=\frac{q}{(1+q)^{1+\frac{1}{q}}}, \quad \lim_{x\to \infty}l(x,y_x)=0,\]
is follows that $l(x,y)$ is bounded, which completes the proof.
\end{proof}

\subsubsection{Estimates for the pullback}

Next, we estimate the pullback under $\phi_t$ of flat forms.
We will prove all estimates on the closed ball $\overline{B}_r\subset \spl$ of radius $r>0$ centered at the origin. We will use the norms $\fnorm{\cdot}_{n,k,r}$ on flat forms $\Omega^i_0(\overline{B}_r)$ defined in \eqref{eq: flat norms}.

\begin{lemma}\label{pullback estimates}
For every $a\in \N^6_0$ with $n=|a|$, any $j,l\in \N_0$ there is a constant $C$ such that for any flat form $\alpha\in \Omega^{d}_0(\overline{B_r})$ and any $t\geq 0$:
\begin{align*}
|D^a(\phi_t^*(\alpha))|&\leq C R_t^{j+l}\mu_{\frac{5}{2}(n+d),5(n+d)}\fnorm{\alpha}_{n+l,j,r}.
\end{align*}
\end{lemma}
\begin{proof}
\underline{Step 1}: We first prove the estimate for $n=d=0$, i.e., for $g\in C^{\infty}_0(\overline{B_r})$ and for $a=0\in \N_0^6$. For $l=0$, the estimate follows
\[|\phi_t^*(g)|=R_t^{j}\frac{|\phi_t^*(g)|}{R_t^{j}}
\le R_t^{j}\fnorm{g}_{0,j,r}.
\]
For $l\geq 1$, we use Taylor's formula with integral remainder to write
\[g(A)=\sum_{|b|=l}A^{b}\int_{0}^1l(1-s)^{l-1}(D^{b} g)(sA)\dif s,\]
where we used that $g$ is flat at the origin, and we denoted $A^{b}=\prod_{i=1}^6 x_i^{b_i}$ in the coordinates introduced in Appendix \ref{subsection: derivatives}. Therefore we obtain:
\begin{align}\label{eq: integral remainder}
    \phi_t^*(g)(A)=\sum_{|b|=l}A_t^{b}\int_{0}^1l(1-s)^{l-1}(D^{b} g)(sA_t)\dif s.
\end{align}
We clearly have that
\[|A_t^{b}|\leq R_t^{l}.\]
For any $j\in \N_0$, we have that
\[|(D^{b} g)(sA_t)|\leq (sR_t)^{j}\fnorm{g}_{l,j,r}.\]
Using these inequalities in \eqref{eq: integral remainder}, we obtain the estimate in this case:
\begin{equation}\label{eq:Step1}
|\phi_t^*(g)|\leq CR_t^{l+j}\fnorm{g}_{l,j,r}.
\end{equation}

\noindent\underline{Step 2}: we prove now the estimate for a flat function $g$ and $a\in \N^6_0$ with $1\le n=|a|$. We use the chain rule (see Appendix \ref{subsection: derivatives}) to write:
\[D^a(g\circ \phi_t)=\sum_{1\leq |b|\leq n}D^b(g)(A_t)\ {\sum}'\prod_{i=1}^{6}\prod_{j=1}^{b_i}D^{a^{ij}}(x_{i,t}),
\]
where $b=(b_i)_{i=1}^6$ and $\sum'$ is the sum over all non-trivial decompositions: 
\[a=\sum_{i=1}^6\sum_{j=1}^{b_i}a^{ij}.\]
We apply \eqref{eq:Step1} with $l\leftarrow n-|b|+l$ to estimate the first term: 
\begin{align*}
    |D^b(g)(A_t)|&\leq  C R_t^{n-|b|+l+j}\fnorm{g}_{n+l,j,r}.
\end{align*}
We estimate the second term using Proposition \ref{lemma:bounds on the flow}
\begin{align*}
\big|\prod_{i=1}^6\prod_{j=1}^{b_i}D^{a^{ij}}(x_{i,t})\big|&\leq C \prod_{i=1}^6\prod_{j=1}^{b_i}
\mu_{2|a^{ij}|+\frac{1}{2},3|a^{i,j}|+2}\\
&\leq C\mu_{2n+\frac{1}{2}|b|,3n+2|b|}.
\end{align*}
The second case follows by using these two estimates, and the inequalities: 
\begin{align*}
R_t^{n-|b|}&\leq R^{n-|b|}=\mu_{0,n-|b|}\\
\mu_{0,n-|b|}\mu_{2n+\frac{1}{2}|b|,3n+2|b|}&\leq C\mu_{2n+\frac{1}{2}|b|,4n+|b|}\leq C\mu_{\frac{5}{2}n,5n}.
\end{align*}

\noindent\underline{Step 3}: let $\alpha\in \Omega^d_0(\overline{B}_r)$, $d\geq 1$, and $a\in \N^6_0$. Note that the coefficients of $\phi_t^*(\alpha)$ are sums of elements of the form 
\[g\circ \phi_t\cdot M(\phi_t),\] 
where $g\in C^{\infty}_0(\overline{B}_r)$ is a coefficient of $\alpha$ and $M(\phi_t)$ denotes the determinant of a minor of rank $d$ of the Jacobian matrix of $\phi_t$. By the Leibniz rule:
\[D^a(g\circ \phi_t\cdot M(\phi_t))=\sum_{b+c=a}D^b(g\circ \phi_t)D^c(M(\phi_t)).\]
For the first term, we apply Step 2 with $n\leftarrow |b|$ and $l\leftarrow |c|+l$:
\begin{align*}
|D^b(g\circ \phi_t)|&\leq C R_t^{|c|+l+j} \mu_{\frac{5|b|}{2},5|b|}\fnorm{\alpha}_{n+l,j,r}\\
&\leq C R_t^{l+j} \mu_{\frac{5|b|}{2},5|b|+|c|}\fnorm{\alpha}_{n+l,j,r}.
\end{align*}
Note that $M(\phi_t)$ is a homogeneous polynomial of degree $d$ in the first order partial derivatives of the components $x_{i,t}$ of $\phi_t$. By Proposition \ref{lemma:bounds on the flow} each partial derivative satisfies: 
\[\Big|D^{c}\Big(\frac{\partial x_{i,t}}{\partial x_j}\Big)\Big|\leq C \mu_{2|c|+\frac{5}{2},3|c|+5}.\]
Therefore, by applying the Leibniz rule and \eqref{submult}, we obtain:
\[|D^c M(\phi_t)|\leq C \mu_{2|c|+\frac{5}{2}d,3|c|+5d}.\]
The estimate follows by combing these two estimates and using that: 
\[
\mu_{\frac{5|b|}{2},5|b|+|c|}
\mu_{2|c|+\frac{5}{2}d,3|c|+5d}\leq 
\mu_{2n+\frac{|b|+5d}{2},4n+|b|+5d}\leq \mu_{\frac{5}{2}(n+d),5(n+d)}.\qedhere\]
\end{proof}

\subsubsection{Proof of Lemma \ref{well-defined}}

Let $\alpha\in \Omega^d_0(\spl)$. As explained in Subsection \ref{subsubsection_limits}, for each $r>0$, we need to bound the partial derivatives of $h'_t(\alpha)$ uniformly on the balls $\overline{B}_r$ by an integrable function.

Let $a\in \N^6_0$, with $n=|a|$. First, we show that for any $j\geq 2$ there is a constant $C=C(n,j)$ such that:
\begin{equation}\label{important_equation}
|D^a(h'_t(\alpha))|\leq C\, R^{3+j}(1+tR^2)^{\frac{5(n+d)+1-j}{2}}\fnorm{\alpha}_{n,j,r}.
\end{equation}

For this recall that
\[h'_t(\alpha)=i_{W}\phi_t^*(\alpha)=\phi_t^*(i_W\alpha)=\sum_{i=1}^6x'_{i,t}\,\phi_t^*(\alpha_i),\]
where $\alpha_i=i_{\partial_{x_i}}\alpha$ and $x'_{i,t}$ are denote the entries of the matrix $A'_t$. 

First we apply the Leibniz rule:
\[D^a(x_{i,t}'\,\phi_t^*(\alpha_i))=\sum_{b+c=a}D^b(x_{i,t}') D^c(\phi_t^*(\alpha_i)).\]
By using Proposition \ref{lemma:bounds on the flow}, we obtain:
\[|D^b(x'_{i,t})|\leq 
C\, \epsilon_t\, \mu_{2|b|+\frac{1}{2}, 3|b|+4}.\]
By applying Lemma \ref{pullback estimates} with $n\leftarrow n-|b|=|c|$ and $l\leftarrow |b|$, we obtain:
\begin{align*}
|D^c(\phi_t^*(\alpha_i))|
&\leq C R_t^{|b|+j}\mu_{\frac{5}{2}(|c|+d),5(|c|+d)}\fnorm{\alpha}_{n,j,r}\\
&\leq C R_t^{j}\mu_{\frac{5}{2}(|c|+d),5(|c|+d)+|b|}\fnorm{\alpha}_{n,j,r}.
\end{align*}
When we combine these inequalities, we first note that:
\begin{align*}
\mu_{2|b|+\frac{1}{2}, 3|b|+4}\mu_{\frac{5}{2}(|c|+d),5(|c|+d)+|b|}&\leq 
C\mu_{2n+\frac{1}{2}(|c|+5d+1),4n+|c|+5d+4}\\
&\leq
\mu_{\frac{1}{2}(5(n+d)+1),5(n+d)+4}.
\end{align*}
Thus, we obtain:
\begin{equation}
|D^a(h'_t(\alpha))|\leq C\,\epsilon_t\, R_t^j\mu_{\frac{1}{2}p,p+3}\fnorm{\alpha}_{n,j,r},
\end{equation}
where $p=5(n+d)+1$. Next, note that: 
\[\mu_{\frac{p}{2},p+3}=\sum_{j=0}^pt^{\frac{p-j}{2}}R^{p-j+3}\leq R^3(1+t^{\frac{1}{2}}R)^{p}\leq C R^3(1+tR^2)^{\frac{p}{2}}, \]
where we have use that $(1+t^{\frac{1}{2}}R)\leq \sqrt{2}(1+tR^2)^{\frac{1}{2}}$. Using this and Lemma \ref{lemma: finite integrals} with $q=\frac{j}{2}\geq 1$, we obtain \eqref{important_equation}.

For $j=5(n+d+1)$, \eqref{important_equation} gives:
\begin{align*}
|D^a(h'_t(\alpha))|\leq C\, \frac{R^{j+3}}{(1+tR^2)^2}\fnorm{\alpha}_{n,j,r}\leq C\, r^{j-1}t^{-2}\fnorm{\alpha}_{n,j,r}.
\end{align*}
The right-hand side is integrable over $t\in [1,\infty)$. Therefore by Appendix \ref{subsubsection_limits}, $h_t(\alpha)$ converges as $t\to\infty$ in all $C^n$-norms on $\overline{B}_r$ to a smooth form $h_{\Sk}(\alpha)$. This form is flat at the origin, because flatness is $C^{\infty}$-closed. 

Next, for $j=5(n+d+1)+k$, \eqref{important_equation} gives:
\begin{align*}
|D^a(h'_t(\alpha))|&\leq C\, \frac{R^{5(n+d+1)+3+k}}{(1+tR^2)^{2+\frac{k}{2}}}\fnorm{\alpha}_{n,5(n+d+1)+k,r}\\
&\leq C\,r^{5(n+d+1)+1} \frac{R^{2+k}}{(1+tR^2)^{2}}\fnorm{\alpha}_{n,5(n+d+1)+k,r}.
\end{align*}
Since $\int_0^{\infty}\frac{R^{2}\dif t}{(1+tR^2)^{2}}=1$, we obtain:
\begin{align*}
|D^a(h_t(\alpha))|&=
\Big|\int_0^t D^a(h'_s(\alpha))\dif s\Big|\\
&\leq C\, r^{5(n+d+1)+1}R^k\fnorm{\alpha}_{n,5(n+d+1)+k,r}.
\end{align*}
And so:
\begin{align*}
\frac{|D^a(h_t(\alpha))|}{R^k}&\leq C\, r^{5(n+d+1)+1}\fnorm{\alpha}_{n,5(n+d+1)+k,r}.
\end{align*}
This implies that:
\begin{align*}
\fnorm{h_t(\alpha)}_{n,k,r}
\leq C\, (1+r)^{5(n+d+1)+1}\fnorm{\alpha}_{n,5(n+d+1)+k,r}.
\end{align*}
Taking $\lim_{t\to\infty}$ on the left implies that $h_{\Sk}$ satisfies property SLB for $(0,5,35)$.
\section{Flat Poisson cohomology of $\spl$}\label{section: flat PC}

In this section we compute the flat Poisson cohomology groups of $(\spl,\pi)$ and we build cochain homotopy operators for the flat Poisson complex which satisfy property SLB (see Theorem \ref{main theorem}).

The section is structured as follows. In Subsection \ref{section: Poisson cohomology in terms of foliated cohomology} we build an isomorphism which describes the flat Poisson complex in terms of the flat foliated complex and the transverse variation of the leafwise symplectic form (see Proposition \ref{proposition:PCH as forms}). For this, we use the description of the Poisson complex as the de Rham complex of a Lie algebroid, and the spectral sequence for computing cohomology of regular Lie algebroids. In Subsection \ref{section: perturbation Lemma} we use the Perturbation Lemma and the explicit description of flat foliated cohomology to construct cochain homotopies for the aforementioned complex (see Proposition \ref{proposition: homotopy}). Combining these results, we obtain a description of the flat Poisson cohomology in Theorem \ref{main theorem}. In Subsection \ref{section: algebra} we explain the algebraic structure of Poisson cohomology.

\subsection{Poisson cohomology in terms of foliated cohomology}\label{section: Poisson cohomology in terms of foliated cohomology}

In this subsection we describe the flat Poisson complex of $\spl$ in terms of the flat foliated complex of $\F$. For this, we use the spectral sequence associated to a regular Poisson manifold, which has been introduced in \cite{Vaisman} and \cite{KV1} (see also \cite[Section 2.4]{DZ}). Here, we use the more general setting of Lie algebroids of \cite{mack05}. We can adapt these techniques to our non-regular case by working with sections that are flat at the singularity,

Recall that the linear Poisson structure $\pi$ is regular away from the origin with underlying foliation $\F$. For simplicity, let us denote:
\[\spl^{\times}:=\spl\backslash \{0\}\]
The vector fields $V_1$ and $V_2$ on $\spl^{\times}$, defined in \eqref{eq: v}, form a basis of the normal bundle $\nu$ to $\F$. Since these vector fields are projectable with respect to $f$ (see \eqref{eq:v transversality}), it follows that they preserve the foliation:
\begin{align}\label{eq: flat}
    [V_i,Y]\in \Gamma(T\F|_{\spl^{\times}}), \quad \forall\, Y\in \Gamma(T\F|_{\spl^{\times}}).
\end{align}

We make use of the extension $\tilde{\omega}_1$ of the leafwise symplectic form on $\F$ from Lemma \ref{lemma: extension omega}. This is uniquely determined such that its kernel is spanned by $V_1$ and $V_2$. For $i=1,2$, define:
\begin{align}\label{eq: gammai}
    \gamma_i:=i_{V_i}\dif \tilde{\omega}_1\in \Omega^2(\spl^{\times}).
\end{align}
The restriction $\gamma_i|_{\F}$, for $i=1,2$, is the \emph{variation of the leafwise symplectic form} $\omega_1$ in the direction of $V_i$ (as foliated forms):
\[\gamma_i\big|_{\F}=\frac{\dif}{\dif t}\Big|_{t=0}(\varphi_{V_i}^t)^*\omega_1.\]
In particular, it follows that $\gamma_i|_{\F}\in \Omega^2(\F)$ are closed foliated 2-forms.

Recall that the singularity of $\tilde{\omega}_1$ at the origin is very ``mild", and the same holds for $\gamma_i$. Therefore the product with $\gamma_i$ gives a well-defined map:
\[e_{\gamma_i}:\CC^{p}\to \CC^{p+2},\quad \eta\mapsto \gamma_i\wedge \eta.\]

\begin{proposition}\label{proposition:PCH as forms}
There is an isomorphism of complexes
\[\tilde{a}:(\mathfrak{X}_0^{\bullet}(\spl),\dif_{\pi})\diffto (\bigoplus_{p+q=\bullet}\CC^{p}\otimes \wedge^q\R^2, \dif +\delta),\]
where $\dif$ acts only on the first component $\CC^{\bullet}$ as the exterior derivative and
\begin{align*}
\delta:\CC^{p}\otimes \wedge^q\R^2&\to \CC^{p+2}\otimes \wedge^{q-1}\R^2,\\
\delta(\eta\otimes e_{i})&=(-1)^p\gamma_i\wedge \eta,\quad\quad i=1,2,\\
\delta(\eta\otimes e_{1}\wedge e_2)&=(-1)^p\big(\gamma_1\wedge\eta\otimes e_2- \gamma_2\wedge\eta\otimes e_1\big),
\end{align*}
where $e_{1},e_2$ is the standard basis of $\R^2$. Moreover, the operators $\tilde{a}$, $\tilde{a}^{-1}$ and $\delta$ satisfy property SLB for $(0,0,2)$, $(0,0,4)$ and $(0,0,3)$, respectively. 
\end{proposition}

The rest of this subsection is devoted to the proof of this proposition.

\subsubsection{Cohomology of regular abelian Lie algebroids}

Proposition \ref{proposition:PCH as forms} can be understood as giving an explicit description of the first page of a natural spectral sequence associated to Poisson cohomology, which was introduced for regular Poisson manifolds already in \cite{Vaisman}. The natural framework for this are Lie algebroids (see \cite{mack05}), and, for clarity, we place our discussion in this more general setting. However, we will use these results only for the cotangent Lie algebroid of a regular Poisson manifold. 

As a reference to the theory of Lie algebroids, see \cite{mack05}. Recall that any Lie algebroid $(A,[\cdot,\cdot],\sharp)$ has an associated de Rham complex
\[(\Omega^{\bullet}(A):=\Gamma(\wedge^{\bullet}A^*),\dif_A).\]
Consider a Lie algebroid $(A,[\cdot,\cdot],\sharp)$ over $M$ with the following properties:
\begin{itemize}
    \item $A$ is \textbf{regular}, i.e., the anchor $\sharp$ has constant rank. Then $M$ has a regular foliation $\F$ with $T\F=\im \sharp \subset TM$, the kernel $K:=\ker \sharp$ is a Lie subalgebroid, and we have a short exact sequence of Lie algebroids
    \begin{align}\label{eq: lad ses}
        0 \to K \xrightarrow{j} A \xrightarrow{\sharp} T\F \to 0
    \end{align}
    \item $K$ is an \textbf{abelian} Lie algebroid. Therefore, the bracket on $A$ induce a representation of $T\F$ on $K$.
\end{itemize}

To such a Lie algebroid one associates the \textbf{extension class} \[[\gamma] \in H^2(\F,K),\] 
which measures the failure of the short exact sequence \eqref{eq: lad ses} to split as a sequence of Lie algebroids. Using a vector bundle splitting $\sigma: T\F \to A$ of \eqref{eq: lad ses}, an explicit representative of the extension class is given by:  
\[ \gamma(Y_1,Y_2):= [\sigma(Y_1),\sigma(Y_2)] -\sigma([Y_1,Y_2]) \]
where $Y_1,Y_2\in \Gamma(T\F)$. The class $[\gamma]$ is independent of the choice of the splitting. For such a representative, we have an induced map:
\begin{align}
&\qquad \qquad \qquad i_{\gamma}: \Omega^p(\F,\wedge^qK^*) \to  \Omega^{p+2}(\F,\wedge^{q-1}K^*) \nonumber \\
  i_{\gamma}(\beta)&\,(Y_1 , \dots ,Y_{p+2})(V_1, \dots ,V_{q-1}) :=\label{eq: contraction gamma}\\  &\sum_{i<j}(-1)^{i+j}\beta(Y_1, \dots, \hat{Y}_i,\dots ,\hat{Y}_j,\dots ,Y_{p+2})(\gamma(Y_i,Y_j),V_1,\dots, V_{q-1})\nonumber
\end{align}
where $V_i \in \Gamma(K)$, $Y_j\in \Gamma(T\F)$. We have the following: 

\begin{proposition}\label{proposition: ladcx}
A splitting $\sigma:T\F \to A$ of \eqref{eq: lad ses} induces an isomorphism
\begin{align*}
    (\Omega^{\bullet}(A), \dif_A)\simeq \big(\bigoplus_{p+q=\bullet}\Omega^{p}(\F,\wedge^{q} K^*), \delta_1+\delta_2\big),
\end{align*}
where $\delta_1$ is the standard differential on foliated forms with values in $\wedge^{q}K^*$
\begin{align*}
    \delta_1=\dif_{\F}:\Omega^{p}(\F,\wedge^{q}K^*)\to \Omega^{p+1}(\F,\wedge^{q}K^*)
\end{align*}
and $\delta_2$ is the cup product with the representative of the extension class: 
\begin{align*}
\delta_2=(-1)^{p}i_{\gamma}:\Omega^{p}(\F,\wedge^{q}K^*)\to \Omega^{p+2}(\F,\wedge^{q-1}K^*).
\end{align*}
\end{proposition}

\begin{proof}
We follow \cite[Section 7.4]{mack05} where most of the statement is proven. A splitting $\sigma : T\F \to A$ of \eqref{eq: lad ses} induces surjective maps
\begin{align}\label{eq: surjection ladch}
&\qquad \qquad a^{p,q}:\Omega^{p+q}(A) \to \Omega^{p}(\F,\wedge^q K^*)\\
a^{p,q}(\alpha)&\,(Y_1, \dots ,Y_p) (V_1, \dots V_q):= \alpha(\sigma(Y_1),\dots , \sigma(Y_p),V_1,\dots, V_q)\nonumber
\end{align}
for $V_i \in \Gamma(K)$ and $Y_j\in \Gamma(T\F)$. Let $\kappa = \id-\sigma \circ \sharp:A\to K$ be the induced projection. A right inverse for $a^{p,q}$ is given by: 
\begin{align}\label{eq: ri ladch}
    &\qquad \qquad b^{p,q}: \Omega^{p}(\F,\wedge^q K^*)\to \Omega^{p+q}_p(A)\\
    b^{p,q}(\beta)&\,(X_1, \dots ,X_{p+q}):= \nonumber \\ 
    &\sum_{\tau\in S_{p,q}}\sign(\tau)\beta(\sharp(X_{\tau_{1}}),\dots,\sharp(X_{\tau_{p}}))(\kappa(X_{\tau_{p+1}}),\dots,\kappa(X_{\tau_{p+q}})),\nonumber
\end{align}
where the sum is over all $(p,q)$ shuffles $\tau$. The isomorphism from the statement is $\oplus_{p+q=\bullet}a^{p,q}$, with inverse  $\oplus_{p+q=\bullet}b^{p,q}$.

We calculate the components of the differential under this isomorphism:
\begin{align*}
    \delta_{i}:= a^{p+i,q+1-i}\circ \dif_A \circ b^{p,q}, \quad i\in \N_0.
\end{align*}
For $i=0$ and $i=1$ these are given in \cite[Proposition 7.4.3 \& 7.4.4]{mack05}, namely, $\delta_{0}=0$ because $K$ is assumed to be abelian and the different sign for $\delta_1$ comes from the fact that we multiplied both $a^{p,q}$ and $b^{p,q}$ with a factor of $(-1)^{pq}$. For $\delta_{2}$ we use a proof given for $q=1$ in \cite[Theorem 7.4.11]{mack05} for foliated cohomology. We adapt the computation there to general $q$ and foliated forms as follows. Let $\beta\in \Omega^p(\F,\wedge^qK^*)$. We need to prove that
\begin{align*}
    (a^{p+2,q-1}\circ \dif_A\circ b^{p,q})(\beta)=&\ (-1)^{p}i_{\gamma}(\beta)
\end{align*}
Let $V_i\in \Gamma(K)$ and $Y_j\in \Gamma(T\F)$. For the computation we use the notation
\begin{equation*}
    \begin{array}{rlcrl}
    \mathcal{V} :=& V_1,\dots ,V_{q-1}&\ & 
         \mathcal{Y}:=& Y_1,\dots,Y_{p+2}   \\
         \mathcal{V}_i :=& V_1,\dots,\hat{V}_i,\dots ,V_{q-1}&\ & 
         \mathcal{Y}_i:=& Y_1,\dots,\hat{Y}_i,\dots ,Y_{p+2}   \\
         \mathcal{V}_{ij}:=& V_1,\dots,\hat{V}_i,\dots,\hat{V}_j,\dots ,V_{q-1}
         & &    \mathcal{Y}_{ij}:=&
         Y_1,\dots,\hat{Y}_i,\dots, \hat{Y}_j,\dots ,Y_{p+2}
    \end{array}
\end{equation*}
where the hat stands for omitting the element and by $\sigma(\mathcal{Y}_i)$ we mean $\sigma$ applied to each element, and similarly for $\sigma(\mathcal{Y}_{ij})$. Then we compute that

\begin{align*}
    (a^{p+2,q-1}\circ &\, \dif_A\circ  b^{p,q})(\beta)(\mathcal{Y})(\mathcal{V}) =(\dif_A\circ b^{p,q})(\beta)(\sigma(\mathcal{Y}),\mathcal{V})\\
   =&\sum_{1\le i<j\le p+2}(-1)^{i+j}(b^{p,q}(\beta))([\sigma(Y_i),\sigma(Y_j)],\sigma(\mathcal{Y}_{ij}),\mathcal{V})\\
    &+\sum_{\substack{1\le i\le q-1\\ 1\le j\le p+2}}(-1)^{p+1+i+j}(b^{p,q}(\beta))([\sigma(Y_i),V_j],\sigma(\mathcal{Y}_i),\mathcal{V}_j)\\
    &+\sum_{1\le i<j\le q-1}(-1)^{i+j}(b^{p,q}(\beta))([V_i,V_j],\sigma(\mathcal{Y}),\mathcal{V}_{ij})
    \\
    &+\sum_{1\le i\le p+2}(-1)^{i+1}\Lie_{Y_i}((b^{p,q}(\beta))(\sigma(\mathcal{Y}_{i}),\mathcal{V}))\\
  =&\sum_{1\le i<j\le p+2}(-1)^{i+j+p}\beta(\mathcal{Y}_{ij})(\gamma(Y_i, Y_j),\mathcal{V})\\
    =&\, (-1)^{p} i_{\gamma}(\beta)(\mathcal{Y})(\mathcal{V})
\end{align*}
Note that only the first sum yields non-zero contributions to the third equality as all the other terms do not have enough sections of $K$ to be non-zero, and we have $\kappa\circ \sigma =0$. By a similar computation $\delta_{i}=0$, for $3\le i$.
\end{proof}

We obtain a slight improvement of results in \cite[Section 7.4]{mack05}, which is classical for Lie algebras (see e.g.\ \cite[Theorem 8]{Hoch1953}):

\begin{corollary}\label{corollary: ladch}
For a regular Lie algebroid $K\hookrightarrow A\twoheadrightarrow T\F$ with abelian isotropy $K$, there is a natural spectral sequence which converges to the Lie algebroid cohomology of $A$:
\[ H^p(\F,\wedge^q K^*)\Rightarrow H^{p+q}(A).\]
Moreover, the differential on the second page is given by the cup product with the extension class:
\[(-1)^{p}i_{[\gamma]}:H^p(\F,\wedge^q K^*)\to H^{p+2}(\F,\wedge^{q-1} K^*).\]
\end{corollary}
\begin{proof}
The natural, convergent spectral sequence was obtained in \cite[Theorem 7.4.6]{mack05}; it is associated with the filtration 
\begin{align}\label{eq: filtration}
    \Omega^{p+q}_{p}(A):= \{\beta\in \Omega^{p+q}(A)| i_{V_0}\dots i_{V_q}\beta =0 \text{ for all } V_j\in \Gamma(K)\}. 
\end{align}
Under the isomorphism from  Proposition \ref{proposition: ladcx}, this filtration corresponds to 
\[\bigoplus_{0\leq i\leq q} \Omega^{p+i}(\F,\wedge^{q-i}K^*).\]
The explicit formula for the differential on the second page follows from the proposition. 
\end{proof}

\subsubsection{Poisson cohomology via Lie algebroids}
Given a Poisson manifold $(M,\pi)$ one has a naturally associated Lie algebroid structure on the cotangent bundle $T^*M$ with the anchor map $\pi^{\sharp}:T^*M \to TM$ and the bracket given by
\[ [\alpha,\beta]_{\pi}:= \Lie _{\pi^{\sharp}(\alpha)}(\beta)-\Lie _{\pi^{\sharp}(\beta)}(\alpha)-\dif \pi(\alpha,\beta).\]
We denote this Lie algebroid by $T^*_{\pi}M$. The complex computing Lie algebroid cohomology of $T^*_{\pi}M$ can be identified with the complex computing the Poisson cohomology of $(M,\pi)$:
\begin{align*}
    (\Omega^{\bullet}(T^*_{\pi}M),\dif) = (\mathfrak{X}^{\bullet}(M),\dif_{\pi}).
\end{align*}

From now on we will assume that $(M,\pi)$ is \textbf{regular}. Then we are in the setting of the previous subsection, i.e., the Lie algebroid $T^*_{\pi}M$ fits into the short exact sequence
\begin{align}\label{eq: ses poisson lad}
    0\to \nu^* \hookrightarrow T^*_{\pi}M\twoheadrightarrow T\F\to 0
\end{align}
where $T\F=\im \pi^{\sharp}$ and $\nu^*$ is the conormal bundle of $\F$, which is an abelian ideal. The leafwise symplectic form will be denoted by: 
\[\omega\in \Omega^2(\F).\]

A splitting $\sigma:T\F\to T^*_{\pi}M$ of \eqref{eq: ses poisson lad} induces an extension $\tilde{\omega}\in \Omega^2(M)$ of $\omega$, uniquely determined by $\ker\tilde{\omega}=\ker (\sigma^*:TM\to T^*\F)$. The representative $\gamma$ of the extension class in $H^2(\F,\nu^*)$ of \eqref{eq: ses poisson lad} corresponding to $\sigma$ is then:
\begin{align*}
    \gamma(Y_1,Y_2)=i_{Y_2}i_{Y_1} \dif\tilde{\omega}\in \Gamma(\nu^*),
\end{align*}
for $Y_1,Y_2\in \Gamma(T\F)$. This follows by a standard calculation using the formula of the Lie bracket on $T^*_{\pi}M$ and that $ \sigma(v)=i_v\tilde{\omega}$.

\smallskip

We apply this discussion to the linear Poisson structure $\pi$ restricted to 
\[\spl^{\times}:=\spl\backslash\{0\}.\] 
The extension $\tilde{\omega}_1$ of the leafwise symplectic structure from Lemma \ref{lemma: extension omega} corresponds to the splitting $\sigma$ whose image annihilates the vector fields $V_1$ and $V_2$ defined in \eqref{eq: v}. We obtain linear maps
\begin{align*}
    a^{p,q}:\mathfrak{X}^{p+q}(\spl^{\times}) \to \Omega^p(\F|_{\spl^{\times}}, \wedge^q\nu)
\end{align*}
as described in \eqref{eq: surjection ladch} with right inverses given by
\begin{align*}
    b^{p,q}:\Omega^p(\F|_{\spl^{\times}}, \wedge^q\nu)\to \mathfrak{X}^{p+q}(\spl^{\times})
\end{align*}
as described in \eqref{eq: ri ladch}.

The normal vector fields $V_1$ and $V_2$ induce a flat frame of the normal bundle, and so give an isomorphism $\nu\simeq \R^2\times{\spl^{\times}}$ of $T\F$-representations.

\begin{lemma}\label{lemma: maps mpq} There are linear isomorphisms:
\[m^{p,q}:\Omega^p(\F|_{\spl^{\times}}, \wedge^q\nu)\simeq
e_{\varphi}(\Omega^p(\spl^{\times}))\otimes \wedge^q\R^2,
\]
under which the operators $\delta_1=\dif_{\F}$ and $\delta_2=(-1)^{p}i_{\gamma}$ correspond to the operators $\dif$ and $\delta$ from Proposition \ref{proposition:PCH as forms}.
\end{lemma}

\begin{proof}
Any form $\eta\in \Omega^{p}(\F|_{\spl^{\times}}, \wedge^{\bullet}\nu)$ can be uniquely written as
\[\eta=\eta_0+\eta_1\otimes [V_1] +\eta_2\otimes[V_2]+\eta_{3}\otimes[V_1\wedge V_2],\quad \eta_i\in \Omega^{p}(\F|_{\spl^{\times}}).\]
The isomorphism $m=\oplus_{p,q} m^{p,q}$ sends such a form to:
\[m(\eta)=e_{\varphi}(\eta_0)+e_{\varphi}(\eta_1)\otimes e_1+e_{\varphi}(\eta_2)\otimes e_2+ e_{\varphi}(\eta_3)\otimes e_1\wedge e_2,\]
where $e_1,e_2$ is the standard basis of $\R^2$. That $\dif_{\F}$ corresponds to $\dif$ follows because $\dif_{\F}[V_i]=0$ (see \eqref{eq: flat}) and because $\dif \varphi=0$. Using \eqref{eq: contraction gamma} we obtain 
\begin{align*}
\delta_2(\eta)&=(-1)^p(i_{\gamma}\eta_0+i_{\gamma}\eta_1\otimes [V_1]+i_{\gamma}\eta_2\otimes [V_2]+i_{\gamma}\eta_3\otimes[V_1\wedge V_2])\\
&=(-1)^p(e_{\gamma_1}(\eta_1)+e_{\gamma_2}(\eta_2)+e_{\gamma_1}(\eta_3)\otimes[V_2]-e_{\gamma_2}(\eta_3)\otimes[V_1])
\end{align*}
since $\gamma$ can be written as
\[\gamma=\gamma_1|_{\F}\otimes \dif f_1+\gamma_2|_{\F}\otimes \dif f_2.\]
This implies the statement.
\end{proof}

Due to the mild singularities of $V_1$,  $V_2$ and $\tilde{\omega}_1$ at the origin, the maps $m^{p,q}\circ a^{p,q}$ induce well-defined, surjective maps between the complexes of flat multivector fields and flat foliated forms:
\begin{align}\label{eq: tilde a}
    \tilde{a}^{p,q}:\mathfrak{X}^{p+q}_{0}(\spl) \to \CC^p\otimes \wedge^q \R^2.
\end{align}
These maps have right inverses induced by $b^{p,q}\circ (m^{p,q})^{-1}$ which we denote
\begin{align}\label{eq: tilde b}
    \tilde{b}^{p,q}: \CC^p\otimes \wedge^q\R^2\to \mathfrak{X}^{p+q}_{0}(\spl).
\end{align}

Putting these maps together into \[\tilde{a}:=\oplus_{p,q}\tilde{a}^{p,q}\quad \textrm{and} \quad \tilde{b}:=\oplus_{p,q}\tilde{b}^{p,q}\] we obtain the linear isomorphism from Proposition \ref{proposition:PCH as forms} and its inverse. That these maps intertwine the Poisson differential and the differential $\dif+\delta$ follows from Proposition \ref{proposition: ladcx} and Lemma \ref{lemma: maps mpq}. To conclude the proof of Proposition \ref{proposition:PCH as forms} we need to show that the map $\tilde{a}$, $\tilde{b}$ and $\delta$ satisfy the claimed SLB property. For this, let us note that these operators are induced by vector bundle maps with a singularity at the origin. For example, $\tilde{a}^{1,1}$ is induced by the vector bundle map:
\begin{align*}
&A^{1,1}:\wedge^{2}T(\spl^{\times})\to \wedge^{3} T^*(\spl^{\times})\otimes \R^2,\\
&A^{1,1}(w)=-\sum_{i=1}^2\varphi\wedge \tilde{\omega}_1^{\flat}(i_{\dif f_i}w)\otimes e_i.
\end{align*}
Similarly, $\tilde{a}^{4,0}$ is induced by the vector bundle map:
\begin{align*}
&A^{4,0}:\wedge^{4}T(\spl^{\times})\to \wedge^{6} T^*(\spl^{\times}),\\
&A^{4,0}(w)=\varphi\wedge (\otimes^4 \tilde{\omega}_1^{\flat})(w).
\end{align*}
The maps $\tilde{b}^{p,q}$ are restrictions of singular vector bundle maps. For example, $\tilde{b}^{2,1}$ is induced by the restriction of the map:
\begin{align*}
&B^{2,1}:\wedge^{4} T^*(\spl^{\times})\otimes \R^2\to \wedge^{3}T(\spl^{\times}),\\
&B^{2,1}\Big(\sum_{i=1}^2\alpha_i\otimes e_i\Big)=
\sum_{i=1}^2(\otimes^2\pi^{\sharp})(i_{V_1\wedge V_2}\alpha_i)\wedge V_i.
\end{align*}
Similarly, $\tilde{b}^{0,2}$ comes from the map:
\begin{align*}
&B^{0,2}:\wedge^{2} T^*(\spl^{\times})\otimes \wedge^2\R^2\to \wedge^{2}T(\spl^{\times}),\\
&B^{0,2}\big(\beta\otimes e_1\wedge e_2\big)=
i_{V_1\wedge V_2}(\beta)\, V_1\wedge V_2.
\end{align*}
A similar description holds also for $\delta$. These vector bundle maps fit in the setting of the following lemma.

\begin{lemma}\label{lemma:singular:vb:maps} Let $V$ and $W$ be two vector bundles over an inner product space $H$. Consider a vector bundle map with a singularity at the origin, i.e., a vector bundle map: 
\[\Phi:V|_{H\backslash\{0\}}\to W|_{H\backslash\{0\}}.\]
Assume that the singularity is such that, for some $l\geq 0$, 
$R^{l}\Phi$ is smooth and that it vanishes at the origin up to order $d\geq 0$, where $R$ denotes the induced norm on $H$. Then $\Phi$ induces a linear operator between flat sections:
\[\Phi_{*}:\Gamma_0(V) \to \Gamma_0(W),\]
which satisfies property SLB for $(0,0,\mathrm{max}(l-d,0))$.
\end{lemma}

\begin{proof}
Without loss of generality we may assume that the vector bundles $V$ and $W$ are trivial. Therefore, $\Phi$ is described by a matrix of the form
\begin{align*}
    \Phi = \begin{pmatrix}
    \frac{1}{R^{l}}\phi_{ij}
    \end{pmatrix} 
\end{align*}
where $\phi_{ij}\in C^{\infty}(H)$ vanish up to order $d$ at the origin. It is enough to show that for $\phi\in C^{\infty}(H)$ vanishing at the origin up to order $d$, we have: 
\begin{align*}
    \fnorm{\frac{1}{R^{l}}\phi\, g}_{n,k,r}\le C\fnorm{g}_{n,\mathrm{max}(0,k+l-d),r}
\end{align*}
for all $g\in C^{\infty}_0(H)$, $n,k\in \N_0$, $r>0$ and a constant $C$. First, by the equivalence of norms from Lemma C.1 in Part II of the paper, we have that:
\begin{align*}
    \fnorm{\frac{1}{R^{l}}\phi\, g}_{n,k,r}\le C\fnorm{\phi\, g}_{n,k+l,r}.
\end{align*}

By the general Leibniz rule we obtain for $a\in \N_0^{\dim(V)}$ with $|a|=n$ that
\begin{align*}
    D^a (\phi\, g) & =\sum_{a^1+a^2=a} D^{a^1}\phi\, D^{a^2}g.
\end{align*}
Next, Taylor's formula with integral remainder for $|a^2|<n$ yields: 
\[ D^{a^2}g(x)=\sum_{|b|=|a^1|}x^{b}\int_{0}^1|a^1|(1-s)^{|a^1|-1}(D^{b+a^2} g)(sx)\dif s.\]
Combining these yields for $x\in B_r\subset H$ and any $m=k+l\in \N_0$ the estimate
\begin{align*}
    \frac{|D^a \phi g|}{R^{m}} (x) &\le \sum_{a^1+a^2=a} \frac{1}{R^{m}}|D^{a^1}\phi||D^{a^2}g |(x)\\
    &\le C\Big(\sum_{0<|b|= n- |a^2|}\frac{R^{\max(|b|,d)}}{R^{m}}\int_{0}^1(1-s)^{|b|-1}|D^{b+a^2} g|(sx)\dif s\\
    &\qquad \quad + \frac{R^{d}}{R^{m}}|D^ag|(x)\Big)\\
    &\le C\Big(\sum_{0<|b|= n- |a^2|}\int_{0}^1\frac{s^{\max(0,m-d)}(1-s)^{|b|-1}}{(sR)^{\max(0,m-d)}}|D^{b+a^2} g|(sx)\dif s\\
    &\qquad \quad + \frac{1}{R^{\max(0,m-d)}}|D^ag|(x)\Big)\\
    &\le C \,\fnorm{g}_{n,\max(0,m-d),r}. \qedhere
\end{align*}
\end{proof}

The lemma implies that the maps $\tilde{a}$, $\tilde{b}$ and $\delta$ satisfy the claimed SLB properties. This follows by analyzing the singularities of our vector bundle maps, for which we obtain:
\begin{itemize}
    \item[-] for $A^{p,q}$: $l=2p$, $d=p+q+2$; hence $l-d\leq 2$;
    \item[-] for $B^{p,q}$: $l=4+2q$, $d=p+q+2$; hence $l-d\leq 4$;
    \item[-] for the map inducing $\delta$: $l=6$ and $d=3$; hence $l-d=3$.
\end{itemize}
This concludes the proof of Proposition \ref{proposition:PCH as forms}.

\subsection{Cochain homotopies}\label{section: perturbation Lemma}

By Proposition \ref{proposition:PCH as forms}, the complex 
\begin{equation}\label{eq:beautiful:complex}
    \Big(\bigoplus_{p+q=\bullet}\CC^{p}\otimes \wedge^{q}\R^2,\dif +\delta\Big)
\end{equation}
computes the flat Poisson cohomology of $(\spl,\pi)$. This complex admits the filtration by the $p$-degree, as in Corollary \ref{corollary: ladch}. The second page of the associated spectral sequence is given by:
\begin{equation*}
    \begin{tikzpicture}[baseline= (a).base]
    \node[scale=1] (a) at (0,0){
    \begin{tikzcd}[row sep=small]
    H^0_0(\F)\otimes\wedge^2\R^2\arrow{rrd}{\phantom{1}i_{[\gamma]}}&0&H^2_0(\F)\otimes\wedge^{2}\R^2 \\
    H^0_0(\F)\otimes\wedge^1\R^2\arrow{rrd}{\phantom{1}i_{[\gamma]}}&0&H^2_0(\F)\otimes\wedge^{1}\R^2\\
    H^0_0(\F)\otimes\wedge^{0}\R^2 &0&H^2_0(\F)\otimes\wedge^{0}\R^2
    \end{tikzcd}
    };
    \end{tikzpicture}
\end{equation*}
where we used that $H^{p}_0(\F)=0$, for $p\notin\{0,2\}$. In particular, all differentials on higher pages vanish, and so the cohomology is determined by the third page. The kernel and cokernel of the maps $i_{[\gamma]}$ will be calculated explicitly in the Subsections \ref{section: cohomology second page} and \ref{section: HE data}. This will provide the complete description of the Poisson cohomology in Subsection \ref{section: flat pch}. The construction of homotopy operators will be obtained by applying the Perturbation Lemma, which is done in the first subsection.

\subsubsection{The Perturbation Lemma}

For the proof of Theorem \ref{theorem: nondegenerate} it does not suffice to calculate the Poisson cohomology, but we need explicit cochain homotopies satisfying estimates as in Definition \ref{definition: property H}. In this subsection we show how the deformation retraction $(p_{\F},h_{\F})$ of $(\CC^{\bullet},\dif)$ induces a deformation retraction of the complex \eqref{eq:beautiful:complex}. For this we use the Perturbation Lemma following \cite{Cr04} and the terminology therein, which we first recall. 

\begin{definition}
A \textbf{homotopy equivalence data} (HE data):
\begin{align}\label{eq: he data}
    (B,\mathrm{b} )\stackrel[\iota]{p }{\leftrightarrows} (C,\dif),h
\end{align}
consists of two complexes $(B,\mathrm{b})$ and $(C,\dif)$ together with two quasi-isomorphisms $p$ and $\iota$, and a cochain homotopy $h$ between $\iota \circ p$ and the identity, i.e.,
\begin{align*}
    \iota \circ p -\id = h\circ \dif +\dif \circ h
\end{align*}
\end{definition}
\begin{remark}
Note that any deformation retraction (see Definition \ref{definition: dr}) defines an HE data. In particular, by Proposition \ref{proposition: property H fol} we have the HE data
\begin{align*}
    (p_{\F}(\CC),0 )\stackrel[\iota]{p_{\F} }{\leftrightarrows} (\CC,\dif),h_{\F}
\end{align*}
\end{remark}
Given an HE data one considers perturbations of this data as follows:
\begin{definition}
A \textbf{perturbation} $\delta$ of an HE data as in \eqref{eq: he data} is a map on $C$ of the same degree as $\dif$ such that $\dif+\delta$ is a differential:
\begin{align*}
    (\dif+\delta)^2=0.
\end{align*}
The perturbation $\delta$ is called \textbf{small} if $(\id - \delta h)$ is invertible. Then we set
\begin{align*}
    A:= (\id -\delta \circ h)^{-1}\circ \delta.
\end{align*}
The \textbf{perturbation data} of the small perturbation $\delta$ is the data: \begin{align}\label{eq: perturbation data}
    (B,\tilde{\mathrm{b}})\stackrel[\tilde{\iota}]{ \tilde{p}}{\leftrightarrows} (C,\dif +\delta),\tilde{h}
\end{align}
\begin{equation*}
    \begin{array}{rclcrcl}
         \tilde{\iota} &=& (\id +h\circ A)\circ \iota, &\ \ \ \ \ & \tilde{p}&=& p\circ (\id +A\circ h), \\
         \tilde{h} &=& h\circ (\id +A\circ h),& &\tilde{\mathrm{b}} &=&\mathrm{b} +p\circ A\circ \iota.
    \end{array}
\end{equation*}
\end{definition}
The main statement in \cite{Cr04} states:
\begin{lemma}\label{lemma: perturbation lemma}
If $\delta$ is a small perturbation of an HE data as in \eqref{eq: he data}, then the associated perturbation data \eqref{eq: perturbation data} is an HE data.
\end{lemma}

We apply the Perturbation Lemma to our complex $(\CC\otimes \wedge \R^2, \dif+\delta)$, which we regard as a perturbation of the differential $\dif$, admitting the deformation retraction $(p_{\F},h_{\F})$. Here we extend all these operators from $\CC$ to $\CC\otimes \wedge \R^2$ by $- \otimes \id_{\wedge\R^2}$. These maps fit in the diagram: 
\begin{equation}\label{eq: complex diag}
    \begin{tikzpicture}[baseline= (a).base]
    \node[scale=.9] (a) at (0,0){
    \begin{tikzcd}
    \CC^0\otimes\wedge^2\R^2\arrow[bend left=5]{r}{\dif}\arrow{rrd}{\delta}&\CC^1\otimes\wedge^2\R^2\arrow[bend left=5]{r}{\dif}\arrow[bend left=5]{l}\arrow{rrd}{\delta}&\CC^2\otimes\wedge^2\R^2\arrow[bend left=5]{r}{\dif}\arrow[bend left=5]{l}&\CC^3\otimes\wedge^2\R^2\arrow[bend left=5]{l}{h_{\F}} \\
    \CC^0\otimes\wedge^1\R^2\arrow[bend left=5]{r}{\dif}\arrow{rrd}{\delta}&\CC^1\otimes\wedge^1\R^2\arrow[bend left=5]{r}\arrow[bend left=5]{l}\arrow{rrd}{\delta}&\CC^2\otimes\wedge^1\R^2\arrow[bend left=5]{r}\arrow[bend left=5]{l}&\CC^3\otimes\wedge^1\R^2 \arrow[bend left=5]{l}{h_{\F}} \\
    \CC^0\otimes\wedge^0\R^2\arrow[bend left=5]{r}{\dif}&\CC^1\otimes\wedge^0\R^2\arrow[bend left=5]{r}\arrow[bend left=5]{l}{h_{\F}}&\CC^2\otimes\wedge^0\R^2\arrow[bend left=5]{r}\arrow[bend left=5]{l}{h_{\F}}&\CC^3\otimes\wedge^0\R^2\arrow[bend left=5]{l}{h_{\F}}
    \end{tikzcd}
    };
    \end{tikzpicture}
\end{equation}

As an immediate consequence of the Perturbation Lemma we obtain:
\begin{corollary}\label{corollary: homotopy}
The complex $(\CC\otimes \wedge\R^2,\dif+\delta)$ admits the HE data
\begin{align*}
    (p_{\F}(\CC)\otimes \wedge\R^2,\tilde{\delta})\stackrel[\tilde{\iota}]{\tilde{p}}{\leftrightarrows}(\CC\otimes \wedge\R^2 ,\dif+ \delta),\tilde{h}
\end{align*}
where the maps $\tilde{\iota}$, $\tilde{p}$ and $\tilde{h}$ are given by
\begin{align*}
    \tilde{\iota} =&\ \sum_{i=0}^2 (h_{\F}\circ \delta)^i\circ \iota,\ \ \ \ \ \ \  \tilde{p} = p_{\F}+p_{\F}\circ \delta\circ h_{\F},\\
    \tilde{h}=&\ h_{\F}\circ \sum_{i=0}^2 (\delta\circ h_{\F})^i, \ \ \ \ \ \ \tilde{\delta} = p_{\F}\circ \delta.
\end{align*}
The maps $\tilde{\iota}$, $\tilde{p}$ and the homotopy $\tilde{h}$ satisfy property SLB for $(0,10,76)$, $(1,10,83)$ and $(0,15,111)$, respectively.
\end{corollary}
\begin{proof}
From Proposition \ref{proposition: property H fol} we know that the complex $(C^{\bullet},\dif)$ given by \begin{align}\label{eq: complex short}
    C^k := \underset{p+q=k}{\bigoplus} \CC^p\otimes \wedge ^q\R^2
\end{align}
admits the HE data 
\begin{align*}
    (p_{\F}(C),0)\stackrel[\iota]{p_{\F} }{\leftrightarrows}(C,\dif),h_{\F}
\end{align*}
Note that $\delta$ is a small perturbation of $\dif$ since $\delta \circ h_{\F}$ is nilpotent:
\begin{align*}
    (\delta \circ h_{\F})^3=0
\end{align*}
and hence
\begin{align*}
    (\id -\delta\circ h_{\F})^{-1}=\sum_{i=0}^2 (\delta\circ h_{\F})^i.
\end{align*}
Hence we can apply Lemma \ref{lemma: perturbation lemma}. The formulas for $\tilde{\iota}$, $\tilde{p}$, $\tilde{h}$ and $\tilde{\delta}$ follow now immediately, by looking at the $(p,q)$ degrees and using that $p_{\F}(\CC^{p})=0$, for $p\notin \{0,2\}$, by Proposition \ref{proposition:flat foliated ch}.

By Proposition \ref{proposition:PCH as forms}, $\delta$ satisfies property SLB for $(0,0,3)$ and, by Proposition \ref{proposition: property H fol} and the remark following it, $p_{\F}$ and $h_{\F}$ satisfy property SLB for $(1,5,40)$ and $(0,5,35)$, respectively. Therefore, by Lemma \ref{lemma: slb property}, the operators $\tilde{\iota}$, $\tilde{p}$ and $\tilde{h}$ satisfy the claimed SLB properties.
\end{proof}

\subsubsection{The differential $\tilde{\delta}$ and its partial inverses}\label{section: cohomology second page}

Corollary \ref{corollary: homotopy} reduces the calculation of the cohomology of the complex $(\CC\otimes \wedge\R^2,\dif +\delta)$ and the construction of ``good'' homotopy operators for it, to the similar questions for the complex $(p_{\F}(\CC)\otimes \wedge\R^2,\tilde{\delta})$:
\begin{equation}\label{eq: diagram}
    \begin{tikzpicture}[baseline= (a).base]
    \node[scale=1] (a) at (0,0){
    \begin{tikzcd}[row sep=small]
    p_{\F}(\CC^0)\otimes\wedge^2\R^2\arrow{rrd}{\phantom{1}\tilde{\delta}}&0&p_{\F}(\CC^2)\otimes\wedge^{2}\R^2 \\
    p_{\F}(\CC^0)\otimes\wedge^1\R^2\arrow{rrd}{\phantom{1}\tilde{\delta}}&0&p_{\F}(\CC^2)\otimes\wedge^{1}\R^2\\
    p_{\F}(\CC^0)\otimes\wedge^{0}\R^2 &0&p_{\F}(\CC^2)\otimes\wedge^{0}\R^2
    \end{tikzcd}
    };
    \end{tikzpicture}
\end{equation}
Here we only pictured the non-trivial differentials and spaces.
In this section we determine these differentials explicitly, we compute partial inverses to them, and give estimates for the latter.

\begin{lemma}\label{lemma: second differential degree one}
The map $\tilde{\delta}$ in degree $(0,q)\to (2,q-1)$ is given by
\begin{align*}
         \tilde{\delta}: p_{\F}(\CC^0)\otimes \wedge^q\R^2&\to  p_{\F}(\CC^2)\otimes \wedge^{q-1}\R^2\\
  e_{\varphi}(g\circ f)\otimes e_i&\mapsto \frac{(-1)^{i+1}}{2\nf}  p_{\F}(e_{\varphi}(g\circ f)\wedge \tilde{\omega}_i)\otimes 1\\
   e_{\varphi}(g\circ f)\otimes e_1\wedge e_2&\mapsto \frac{1}{2\nf}\sum_{i=1}^{2} p_{\F}(e_{\varphi}(g\circ f)\wedge \tilde{\omega}_i)\otimes e_{3-i}
\end{align*}
with $g\in C^{\infty}_0(\C)$, where we used the isomorphism \eqref{eq: zero fol iso}.
\end{lemma}

\begin{proof}
We prove the result in two steps.\\
\noindent\underline{Step 1}: By Proposition \ref{proposition: chain isomorphism}, the map $\rho^*$ from \eqref{eq:cohomology map} fits into a commutative diagram
\begin{equation*}
    \begin{tikzpicture}[baseline= (a).base]
    \node[scale=.9] (a) at (0,0){
    \begin{tikzcd}
    p_{\F}(\CC^0)\otimes\wedge^q\R^2\arrow{d}{\simeq}\arrow{d}[swap]{\rho^*}\arrow{r}{\tilde{\delta}}&p_{\F}(\CC^{2})\otimes\wedge^{q-1}\R^2\arrow{d}{\simeq}\arrow{d}[swap]{\rho^*}\\
    \mathscr{D}^0\otimes\wedge^q\R^2\arrow{r}{\delta_{\mathscr{D}}}&\mathscr{D}^{2}\otimes\wedge^{q-1}\R^2 
    \end{tikzcd}
    };
    \end{tikzpicture}
\end{equation*}
We claim that the map $\delta_{\mathscr{D}}$ is given by
\begin{align}\label{eq:delta:on:skeleton}
 -\frac{\lambda_1}{2|\lambda|^2}e_{\omega_{S^2}}\otimes i_{\varepsilon^1}- \frac{\lambda_2}{2|\lambda|^2}e_{\omega_{S^2}}\otimes i_{\varepsilon^2}
\end{align}
where $\lambda=\lambda_1+i\lambda_2$ are the coordinates on $\C$ and $\varepsilon^i\in (\R^2)^*$ denotes the dual basis to $e_i\in \R^2$. For this, we first recall from \eqref{eq:pullback forms} that
\[ \rho^*\tilde{\omega}_1 =-\lambda_1\cdot \omega_{S^2}\in \Omega^2(S^2\times \C)^{SU(2)\times \Z_2}\]
Hence we obtain 
\[ \rho^*\dif\tilde{\omega}_1 =-\omega_{S^2}\wedge \dif \lambda_1.\]
Next, we define \begin{align}\label{eq: w}
    W_1:=  \frac{\lambda_1\partial_{\lambda_1}-\lambda_2\partial_{\lambda_2}}{2|\lambda |^2}\quad \text{and}\quad
    W_2:= \frac{\lambda_2\partial_{\lambda_1}+\lambda_1\partial_{\lambda_2}}{2|\lambda |^2}. 
\end{align}
Using \eqref{eq:cohomology map explicit} it is easy to check that $V_i$ and $W_i$ are $\rho$-related, i.e.,
\begin{align*}
    \rho_*(W_i)=V_i|_{\Sk\backslash\{0\}}
\end{align*}
From the definition of the $\gamma_i$ in \eqref{eq: gammai} we obtain
\[ \rho^*\gamma_i = \frac{-\lambda_i}{2|\lambda|^2}\omega_{S^2}.\]
Recall from Proposition \ref{proposition:PCH as forms} that $\delta=e_{\gamma_1}\otimes i_{\epsilon^1}+e_{\gamma_2}\otimes i_{\epsilon^2}$. This implies that for $\xi\in p_{\F}(\CC^0)\otimes\wedge^q\R^2$, we have that $\delta_{0}\rho^*(\xi)=\rho^*\delta(\xi)$, where 
$\delta_0$ is the operator in \eqref{eq:delta:on:skeleton}. On the other hand $\tilde{\delta}(\xi)-\delta(\xi)=\dif h_{\F} \delta(\xi)$. Note that $h_{\F}\delta(\xi)$ is $SU(2)$-invariant, and therefore 
\[\rho^*(h_{\F}\delta(\xi))\in e_{\rho^*(\varphi)}(\Omega^1(S^2\times \C)^{SU(2)\times \mathbb{Z}_2})\otimes \wedge^{q-1} \R^2=0.\]
Thus $\rho^*\tilde{\delta}(\xi)=\rho^*\delta(\xi)=\delta_0\rho^*(\xi)$, and this concludes the proof of the claim.

\noindent\underline{Step 2}: The map $\tilde{\delta}$ in degree $(0,1)\to (2,0)$ is obtained as the composition of the following maps. The map $\rho^*$ identifies $p_{\F}(\CC^0)$ with $\mathscr{D}^0$ by
\begin{align*}
    e_{\varphi}(g \circ f)\otimes e_i&\ \mapsto  e_{\rho^*(\varphi)}(g \circ \mathrm{sq})\otimes e_i
\end{align*}
Then we use the map $\delta_{\mathscr{D}}$ from Step 1 to get
\begin{align*}
    e_{\rho^*(\varphi)}(g \circ \mathrm{sq})\otimes e_i&\, \mapsto -\frac{\lambda_i}{2|\lambda|} g (\lambda^2)\, e_{\rho^*(\varphi)}(\omega_{S^2})\otimes 1
\end{align*}
By \eqref{eq:pullback forms}, the inverse of $\rho^*$ in the diagram in Step 1 is given by:
\begin{align*}
    -\frac{\lambda_i}{2|\lambda|}g (\lambda^2)\, e_{\rho^*(\varphi)}(\omega_{S^2})\otimes 1&\, \mapsto \frac{(-1)^{i+1}}{2\nf}p_{\F}( e_{\varphi}(g \circ f)\wedge \tilde{\omega}_{i})\otimes 1
\end{align*}
So we obtain the form of $\tilde{\delta}$ from the statement. The case $(0,2)\to (2,1)$ follows similarly. 
\end{proof}


To obtain homotopy operators for the Poisson complex, we need explicit inverses for the map $\tilde{\delta}$.

\begin{lemma}\label{lemma: right inverse}
A right inverse for $\tilde{\delta}$ in degree $(0,1)\to (2,0)$ is given by 
\begin{align}\label{eq: right inverse map}
    &\mathscr{R}: p_{\F}(\CC^2)\otimes \wedge^0\R^2 \to p_{\F}(\CC^0)\otimes \wedge^1\R^2 \\
    &\sum_{i=1}^2  p_{\F}(e_{\varphi}(g_i\circ f)\wedge \tilde{\omega}_i) \otimes 1 \mapsto  2\nf \sum_{i=1}^2 (-1)^{i+1} e_{\varphi}(g_i\circ f)\otimes e_i\nonumber
\end{align}
with $(g_1,g_2)\in \mathcal{M}$. 
\end{lemma}

\begin{proof} By Propositions \ref{proposition:flat foliated ch} and \ref{proposition: property H fol}, every element in $p_{\F}(\CC^2)$ is of the form
\begin{align*}
    \sum_{i=1}^2 p_{\F}(e_{\varphi}(g_i\circ f) \wedge  \tilde{\omega}_i), \quad \text{ for unique }\ (g_1,g_2)\in \mathcal{M}.
\end{align*}
Hence $\mathscr{R}$ is well-defined, and it is clearly an right inverse to $\tilde{\delta}$.
\end{proof}

Next, we describe a left inverse for $\tilde{\delta}$ in degree $(0,2)\to (2,1)$. For this we will use the following:

\begin{lemma}\label{lemma: trivial module 3}
We have an isomorphism of $C^{\infty}_0(\C)$-modules:
\[C^{\infty}_0(\C)\oplus C^{\infty}_0(\C) \simeq  p_{\F}(\CC^2)\otimes \R^2.\]
More explicitly, for every element $\xi\in p_{\F}(\CC^2)\otimes\R^2$ there exist unique elements $g_1,g_2\in C^{\infty}_0(\C)$ such that
\begin{align}\label{eq: trivialization in (2,1)}
    \xi= g_1\circ f\cdot b_1 +g_2\circ f\cdot b_2,
\end{align}
where $b_1$ and $b_2$ are the sections on the set $\spl\setminus S_0$ given by:
\begin{align*}
    b_1 :=&\,e_{\varphi}\big(p_{\F}(\tilde{\omega}_1)\otimes e_1 - 
    p_{\F}(\tilde{\omega}_2)\otimes e_2\big)\\
    b_2 :=&\,e_{\varphi}\big(p_{\F}(\tilde{\omega}_2)\otimes e_1 + p_{\F}(\tilde{\omega}_1)\otimes e_2\big).
\end{align*}
\end{lemma}
\begin{proof}
First recall that by Proposition \ref{proposition: twisted module} and Proposition \ref{proposition:flat foliated ch} we have the isomorphisms of $C^{\infty}_0(\C)$-modules:
\begin{equation}\label{eq:composition of isomorphisms}
C^{\infty}_0(\C)\oplus C^{\infty}_0(\C)=\mathcal{M}\oplus \mathcal{K}\simeq 
 \mathcal{M}\oplus  \mathcal{M} 
 \simeq  p_{\F}(\CC^2)\otimes \R^2.
\end{equation}
The statement follows by applying these isomorphism to a basis on the left. More explicitly, consider the map $\chi:C^{\infty}_0(\C)\oplus C^{\infty}_0(\C)\to p_{\F}(\CC^2)$ given by 
\[\chi(g_1,g_2):=\sum_{i=1}^2 p_{\F}(e_{\varphi}(g_i\circ f) \wedge  \tilde{\omega}_i).\]
We know that $\ker \chi=\mathcal{K}$ and its restriction to $\mathcal{M}$ is an isomorphism. Denote $J(g_1,g_2)=(-g_2,g_1)$. Using that $\id=p_{\mathcal{M}}+ p_{\mathcal{K}}$ and that $J(\mathcal{M})=\mathcal{K}$, we obtain the isomorphism from the statement:
\begin{align*}
(g_1,g_2)\mapsto& (\chi \circ p_{\mathcal{M}}(g_1,g_2))\otimes e_1 +
(\chi \circ (-J)\circ p_{\mathcal{\mathcal{K}}}(g_1,g_2))\otimes e_2=\\
&= \chi (g_1,g_2)\otimes e_1+ \chi(g_2,-g_1)\otimes e_2=\\
&=g_1\circ f \cdot b_1+ g_2\circ f\cdot b_2.\qedhere
\end{align*}

\end{proof}
With the notation from the lemma, $\tilde{\delta}$ in degree $(0,2)\to (2,1)$ becomes:
\begin{align}\label{eq: delta degree 2}
    \tilde{\delta}(e_{\varphi}(g\circ f)\otimes e_1\wedge e_2)= \frac{1}{2|f|}g\circ f\cdot b_2.
\end{align}
Therefore, the lemma implies the following: 

\begin{lemma}\label{lemma: second differential degree two}
A left inverse of $\tilde{\delta}$ in degree $(0,2)\to (2,1)$ is given by 
\begin{align}\label{eq: left inverse}
\mathscr{L}&: p_{\F}(\CC^2)\otimes \wedge^1\R^2 \to p_{\F}(\CC^0)\otimes \wedge^2\R^2 \\
&\sum_{i=1}^2 g_i\circ f\cdot b_i\mapsto  2\nf\, e_{\varphi}(g_2\circ f)\otimes e_1\wedge e_2.\nonumber
\end{align}
\end{lemma}

Next, we estimate these maps. 

\begin{lemma}\label{lemma: r}
The maps $\mathscr{R}$ and $\mathscr{L}$ satisfy property SLB for $(0,1,2)$.
\end{lemma}
\begin{proof} 
In order to obtain the estimate for the map $\mathscr{R}$ \eqref{eq: right inverse map}, we write it as a composition of several maps which we estimate one by one: 
\begin{equation}\label{eq: op contraction}
    \begin{array}{ccccccc}
        p_{\F}(\CC^2)&\xrightarrow{\rho^*}&\mathscr{D}^{2}&\xrightarrow{\zeta }&\mathscr{D}^{0} \otimes \R^2&\xrightarrow{(\rho^*)^{-1}\otimes \id }&p_{\F}(\CC^0)\otimes \R^2 
    \end{array}
\end{equation}
where we use the cochain isomorphism $\rho^*$ from Proposition \ref{proposition: chain isomorphism}, and
\[ \zeta :=-2\lambda_1 i_{\pi_{S^2}}\otimes e_1-2\lambda_2 i_{\pi_{S^2}}\otimes e_2\]
with $\pi_{S^2}=\omega_{S^2}^{-1}$ the Poisson structure on $S^2$ dual to $\omega_{S^2}$. Next, let us see that this composition is indeed $\mathscr{R}$. Using \eqref{eq:pullback forms} for $\rho^*(\tilde{\omega}_i)$ and \eqref{eq:final:to:prove} we obtain for $(g_1,g_2)\in \mathcal{M}$ that
\begin{align*}
    \sum_{i=1}^2g_i\circ f\cdot e_{\varphi}(p_{\F}(\tilde{\omega}_i))\mapsto &\ \sum_{i=1}^2(-1)^{i}\lambda_i g_i\circ \mathrm{sq}\cdot  \rho^*(\varphi)\wedge \omega_{S^2}\\
    \mapsto &\ 2\sum_{j=1}^2(-1)^{j+1}(|z|\cdot g_j)\circ \mathrm{sq}\cdot \rho^*(\varphi)\otimes e_j\\
    \mapsto &\ 2\nf \sum_{j=1}^2(-1)^{j+1} e_{\varphi}(g_j\circ f)\otimes e_j
\end{align*}
Here we used for the second map that
\begin{align*}
    (g_1,g_2)\in \mathcal{M} \quad \Longleftrightarrow \quad \lambda_2 g_1\circ \mathrm{sq}=-\lambda_1 g_2\circ \mathrm{sq}  
\end{align*}
by pulling back the relation in \eqref{eq: definition M} by $\mathrm{sq}:\C\to \C$ and hence, for $j=1,2$:
\begin{align*}
    -\lambda_j\sum_{i=1}^2(-1)^{i}\lambda_i g_i\circ \mathrm{sq}=(-1)^{j+1}|\lambda|^2\cdot g_j\circ\mathrm{sq} =(-1)^{j+1}(|z|\cdot g_j)\circ \mathrm{sq}.
\end{align*}
In the following we estimate each of the maps in \eqref{eq: op contraction} separately.\\
 \noindent\underline{Step 1}: We have for $\beta\in \Omega_0^{\bullet}(\spl)$ the inequality
\begin{align}\label{Step four crazy Lemma}
    \fnorm{\rho^*(\beta)}_{n,k,\frac{r}{\sqrt{2}}}^{S^2\times \C}\le C\cdot \fnorm{\beta}_{n,k,r}
\end{align}
for a constant $C=C(n,k,r)$ depending continuously on $r>0$. Here the semninorms on the left are defined as follows. Equip the $n$-th jet bundle of the vector bundle
\begin{align*}
    \wedge^\bullet T^*(S^2\times\C)\to S^2\times\C
\end{align*}
with a norm $|\cdot|^{jet}$. Then, for $\alpha \in \Omega^{\bullet}_0(S^2\times \C)$, define: 
\begin{align*}
    \fnorm{\alpha}_{n,k,s}^{S^2\times \C}:= \sup_{(w,\lambda)\in S^2\times\overline{B}_s}\ |\lambda|^{-k}|j^n\alpha(w,\lambda)|^{jet}.
\end{align*}
We obtain \eqref{Step four crazy Lemma} by compactness of $S^2$ and using
\begin{align*}
    \sqrt{2}|\lambda| = R\circ \rho (w,\lambda),
\end{align*}
which follows by Lemma \ref{lemma: characterization skeleton} (3) and by Lemma \ref{lemma:identification of skeleton} (2). 

\noindent\underline{Step 2}: Note that for $\alpha\in \mathscr{D}^2$ we have
\begin{align}\label{Step three crazy Lemma}
    \fnorm{\zeta (\alpha)}_{n,k,s}^{\C} \le C\cdot \fnorm{\alpha}_{n,k,s}^{S^2\times \C}
\end{align}
for a constant $C(n,k,s)$ depending continuously on $0<s$. Here we used the seminorms $\fnorm{\cdot }_{n,k,s}^{\C}$ on $C^{\infty}_0(\C)$, defined as in \eqref{eq: flat norms}.

\noindent\underline{Step 3}: We show that for $g\in C^{\infty}_0(\C)$ we have
\begin{align}\label{Step one in crazy lemma}
    \fnorm{g\circ f\cdot \varphi }_{n,k,r}\le C\cdot \fnorm{g\circ \mathrm{sq}\cdot \rho^*(\varphi)
    }^{\C}_{n,k+n+2,\frac{r}{\sqrt{2}}}.
\end{align}
Clearly, we have that:
\begin{align*}
    \fnorm{g\circ f\cdot \varphi}_{n,k,r}\le C\cdot \fnorm{g\circ f}_{n,k,r}
\end{align*}
By the explicit form of the vector fields $W_1$ and $W_2$ from \eqref{eq: w}, we have that, for any $g\in C^{\infty}_0(\C)$, and for $i=1,2$:
\begin{align*}
    (\partial_i g)\circ \mathrm{sq} = \Lie_{W_i} (g\circ \mathrm{sq}). 
\end{align*}
For all $u\in C^{\infty}_0(\C)$ and $i=1,2$ we have:
\begin{align}\label{eq:norm of Lie W}
    \fnorm{\Lie_{W_i}u}_{n,k,s}^{\C}\le C\cdot \fnorm{u}_{n+1,k+1,s}^{\C}.
\end{align}
This follows by Lemma \ref{lemma:singular:vb:maps} applied to the coefficients of $W_i$. Using the general chain rule from Subsection \ref{subsection: derivatives} and $2\nf\le R^2$ from \eqref{eq:norm-fiber inequality}, we obtain for $a\in \N_0^6$, $k\in \N_0$, $0< R \le r$ and $g\in C^{\infty}_0(\C)$ that
\[\frac{|D^a(g\circ f)|}{R^k} \le  C\cdot \frac{1}{\nf^{\frac{k}{2}}} \sum_{1\leq |b|\leq |a|}|D^b( g)\circ f|.\]
By $2\nf\le R^2$ and by applying $f$ to diagonal matrices, we note that:
\[f(\overline{B}_r)=\overline{B}_{r^2/2}.\]
For $A\in \overline{B}_r$, let $\lambda\in \overline{B}_{r/\sqrt{2}}$ such that $f(A)=\lambda^2$. Then the above gives:
\begin{align*}
    \frac{|D^a(g\circ f)|}{R^k}     \le &\ C\cdot \frac{1}{|\lambda|^{k}} \sum_{1\leq |b|\leq |a|}|D^b( g)\circ \mathrm{sq}|\\
    = &\ C\cdot \frac{1}{|\lambda|^{k}}\sum_{1\leq |b|\leq |a|}|\Lie_{W_1}^{b_1}\Lie_{W_2}^{b_2} (g\circ \mathrm{sq})|\\
    \le &\ C\cdot\sum_{1\leq |b|\leq |a|}\fnorm{\Lie_{W_1}^{b_1}\Lie_{W_2}^{b_2} (g\circ \mathrm{sq})}_{0,k,\frac{r}{\sqrt{2}}}^{\C}\\
    \le &\ C\cdot \fnorm{g\circ \mathrm{sq}}_{|a|,k+|a|,\frac{r}{\sqrt{2}}}^{\C},
\end{align*}
where in the last step we used \eqref{eq:norm of Lie W} inductively. This proves \eqref{Step one in crazy lemma} using Lemma \ref{lemma:singular:vb:maps} and the form of $\rho^*(\varphi)$ \eqref{eq:pull back of varphi}. 

Applying the estimates from Step 1 to 3 in reverse order proves that the map $\mathscr{R}$ satisfies property SLB for $(0,1,2)$. 

To estimate $\mathscr{L}$ we denote by $e_{e_j}$ and $i_{\varepsilon^j}$ the exterior and interior product on $\wedge^q\R^2$ with $e_j$ and $\varepsilon^j$, respectively. We claim that:
\begin{align*}
    \mathscr{L}=-(\id\otimes e_{e_1})\circ \mathscr{R}\circ (\id \otimes i_{\varepsilon ^1})-(
    \id\otimes e_{e_2})\circ \mathscr{R}\circ (\id \otimes i_{\varepsilon^2})
\end{align*}
For this, consider an arbitrary element $g_1\circ f\cdot b_1+g_2\circ f\cdot b_2\in p_{\F}(\CC^2)\otimes \R^2$. 
Then applying the first operator in the sum, we obtain:
\begin{align*}
g_1\circ f\cdot b_1+g_2\circ f\cdot b_2\ \  \stackrel{\id\otimes i_{\varepsilon^1}}{\mapsto}\quad \, &  g_1\circ f\cdot p_{\F}(e_{\varphi}(\tilde{\omega}_1))+g_2\circ f \cdot p_{\F}(e_{\varphi}(\tilde{\omega}_2))\\
=\qquad & \big(\frac{1}{2|z|}((|z|+x)g_1-yg_2)\big)\circ f\cdot p_{\F}(e_{\varphi}(\tilde{\omega}_1))\\
&+\big(\frac{1}{2|z|}((|z|-x)g_2-yg_1)\big)\circ f \cdot p_{\F}(e_{\varphi}(\tilde{\omega}_2))\\
\stackrel{-(\id\otimes e_{e_1})\circ \mathscr{R}}{\mapsto} &\,((|z|-x)g_2-yg_1)\circ f\cdot \varphi\otimes e_1\wedge e_2,
\end{align*}
where in the equality we have used the projection $p_{\mathcal{M}}$. Similarly, for the second operator, one obtains:
\begin{align*}
g_1\circ f\cdot b_1+g_2\circ f\cdot b_2\ \  \stackrel{\id\otimes i_{\varepsilon^2}}{\mapsto}\quad \, &  -g_1\circ f\cdot p_{\F}(e_{\varphi}(\tilde{\omega}_2))+g_2\circ f \cdot p_{\F}(e_{\varphi}(\tilde{\omega}_1))\\
=\qquad & \big(\frac{1}{2|z|}((|z|+x)g_2+yg_1)\big)\circ f\cdot p_{\F}(e_{\varphi}(\tilde{\omega}_1))\\
& -\big(\frac{1}{2|z|}((|z|-x)g_1+yg_2)\big)\circ f \cdot p_{\F}(e_{\varphi}(\tilde{\omega}_2))\\
\stackrel{-(\id\otimes e_{e_2})\circ \mathscr{R}}{\mapsto} &\,((|z|+x)g_2+yg_1)\circ f\cdot \varphi\otimes e_1\wedge e_2,
\end{align*}
which yields the claim. Since $\mathscr{L}$ is obtained by combining $\mathscr{R}$ with the exterior and interior products by $e_i$ and $\varepsilon^i$, respectively, it follows that $\mathscr{L}$ also satisfies property SLB for $(0,1,2)$.
\end{proof}

\subsubsection{A homotopy equivalence data}\label{section: HE data}

Corollary \ref{corollary: homotopy} reduces the complex $(\CC^{\bullet}\otimes \wedge^{\bullet}\R^2,\dif +\delta)$, and the construction of ``good'' homotopy operators for it, to the complex $(p_{\F}(\CC^{\bullet})\otimes \wedge^{\bullet}\R^2,\tilde{\delta})$. The previous subsection allows us to deal with the latter.   
\begin{corollary}\label{proposition: homotopy}
The complex $(p_{\F}(\CC^{\bullet})\otimes \wedge^{\bullet}\R^2,\tilde{\delta})$ admits the HE data
\begin{align}\label{eq: HE data E}
    (\widetilde{\mathscr{E}}^{\bullet},0)\stackrel[\iota ]{p}{\leftrightarrows}(p_{\F}(\CC)\otimes \wedge\R^2,\tilde{\delta}),h
\end{align}
where $\widetilde{\mathscr{E}}$ is the subcomplex given by 
\begin{align*}
    \widetilde{\mathscr{E}}^k:= \begin{cases}
        p_{\F}(\CC^0)\otimes \wedge^0\R^2 &\text{ for } k=0\\
        \ker \tilde{\delta}&\text{ for } k=1\\
        \mathrm{coker}\, \tilde{\delta}&\text{ for } k=3\\
        p_{\F}(\CC^2)\otimes \wedge^2\R^2 &\text{ for } k=4\\
        0&\text{ for }k=2,5,6
    \end{cases}
\end{align*}
satisfying
\begin{align*}
    p \circ \iota =\id_{\widetilde{\mathscr{E}}},
\end{align*}
and $h$ satisfy property SLB for $(0,1,2)$. \end{corollary}

\begin{proof}
The maps $\iota$ are the standard inclusions. We give the maps $p$ and $h$ in every degree. \\
\noindent\textbf{In degree $0$}: The map $p=\id$ and $h=0$.\\
\noindent\textbf{In degree $1$}: The map $p$ is defined by
\begin{align*}
    p:= \id -\mathscr{R}\circ \tilde{\delta}
\end{align*}
and the map $h$ is given by
\begin{equation*}
    \begin{array}{cccc}
         h:& p_{\F}(\CC^0)\otimes \wedge ^2\R^2\oplus\, 0 \oplus p_{\F}(\CC^2)\otimes \wedge ^0\R^2&\to & p_{\F}(\CC^0)\otimes \wedge ^1\R^2\oplus\, 0 \\
         & (\alpha^{0,2},0,\alpha^{2,0})&\mapsto &(-\mathscr{R}(\alpha^{2,0}),0)
    \end{array}
\end{equation*}
\noindent\textbf{In degree $2$}: Here $p=0$ and $h$ is defined by
\begin{equation*}
    \begin{array}{cccc}
         h:& 0\oplus p_{\F}(\CC^2)\otimes \wedge ^1\R^2\oplus \, 0&\to &p_{\F}(\CC^0)\otimes \wedge ^2\R^2\oplus\, 0 \oplus p_{\F}(\CC^2)\otimes \wedge ^0\R^2 \\
         & (0,\alpha^{2,1},0)&\mapsto &(-\mathscr{L}(\alpha^{2,1}),0,0)
    \end{array}
\end{equation*}
where the map $\mathscr{L}$ is given in \eqref{eq: left inverse}.

\noindent\textbf{In degree $3$}: The projection map $p$ is given by
\begin{align*}
    p := \id -\tilde{\delta}\circ \mathscr{L} 
\end{align*}
and $h$ is the zero map.\\
\noindent\textbf{In degree $4$}: Here $p=\id$ and $h=0$. 

Finally, by Lemma \ref{lemma: r}, $h$ has property SLB for $(0,1,2)$.
\end{proof}

\subsubsection{The flat Poisson cohomology of $\spl$}\label{section: flat pch}

As a corollary of the previous subsections, we obtain the description of the flat Poisson cohomology groups of $(\spl,\pi)$ together with a deformation retraction to a complex with trivial differential. The algebraic structure of the flat Poisson cohomology groups is discussed in Subsection \ref{section: algebra}.

\begin{theorem}\label{main theorem}
The flat Poisson complex $(\mathfrak{X}^{\bullet}_0(\spl),\dif_{\pi})$ admits HE data
\[ \begin{tikzcd}[every arrow/.append style={shift right}] 
(\mathscr{E}^{\bullet},0) \arrow{r}[swap]{\iota _{\spl}} &(\mathfrak{X}^{\bullet}_0(\spl) ,\dif_{\pi}),h_{\spl} \arrow{l}[swap]{p_{\spl}}
\end{tikzcd}\]
where $\mathscr{E}$ is given by 
\begin{align*}
    \mathscr{E}^k:= \begin{cases}
    C^{\infty}_0(\C) &\text{ for } k=0,3\\
    \mathcal{M}&\text{ for } k=1,4\\
        0&\text{ for }k=2,5,6
    \end{cases}
\end{align*}
such that
\[ p_{\spl}\circ \iota_{\spl}=\id_{\mathscr{E}},\]
and such that the maps $\iota_{\spl}\circ p_{\spl}$ and $h_{\spl}$ satisfies property SLB for $(2,21,188)$ and $(1,21,167)$, respectively. 

In particular, the map $\iota_{\spl}$ induces an isomorphism of $C^{\infty}_0(\mathbb{C})$-modules
\begin{equation}\label{eq: isoPC}
    I^{\bullet}:\mathscr{E}^{\bullet}\diffto H^{\bullet}_0(\spl,\pi)
\end{equation}
and in degree $2$ we have the homotopy relation
\begin{align*}
    \id_{\mathfrak{X}^{2}_0(\spl)}=h_{\spl}\circ \dif_{\pi}+\dif_{\pi}\circ h_{\spl}
\end{align*}
\end{theorem}

\begin{proof}
Note that we have an isomorphism $\epsilon:\mathscr{E}^{\bullet}\diffto \widetilde{\mathscr{E}}^{\bullet}$ given as follows:\\ 
\noindent\textbf{In degree $0$}, $\epsilon$ is the isomorphism from \eqref{eq: zero fol iso}.\\
\noindent\textbf{In degree $1$}, $\epsilon$ is given by
\begin{align}\label{eq: kernel 01}
   \epsilon: \quad \mathcal{M}\quad &\diffto \qquad\qquad  \ker \tilde{\delta} \\
    (g_1,g_2)  &\mapsto  \ e_{\varphi}(g_2\circ f)\otimes e_1  + e_{\varphi}( g_1\circ f) \otimes e_2\nonumber
\end{align}
That this map is a well-defined isomorphism follows by Lemma \ref{lemma: second differential degree one}, the description of the kernel of the map \eqref{eq: quotient} in Proposition \ref{proposition:flat foliated ch}, and the isomorphism $\mathcal{M}\simeq \mathcal{K}$ from Proposition
\ref{proposition: twisted module}.\\
\noindent\textbf{In degree $3$}, $\epsilon$ is given by
\begin{align}\label{eq: iso degree 3}
    \epsilon: C^{\infty}_{0}(\C)&\diffto \quad \mathrm{coker}\, \tilde{\delta}   \\
     g&\mapsto\ g\circ f\cdot e_{\varphi}(b_1)  \nonumber
\end{align}
This is an isomorphism by Lemma \ref{lemma: trivial module 3} and \eqref{eq: delta degree 2}.\\
\noindent\textbf{In degree $4$}, $\epsilon$ is given by
\begin{align}\label{eq: iso degree 4}
    \epsilon: \mathcal{M}&\diffto p_{\F}(\CC^2)\otimes \wedge^2\R^2 \\
    (g_1,g_2)\mapsto &\sum_{i=1}^2 p_{\F}(g_i\circ f\cdot e_{\varphi}(\tilde{\omega}_i))\otimes e_1\wedge e_2\nonumber
\end{align}
which, by Proposition \ref{proposition:flat foliated ch}, is an isomorphism. 

The HE data is obtained from the cochain isomorphism $\tilde{a}$ in Proposition \ref{proposition:PCH as forms}, the isomorphism $\epsilon$ and the HE data from Corollary \ref{corollary: homotopy} and Corollary \ref{proposition: homotopy}.
By using that HE data behave functorial \cite[Lemma 2.4]{LS87}, the formulas for the composition yield:
\begin{align*}
    \iota_{\spl}:= \tilde{a}^{-1}\circ \tilde{\iota} \circ \iota\circ \epsilon,& \qquad p_{\spl}:= \epsilon^{-1}\circ p\circ \tilde{p} \circ \tilde{a} \\ h_{\spl}=\tilde{a}^{-1}\, \circ &\, (\tilde{\iota}\circ h\circ  \tilde{p}+\tilde{h})\circ \tilde{a}.
\end{align*}
That these form an HE data, can be checked directly:
\begin{align*}
    \iota _{\spl} \circ p_{\spl}-\id =&\ \tilde{a}^{-1}\circ \tilde{\iota}\circ \iota\circ p \circ  \tilde{p}\circ \tilde{a}-\id\\
    =&\ \tilde{a}^{-1}\circ \tilde{\iota}\circ (\iota \circ p -\id ) \circ  \tilde{p}\circ \tilde{a}+\tilde{a}^{-1}\circ\big(\tilde{\iota} \circ  \tilde{p}- \id\big)\circ \tilde{a}\\
    =&\ \tilde{a}^{-1}\circ \big( \tilde{\iota}\circ (\tilde{\delta}\circ h+h\circ \tilde{\delta} ) \circ  \tilde{p}\big) \circ \tilde{a} \\
    &\ +\tilde{a}^{-1}\circ\big((\dif +\delta) \circ \tilde{h} +\tilde{h}\circ (\dif+\delta)\big)\circ \tilde{a}\\
    =&\ \tilde{a}^{-1}\circ \big( (\dif+ \delta) \circ \big( \tilde{\iota}\circ h\circ  \tilde{p}+\tilde{h}\big)\big)\circ \tilde{a}\\
    &\ + \tilde{a}^{-1}\circ \big( \tilde{\iota}\circ h\circ  \tilde{p}+\tilde{h}\big) \circ (\dif+\delta)\big)\circ \tilde{a}\\
    =&\ \tilde{a}^{-1}\circ (\dif+ \delta)\circ \tilde{a}\circ   h_{\spl} + 
     h_{\spl}\circ \tilde{a}^{-1}\circ  (\dif+ \delta)\circ \tilde{a}\\
     =&\ \dif_{\pi}\circ h_{\spl}+h_{\spl}\circ \dif_{\spl}.
\end{align*}
Finally, Proposition \ref{proposition:PCH as forms}, Corollary \ref{corollary: homotopy}, Corollary \ref{proposition: homotopy} and Lemma \ref{lemma: slb property} imply that  $h_{\spl}$ satisfy the SLB property for $(1,21,188)$. Since $\dif_{\pi}$ satisfies the SLB property for $(1,0,0)$, again Lemma \ref{lemma: slb property} implies the claimed SLB property for $\iota_{\spl}\circ p_{\spl}$.
\end{proof}

\subsection{The proof of Theorem \ref{theorem: PC}}\label{section: algebra}

In degree 0, the isomorphism was already proven in Proposition \ref{prop:Casimirs}. For the other degrees, we will give an explicit description for the isomorphism $I^{\bullet}$ in \eqref{eq: isoPC}. First of all, to identify the $C^{\infty}(\C)$-module structure, note that the maps $I^{\bullet}$ are $C^{\infty}(\C)$-linear (where we identify $H^{0}(\spl,\pi)\simeq C^{\infty}(\C)$). 

First, we claim that, for $(g_1,g_2)\in \mathcal{M}$ the class $I^1(g_1,g_2)\in H^1_0(\spl,\pi)$ acts on $C^{\infty}(\C)$ via the derivation 
\[\frac{1}{2|z|}(g_1\cdot Y_1+ g_2\cdot Y_2)\in \out.\]
Note that this suffices to obtain the description from the theorem of the first Poisson cohomology group: it shows that the claimed map is an isomorphism, and since the Lie bracket is determined by the action on functions, it follows that the isomorphism preserves the Lie bracket. 

Indeed, applying the maps $\epsilon$ in \eqref{eq: kernel 01}, the inclusion $\iota$ in Corollary \ref{proposition: homotopy} and $\tilde{\iota}$ in Corollary \ref{corollary: homotopy}, we obtain that
\[ I^1(g_1,g_2)=\tilde{a}^{-1} (\epsilon(g_1,g_2)+h_{\F}\circ \delta(\epsilon(g_1,g_2))).\]
Note that, since $h_{\F}\circ \delta(\epsilon(g_1,g_2))\in \CC^1\otimes \wedge^0\R^2$, we have that, for $i=1,2$:
\[ i_{\dif f_i} (\tilde{a}^{-1} \circ h_{\F}\circ \delta(\epsilon(g_1,g_2)))=0.\]
By definition of $\tilde{a}^{-1}$ in \eqref{eq: tilde b} and of $\epsilon$ we get for the first term
\begin{equation}\label{eq: deg 1 top part}
    \tilde{a}^{-1} (\epsilon(g_1,g_2))= g_2\circ f\cdot V_1+g_1\circ f\cdot V_2.
\end{equation}
By using \eqref{eq:v transversality}, we conclude that $I^1(g_1,g_2)$ acts on $C^{\infty}(\C)$ as:
\[g_2\cdot \partial_x+ g_1\cdot \partial_y=\frac{1}{2|z|}(g_1\cdot Y_1+ g_2\cdot Y_2),\]
where in the last step we used the relation \eqref{eq: definition M} defining $\mathcal{M}$.


In degree $3$ we want to show that every class in $H^3(\spl,\pi)$ is uniquely associated to $(g_1,[g_2])\in C^{\infty}(\C)\oplus C^{\infty}(\C)/C^{\infty}_0(\C)$ via: 
    \[[g_1\circ f\cdot C_{\mathcal{R}}+g_2\circ f \cdot C_{\mathcal{I}}] \in H^3(\spl,\pi).\]
Note that by Proposition \ref{ses smooth formal} and the formal Poisson cohomology in Proposition \ref{formal cohomology sl2} we only need to show that any flat cohomology class is uniquely determined by a $g\in C^{\infty}_0(\C)$ via
\[[g\circ f\cdot C_{\mathcal{R}}] \in H^3_0(\spl,\pi).\]
To obtain this result we note that the cocycle $C_{\mathcal{R}}$ satisfies for $i=1,2$ that 
 \begin{align}\label{eq: max trans 3}
     i_{\dif f_i} (C_{\mathcal{R}})= \pi_i.
 \end{align}
Hence the only non-trivial parts of  $\tilde{a}(C_{\mathcal{R}})$ are in bi-degrees $(2,1)$ and $(3,0)$, see \eqref{eq: complex diag}. Moreover, using the relations:
\begin{align}\label{eq: bivector to forms}
    \tilde{\omega}_1^{\flat}(\pi_1)= -\tilde{\omega}_1 \quad \textrm{ and } \quad  \tilde{\omega}_1^{\flat}(\pi_2)= \tilde{\omega}_2,
\end{align}
which can be checked directly, we obtain: 
\begin{align*}
    \big(\tilde{a}(g\circ f\cdot C_{\mathcal{R}})\big)^{(2,1)}=&\, \tilde{a}(g\circ f\cdot (V_1\wedge \pi_1 +V_2\wedge \pi_2))\\
    =& \, g\circ f\cdot e_{\varphi}\big(-\tilde{\omega}_1\otimes e_1 +
    \tilde{\omega}_2\otimes e_2\big)
\end{align*}
Applying the map $\tilde{p}$ from Corollary \ref{corollary: homotopy} we note that the only non trivial part is the $(2,1)$-part (see \eqref{eq: diagram}). Hence we get using \eqref{eq: iso degree 3} that
\begin{align}
    \tilde{p}\circ \tilde{a}(g\circ f\cdot C_{\mathcal{R}})=&\,  p_{\F}\otimes \id\big(g\circ f\cdot e_{\varphi}(-\tilde{\omega}_1\otimes e_1 +
    \tilde{\omega}_2\otimes e_2)\big)\label{eq: cr as b1}\\
    =&\, -g\circ f \cdot b_1 =\epsilon(-g).\nonumber
\end{align}
Therefore the statement holds by Theorem \ref{main theorem}.

Next we show that in degree $4$ we have an isomorphism: 
\[[X]\in H^1(\spl,\pi)\quad \leftrightarrow \quad [X\wedge C_\mathcal{R}]\in H^4(\spl,\pi).\]
From \eqref{eq: deg 1 top part} and \eqref{eq: max trans 3} we get for $(g_1,g_2)\in \mathcal{M}$ that
\[ i_{\dif f_1\wedge \dif f_2}(I(g_1,g_2)\wedge C_{\mathcal{R}})= g_2\circ f\cdot \pi_2- g_1\circ f\cdot \pi_1 \]
and therefore using \eqref{eq: bivector to forms}, Corollary \ref{corollary: homotopy} and \eqref{eq: iso degree 4} we obtain
\[ \tilde{p}\circ \tilde{a}(I(g_1,g_2)\wedge C_{\mathcal{R}})=\epsilon(g_1,g_2),\]
which implies the claim by Theorem \ref{main theorem}.

Finally, we prove Proposition \ref{Proposition:action:H3}, which is equivalent to the relations
\[[[\iota_{\spl}(g_1,g_2),C_{\mathcal{R}}]]=\left[\frac{g_2\circ f}{|f|}C_{\mathcal{R}}\right] \quad \text{ and } \quad [[X,C_{\mathcal{I}}]]=0 \]
in $H^3(\spl,\pi)$ for any $[\iota_{\spl}(g_1,g_2)]\in H^1(\spl,\pi)$.

First, using that $\pi_{\C}$ is linear and holomorphic, we obtain for the complex Euler vector field $E_{\C}=E_1+iE_2$ on $\spl$ the relation
\begin{equation*}
[\pi_1,E_{\C}]=2[\pi_{\C}+\overline{\pi_{\C}},E_{\C}]=2\pi_{\C},
\end{equation*}
which implies that
\[\pi_2=2[\pi_1,E_2].\]
Hence for a Poisson vector field $X=\iota_{\spl}(g_1,g_2)$, where $(g_1,g_2)\in \mathcal{M}$, a direct computation yields
\begin{align}
    [X,4V_{\C}\wedge \pi_{\C}]=&\, 4[X,V_{\C}]\wedge \pi_{\C}+ V_{\C}\wedge\big([X, \pi_1]+i[X,\pi_2]\big) \nonumber \\
    =&\,4[X,V_{\C}]\wedge \pi_{\C}+2iV_{\C}\wedge[X, [\pi_1,E_2]]\nonumber\\
    =&\,4[X,V_{\C}]\wedge \pi_{\C}+2iV_{\C}\wedge[\pi_1,[X ,E_2]]\label{eq: cartan in cohomology}\\
    =&\, 4[X,V_{\C}]\wedge \pi_{\C}+2i[\pi_1,V_{\C}]\wedge[X ,E_2]-2i[\pi_1,V_{\C}\wedge[X ,E_2]] \nonumber
\end{align}
We consider the real and imaginary part separately. The last term is trivial in cohomology in both cases. For the first term we use
\begin{align*}
    \iota_{\dif f_1}([X,V_{\C}])=&\, ((-\partial_x+i\partial_y)g_2)\circ f\\
    \iota_{\dif f_2}([X,V_{\C}])=&\, ((-\partial_x+i\partial_y)g_1)\circ f
\end{align*}
which gives the part transverse to the foliation, i.e., it implies that 
\[4[X,V_{\C}]\wedge \pi_{\C}+
4\big(((\partial_x-i\partial_y)g_2)\circ f\cdot  V_1+
((\partial_x-i\partial_y)g_1)\circ f\cdot  V_2\Big) \wedge \pi_{\C}\in \wedge^3 T\F.\]
This together with \eqref{eq: bivector to forms} imply for the real part that
\begin{align*}
    \tilde{p}\circ \tilde{a} (\Re (&\,4[X,V_{\C}]\wedge \pi_{\C}))= \\
     &\, p_{\F}(e_{\varphi}(\partial_xg_2\circ f)\wedge \tilde{\omega}_1) \otimes e_1-p_{\F}(e_{\varphi}(\partial_yg_2\circ f)\wedge \tilde{\omega}_2) \otimes e_1\\
    &\, +p_{\F}(e_{\varphi}(\partial_xg_1\circ f)\wedge \tilde{\omega}_1)\otimes e_2-p_{\F}(e_{\varphi}(\partial_yg_1\circ f)\wedge \tilde{\omega}_2) \otimes e_2
\end{align*}
In order to obtain the corresponding coefficient for $b_1$ in the decomposition \eqref{eq: trivialization in (2,1)} we follow the sequence of identifications in \eqref{eq:composition of isomorphisms}, i.e. we have
\begin{align*}
    p_{\mathcal{M}}(\partial_xg_i,-\partial_yg_i)=&\, \frac{1}{2|z|}((|z|+x)\partial_xg_i+y\partial_yg_i,-y\partial_xg_i-(|z|-x)\partial_yg_i)
\end{align*}
for $i=1,2$ and hence the coefficient for $b_1$ is given by
\begin{equation*}
    \frac{1}{2|f|}((|z|+x)\partial_xg_2+y\partial_yg_2+y\partial_xg_1+(|z|-x)\partial_yg_1)\circ f= -\frac{1}{2|f|}g_2\circ f
\end{equation*}
Here we used derivatives of the defining equation for $(g_1,g_2)\in \mathcal{M}$, i.e.
\begin{equation}\label{eq: set1}
    \partial_x(yg_1)=-\partial_x((|z|+x)g_2) \quad \text{ and }\quad \partial_y((|z|-x)g_1)=-\partial_y(yg_2).
\end{equation}
For the second term in \eqref{eq: cartan in cohomology} we first calculate the transverse parts of $[X,E_2]$:
\begin{align*}
   h_1\circ f:=\iota_{\dif f_1}([X ,E_2])=((-y\partial_x+x\partial_y)g_2 +g_1)\circ f\\
    h_2\circ f:=\iota_{\dif f_2}([X ,E_2])=((-y\partial_x+x\partial_y)g_1-g_2)\circ f
\end{align*}
For the real part, we have that
\begin{align*}
    \iota_{\id \otimes (1\otimes \epsilon^1)}\circ \tilde{p}\circ \tilde{a} (\Re(&\,2i[\pi_1,V_{\C}]\wedge[X ,E_2]))\\
    &\, = \tilde{p}\circ \iota_{\id \otimes (1\otimes \epsilon^1)}\circ \tilde{a} (2[\pi_1,V_2]\wedge[X ,E_2]))\\
    &\, = \tilde{p}\circ \tilde{a} (2\iota_{\dif f_1}([X ,E_2])[\pi_1,V_2])\\
    &\, = \tilde{p}\circ \tilde{a} ([\pi_1,2h_1\circ f \cdot V_2])\\
    &\, = \tilde{p}\circ  \tilde{a} \circ \dif_{\pi_1}\circ \tilde{a}^{-1}(e_{\varphi}(2h_1\circ f) \otimes e_2)\\
    &\, = \tilde{p}\circ  (\dif _{\F} +\delta)(e_{\varphi}(2h_1\circ f )\otimes e_2)\\
    &\, = \tilde{\delta}(e_{\varphi}(2h_1\circ f) \otimes e_2)\\
    &\, =p_{\F}(e_{\varphi}(\frac{1}{|f|}((y\partial_x-x\partial_y)g_2-g_1)\circ f)\wedge \tilde{\omega}_2) \otimes 1
\end{align*} 
where we used Lemma \ref{lemma: second differential degree one} in the last step. Similarly, we have
\begin{align*}
   \iota_{\id \otimes (1\otimes \epsilon^2)}\circ \tilde{p}\circ \tilde{a} (\Re(&\,2i[\pi_1,V_{\C}]\wedge[X ,E_2]))\\
   &\, = p_{\F}(e_{\varphi}(\frac{1}{|f|}((y\partial_x-x\partial_y)g_1+g_2))\circ f)\wedge \tilde{\omega}_2) \otimes 1.
\end{align*}
Hence we obtain
\begin{align*}
   \tilde{p}\circ \tilde{a} (\Re(&\,2i[\pi_1,V_{\C}]\wedge[X ,E_2]))\\
   &\, =p_{\F}(e_{\varphi}(\frac{1}{|f|}((y\partial_x-x\partial_y)g_2-g_1)\circ f)\wedge \tilde{\omega}_2) \otimes e_1\\
   &\ \quad + p_{\F}(e_{\varphi}(\frac{1}{|f|}((y\partial_x-x\partial_y)g_1+g_2))\circ f)\wedge \tilde{\omega}_2) \otimes e_2.
\end{align*}
By applying the projection onto the $b_1$-component in \eqref{eq: trivialization in (2,1)} we get
\begin{align*}
    \frac{1}{2|f|^2}(-y(y\partial_xg_2+(|z|-x)\partial_xg_1)+x(y\partial_y g_2+&\, (|z|-x)\partial_y g_1))\circ f \\
    +\frac{1}{2|f|^2}(yg_1-(|z|-x)g_2)\circ f&\, =
    -\frac{1}{2|f|}g_2\circ f
\end{align*}
where we used additionally the equations
\begin{align}
    \partial_y(yg_1)=-\partial_y((|z|+x)g_2) \quad &\, \text{ and }\quad \partial_x((|z|-x)g_1)=-\partial_x(yg_2),\label{eq: set2}\\
    yg_1=-(|z|+x)g_2 \quad &\, \text{ and }\quad (|z|-x)g_1=-yg_2.\label{eq: set3}
\end{align}
Using that $\tilde{p}\circ \tilde{a}$ maps exact multivector fields to multiples of $b_2$ we obtain as a result of the calculations that
\begin{align*}
    \tilde{p}\circ \tilde{a}([X,C_{\mathcal{R}}])=-\frac{1}{|f|}g_2\circ f \cdot b_1 \mod C^\infty_0(\C)\circ f\cdot b_2.
\end{align*}
This implies the first statement about $C_\mathcal{R}$ by \eqref{eq: cr as b1}. 

The result for $C_{\mathcal{I}}$ follows from a similar computation of the imaginary part in \eqref{eq: cartan in cohomology}, i.e. the first part yields
\begin{align*}
    \tilde{p}\circ \tilde{a} (\Im(&\,4[X,V_{\C}]\wedge \pi_{\C}))= \\
     &\, -p_{\F}(e_{\varphi}(\partial_yg_2\circ f)\wedge \tilde{\omega}_1) \otimes e_1-p_{\F}(e_{\varphi}(\partial_xg_2\circ f)\wedge \tilde{\omega}_2) \otimes e_1\\
    &\, -p_{\F}(e_{\varphi}(\partial_yg_1\circ f)\wedge \tilde{\omega}_1)\otimes e_2-p_{\F}(e_{\varphi}(\partial_xg_1\circ f)\wedge \tilde{\omega}_2) \otimes e_2
\end{align*}
and hence using that
\begin{align*}
    p_{\mathcal{M}}(-\partial_yg_i,-\partial_xg_i)=&\, \frac{1}{2|z|}(-(|z|+x)\partial_y g_i+y\partial_xg_i,y\partial_yg_i-(|z|-x)\partial_xg_i)
\end{align*}
for $i=1,2$, the coefficient for $b_1$ is given by
\begin{equation*}
    \frac{1}{2|f|}(-(|z|+x)\partial_yg_2+y\partial_xg_2-y\partial_yg_1+(|z|-x)\partial_xg_1)\circ f= \frac{1}{2|f|}g_1\circ f
\end{equation*}
where we used \eqref{eq: set2}. For the second term in \eqref{eq: cartan in cohomology} we obtain
\begin{align*}
    \tilde{p}\circ \tilde{a} (\Im(&\,2i[\pi_1,V_{\C}]\wedge[X ,E_2]))= \\
     &\, p_{\F}(e_{\varphi}(\frac{1}{|f|}((y\partial_x-x\partial_y)g_2-g_1)\circ f)\wedge \tilde{\omega}_1) \otimes e_1\\
    &\, +p_{\F}(e_{\varphi}(\frac{1}{|f|}((y\partial_x-x\partial_y)g_1+g_2))\circ f)\wedge \tilde{\omega}_1) \otimes e_2
\end{align*}
Using $p_{\mathcal{M}}$ and the identifications in \eqref{eq:composition of isomorphisms}  we get for the coefficient of $b_1$ with respect to the decomposition \eqref{eq: trivialization in (2,1)} the term
\begin{align*}
    \frac{1}{2|f|^2}(y((|z|+x)\partial_xg_2+y\partial_xg_1)+x((|z|+x)\partial_y g_2+&\,y\partial_y g_1))\circ f \\
    +\frac{1}{2|f|^2}(yg_2-(|z|+x)g_1)\circ f&\, =
    -\frac{1}{2|f|}g_1\circ f
\end{align*}
where we used \eqref{eq: set1} and \eqref{eq: set3}. This implies that:
\begin{align*}
    \tilde{p}\circ \tilde{a}([X,C_{\mathcal{R}}])\in C^\infty_0(\C)\circ f\cdot b_2,
\end{align*}
which yields the result for $C_{\mathcal{I}}$.

\appendix
 \section{Some decompositions of smooth functions on $\C$}\label{eigenspaces}

In this section, we will prove Lemma \ref{lemma:pullbacks and invariant functions}. We will use the notation from Subsection \ref{subsection:desing} for the eigenspaces of the pullback along the maps $\sigma=-\mathrm{id}_{\C}$ and $\tau=$ conjugation. As the operators $\sigma^*$ and $\tau^*$ commute they preserve their respective eigenspaces. We use similar notation for their common eigenspaces, for example:
\[C^{\infty}_{\sigma,-\tau}(\C):=\{g\in C^{\infty}(\C)\,|\, g=\sigma^*(g)=-\tau^*(g)\},\]
and similarly for functions flat at the origin.

We have the following decomposition of smooth functions: 
\begin{lemma}\label{lemma:eigenspaces}
Any function $g\in C^{\infty}(\C)$ can be written uniquely as: 
\[ g= g_0 + x\cdot g_x +y\cdot g_y +xy\cdot g_{xy}\]
where $g_0,g_x,g_y,g_{xy}\in C^{\infty}_{\sigma,\tau}(\C)$ and similarly for functions flat at the origin.
\end{lemma}
\begin{proof}
The four functions are constructed by the usual projections onto the common eigenspaces of $\sigma^*$ and $\tau^*$, namely: 
\begin{align*}
g_{0}:=&\frac{1}{4}(\mathrm{id}+\sigma^*)(\mathrm{id}+\tau^*)(g)\\
x\cdot g_{x}:=&\frac{1}{4}(\mathrm{id}-\sigma^*)(\mathrm{id}+\tau^*)(g)\\
y\cdot g_{y}:=&\frac{1}{4}(\mathrm{id}-\sigma^*)(\mathrm{id}-\tau^*)(g)\\
xy\cdot g_{xy}:=&\frac{1}{4}(\mathrm{id}+\sigma^*)(\mathrm{id}-\tau^*)(g).
\end{align*}
That these functions are indeed smooth follows from Hadamard's lemma. For example, in the first case, the function 
\[h:=\frac{1}{4}(\mathrm{id}-\sigma^*)(\mathrm{id}+\tau^*)(g)\]
satisfies $h(0,y)=0$, and therefore, by Hadamard's lemma, $h=xg_x$, with $g_x\in C^{\infty}(\C)$. 
\end{proof}

Flat functions invariant under $\sigma$ and $\tau$ can be described more explicitly: 
\begin{lemma}\label{lemma: lambda casimir}
The map defined by
\begin{equation*}
    \begin{array}{ccc}
        C^{\infty}(\R^2_{\ge 0} ) &\to &C_{\sigma,\tau}^{\infty}(\C), \\
         g &\mapsto & \widetilde{g}(x,y):= g(x^2 ,y^2)  
    \end{array}
\end{equation*}
is an isomorphism of rings, and similarly for functions flat at the origin.
\end{lemma}
\begin{proof}
Clearly, we only need to check that this map is surjective. For this, consider $\widetilde{g}\in C_{\sigma,\tau}^{\infty}(\C)$. We need to show that the function:
\[g:\R^2_{\geq 0}\to \R, \quad g(p,q)=\widetilde{g}(\sqrt{p},\sqrt{q}),\]
which is smooth on $\R^2_{>0}$, is smooth also on $\big(\{0\}\times \R_{\geq 0}\big)\cup \big(\R_{\geq 0}\times\{0\}\big)$. For this it suffices to check that all its partial derivatives extend continuously. Note that, for $p,q>0$, we have that 
\[    \partial_p^i\partial_q^j g(p,q)=
\Big(\frac{\partial_x}{2x}\Big)^{i}\Big(\frac{\partial_y}{2y}\Big)^{j}\widetilde{g}(x,y)\big|_{x=\sqrt{p},y=\sqrt{q}}.\]
Therefore, the proof will be complete once we show that the functions  
\[\Big(\frac{\partial_x}{2x}\Big)^{i}\Big(\frac{\partial_y}{2y}\Big)^{j}\widetilde{g}(x,y)\]
extend smoothly to $\R^2$. First, because $\widetilde{g}$ is invariant under $\sigma^*$ and $\tau^*$, we have that $\partial_xg(0,y)=0$ and $\partial_yg(x,0)=0$. Therefore, by Hadamard's lemma,
\[\frac{\partial_x\widetilde{g}}{2x}, \frac{\partial_y\widetilde{g}}{2y}\in  C^{\infty}(\C).\]
Moreover, note that these functions are also invariant under $\sigma^*$ and $\tau^*$, and so we can iterate this observation, and obtain that, for all $i,j\geq 0$,
\begin{align*}
\Big(\frac{\partial_x}{2x}\Big)^{i}\Big(\frac{\partial_y}{2y}\Big)^{j}\widetilde{g}\in C^{\infty}_{\sigma,\tau}(\C).  
\end{align*}
This concludes the proof. 
\end{proof}

\begin{proof}[Proof of Lemma \ref{lemma:pullbacks and invariant functions}]
First, we show that the map from (1) is an isomorphism. 
Since $\mathrm{sq}:\C\to \C$ is surjective, the map is clearly injective. For surjectivity, consider first $h\in C^{\infty}_{0,\sigma,\tau}(\C)$. By Lemma \ref{lemma: lambda casimir}, we can write 
\[h(\lambda_1,\lambda_2)=k(2\lambda_1^2,2\lambda_2^2),\]
for some $k\in C^{\infty}_0(\R^2_{\geq 0})$. Define:
\[g:\C\to \R,\quad  g(z)=k(|z|+x,|z|-x), \quad z=x+yi\in \C.\]
Clearly, $g$ is smooth on $\C^{\times}$. Because $k$ is flat at $0$, it is easy to see that all partial derivatives of $g$ extend continuously as $0$ at $z=0$. Therefore $g\in C^{\infty}_0(\C)$. Clearly, $g$ is also $\tau^*$-invariant. Finally, note that
\[g(\lambda^2)=k(2\lambda_1^2,2\lambda_2^2)=h(\lambda).
\]
For a general element $h\in C^{\infty}_{0,\sigma}(\C)$, by Lemma \ref{lemma:eigenspaces}, we can find $h_0,h_1\in C^{\infty}_{0,\sigma,\tau}(\C)$ such that
\[h(\lambda)=h_0(\lambda)+2\lambda_1\lambda_2\cdot h_1(\lambda).\]
By the first part, we can write $h_i(\lambda)=g_i(\lambda^2)$, with $g_i\in C^{\infty}_{0,\sigma,\tau}(\C)$. Thus
\[h(\lambda)=g(\lambda^2),\quad \textrm{with}\quad g\in C^{\infty}_0(\C),\]
where $g(z)=g_1(z)+y\cdot g_2(z)$, for $z=x+iy$. This proves (1).

Next, we show that the map from (2) is surjective. Let $h\in C^{\infty}_{0,-\sigma}(\C)$. By Lemma \ref{lemma:eigenspaces}, we can write:
\[h(\lambda)=-\lambda_1\cdot h_1(\lambda)+\lambda_2\cdot h_2(\lambda),\]
with $h_i\in C^{\infty}_{0,\sigma,\tau}(\C)$. Again by the first part, there are $g_i\in C^{\infty}_{0,\tau}(\C)$ such that $h_i(\lambda)=g_i(\lambda^2)$, and so
\[h(\lambda)=-\lambda_1\cdot g_1(\lambda^2)+\lambda_2\cdot g_2(\lambda^2).\]
This proves surjectivity. 

To show understand the kernel, let $g_i\in C^{\infty}_{0}(\C)$ such that 
\[-\lambda_1\cdot g_1(\lambda^2)+\lambda_2\cdot g_2(\lambda^2)=0.\]
Note that this is equivalent to: 
\[2\lambda_1\lambda_2\cdot g_1(\lambda^2)=2\lambda_2^2\cdot g_2(\lambda^2).\]
If we denote $z=\lambda^2$, then $2\lambda_1\lambda_2=y$ and $2\lambda_2^2=|z|-x$. Thus the above identity is equivalent to
\begin{equation}\label{eq:relation:for:kernel}
y\cdot g_1(z)=(|z|-x)\cdot g_2(z), \quad \forall z\in \C.
\end{equation}
This concludes the proof of (2).

Item (3) follows from Proposition \ref{proposition: twisted module}. 
\end{proof}

\section{Fr\'echet spaces of smooth functions}

\subsection{Partial derivatives, Leibniz rule, chain rule}\label{subsection: derivatives}

Here we recall some well-known facts from calculus, which are frequently used in the proofs, and we fix the notation. 

We denote the partial derivative corresponding to a multi-index \[a=(a_1,\ldots,a_m)\in \N_0^m \ \ \textrm{by}\ \  D^{a}: = \frac{1}{a_1!}\partial^{a_1}_{x_1}\ldots \frac{1}{a_m!}\partial^{a_m}_{x_m}.\] 
We will often use the general Leibniz rule:
\[D^a(g_1\cdot \ldots\cdot g_k)=\sum_{a^1+\ldots+a^k=a}D^{a^1}(g_1)\cdot\ldots \cdot D^{a^k}(g_k),\]
for $g_1,\ldots,g_k\in C^{\infty}(\R^m)$, and the general chain rule:
\[D^a(h(g_1,\ldots,g_k))=\sum_{1\leq |b|\leq |a|}D^b(h)(g_1,\ldots,g_k){\sum}' \prod_{i=1}^k\prod_{j=1}^{b_i} D^{a^{ij}}(g_i),\]
where $b=(b_1,\ldots,b_k)$ and $\sum'$ is the sum over all non-trivial decompositions  \[a=\sum_{i=1}^k\sum_{j=1}^{b_i}a^{ij}, \ \ \  a, a^{ij}\in \N^m_0\backslash\{0\}.\]

\subsection{Limits of families of smooth functions}\label{subsubsection_limits}

We give here a useful criterion for proving the existence of limits of smooth families of functions with respect to the compact open $C^{\infty}$-topology.

Let $U\subset \R^m$ be an open set. For a compact subset $K\subset U$, we define the corresponding $C^n$-seminorm on $C^{\infty}(U,\R^l)$ as:
\[\norm{f}_{n,K}:=\sup_{x\in K}\sup_{a\in \N^m_0\, :\, |a|\leq n}\left|D^af(x)\right|,\]
where $|a|=a_1+\ldots+a_m$. These seminorms yield the Fr\'echet space:
\[\big(C^{\infty}(U,\R^l), \{\norm{\cdot}_{n,K}\, |\, n\in \N_0,\,   K\subset U\}\big)\]
whose topology is called the compact-open $C^{\infty}$-topology. 

Consider a family $f_t\in C^{\infty}(U,\R^l)$, defined for $t\geq t_0$. Assume that, for each compact $K\subset U$ and each $n\geq 0$, we find a function 
\[l_{n,K}:[t_0,\infty)\to [0,\infty),\]
such that
\[\forall \  s\geq t\ :\ \ \norm{f_s-f_t}_{n,K}\leq l_{n,K}(t) \ \ \ \ \textrm{and} \ \ \ \ \lim_{t\to \infty}l_{n,K}(t)=0.\]
Then, since $C^{\infty}(U,\R^l)$ is a Fr\'echet space, the limit $t\to \infty$ exists: \[f_{\infty}:=\lim_{t\to \infty}f_{t}\in C^{\infty}(U,\R^l).\]
So all partial derivatives of $f_t$ converge uniformly on all compact subsets to those of $f_{\infty}$.

Assume further that $f_t$ is smooth also in $t$. Writing  
\[D^a(f_t-f_s)=\int_s^tD^a(f'_h)\dif h,\]
we obtain that 
\[\norm{f_s-f_t}_{n,K}\leq \int_s^t\norm{f'_h}_{n,K}\dif h.\]
So, if we find a function 
\[b_{n,K}:[t_0,\infty)\to [0,\infty),\]
such that
\begin{equation}\label{conditions for limit}
\forall \  t\geq t_0\ :\ \ \norm{f'_t}_{n,K}\leq b_{n,K}(t)\ \ \ \ \textrm{and} \ \ \ \ \int_{t_0}^{\infty}b_{n,K}(t)\dif t<\infty,
\end{equation}
then we can also conclude that $\lim_{t\to\infty} f_t$ exists.

\bibliographystyle{alpha}
\bibliography{Bibtex}

\Addresses
\end{document}